\documentclass[a4paper]{article}
\usepackage{amsmath, amstext, amsxtra, amsthm, amscd, amsgen,amsbsy, amsopn, amsfonts, latexsym, amssymb, amscd, amsthm}
\usepackage{color}

\usepackage{tabularx}

\usepackage{dynkin-diagrams}

\usepackage{bbold}

\usepackage[all]{xy}

\usepackage{tikz-cd}

\usepackage[a4paper, total={6.5in, 8.5in}]{geometry}

\usepackage{tensor}

\makeatletter
\renewcommand\subsection{\@startsection{subsection}{2}{\z@}%
                                     {-3.25ex\@plus -1ex \@minus -.2ex}%
                                     {-1.5ex \@plus .2ex}%
                                     {\normalfont\large\bfseries}}

\def\wrt{with respect to}

\def\proofend{\hbox to 1em{\hss}\hfill $\blacksquare $\bigskip }

\newtheorem*{theorem*}{Main Theorem}
\newtheorem*{theoremA}{Theorem A}
\newtheorem*{theoremB}{Theorem B}
\newtheorem*{theoremC}{Theorem C}
\newtheorem*{theoremD}{Theorem D}

\newtheorem{theorem}{Theorem}[section]
\newtheorem{proposition}[theorem]{Proposition}
\newtheorem{lemma}[theorem]{Lemma}
\newtheorem{remark}[theorem]{Remark}

\newtheorem{definition}[theorem]{Definition}
\newtheorem{corollary}[theorem]{Corollary}
\newtheorem{example}[theorem]{Example}
\newtheorem{examples}[theorem]{Examples}

\newtheorem{condition}[theorem]{Condition}

\def\Z{{\mathbb Z}}
\def\R{{\mathbb R}}
\def\Q{{\mathbb Q}}
\def\C{{\mathbb C}}
\def\N{{\mathbb N}}

\def\Zp #1{{\mathbb Z}_{#1}}

\def\ind{\mathrm{ind}}
\def\tr{\mathrm{tr}}
\def\sign{\mathrm{sign}}

\def\id{\mathrm {id}}

\def\Diractwist#1#2{D_{#1,#2}}

\def\DiractwistAPS#1#2{D^+_{#1,#2}}

\def\so2{S^1}

\def\vvv{\textsf{v}}

\begin{document}


\title{Odd-dimensional manifolds with infinitely many different geometries of positive Ricci curvature}

\author{Anand Dessai}

\date{\today}

\maketitle
\begin{abstract} In every odd dimension $n\geq 5$ we exhibit large classes of closed $n$-dimensional manifolds which admit infinitely many different geometries of positive Ricci curvature, i.e., manifolds for which their moduli space of metrics of positive Ricci curvature has infinitely many connected components.
\end{abstract}

\section{Introduction}\label{Section: Introduction}

For a closed $n$-dimensional manifold $M$, let $\mathcal{R}_{Ric>0}(M)$ denote the space of Riemannian metrics of positive Ricci curvature on $M$, equipped with the smooth topology. The diffeomorphism group $\text{Diff}(M)$ acts on $\mathcal{R}_{Ric>0}(M)$ by pulling back metrics. Taking the quotient, one obtains the moduli space $\mathcal{M}_{Ric>0}(M)$ of metrics of positive Ricci curvature on $M$.

If $M$ admits a metric of positive Ricci curvature ($Ric> 0$ metric for short), one would like to understand whether $\mathcal{M}_{Ric>0}(M)$ carries non-trivial topology, e.g., whether the moduli space is disconnected or has non-trivial higher homotopy or homology groups. In every odd dimension $n\geq 5$ manifolds are known for which their moduli space of $Ric> 0$ metrics has infinitely many connected components. However, no other topological information seems to be available in dimension $>3$. This is in sharp contrast to the case of positive scalar curvature or nonnegative Ricci curvature, where manifolds have been found for which the corresponding moduli spaces have non-trivial higher homotopy or homology groups \cite{BHSW10,TW22}.

In the following, we will say that two $Ric> 0$ metrics on $M$ represent {\em different geometries of positive Ricci curvature} if the corresponding elements in $\mathcal{M}_{Ric>0}(M)$ belong to different connected components.

Starting with work of Kreck and Stolz \cite{KS93}, odd-dimensional manifolds have been found which admit infinitely many different geometries of positive Ricci curvature. These examples typically come in infinite families and belong to special classes of manifolds. They include certain Witten spaces and twisted variants \cite{KS93,DKT18,D22}, various rational homotopy spheres, highly connected manifolds and manifolds with core metric of dimension $4k+3$ \cite{B20,W11,TW15,D17,G20,W22,D24,RW25}, certain $4k+1$-dimensional homotopy real projective spaces \cite{DG21}, as well as particular families of five- and seven-dimensional manifolds \cite{G20,W21,GW22,G24,X25}.

In this paper we exhibit in every odd dimension $n\geq 5$ new classes of $n$-dimensional manifolds which admit infinitely many different geometries of positive Ricci curvature. These manifolds occur (up to a cover) as the boundary of manifolds which admit a {\em nice metric coordinate function}. The latter signifies, vaguely speaking, that the cobounding manifold shares geometric properties reminiscent to the process of shrinking of the fibers. We refer to Definitions \ref{nice coordinate function} and \ref{nice equivariant coordinate function} for precise information and just remark that many disk bundles and, more generally, plumbings of disk bundles fall into this setting. The construction of metrics of positive Ricci curvature is based on work of Wraith \cite{W98,W11} and Reiser \cite{R23}.

In dimension $n=4k+3\geq 7$ a large class of manifolds which admit infinitely many different geometries of positive Ricci curvature is given in Theorem \ref{AB*}. As an illustration, we point out the following two special cases.

\bigskip

\begin{theoremA}\label{theorem A} Let $B$ be a $p$-dimensional closed Riemannian manifold of positive Ricci curvature, $\xi $ a vector bundle of rank $q$ over $B$, and $M$ the total space of the associated sphere bundle. Suppose $M$ is spin, all Pontryagin classes of $M$ are torsion classes, and $\{p,q\}= \{4s, 4t\}$, $1\leq s\leq t<2s$.  Then $M$ admits infinitely many different geometries of positive Ricci curvature.
\end{theoremA}
The theorem seems to be new even in the case when the bundle is trivial. Special cases and variants of Theorem A are discussed in Section \ref{section 4k+3} (see Corollary \ref{Cor 1 Thm A},  Corollary \ref{Cor 2 Thm A}, and Remark \ref{remark rp4}).

Next, consider the plumbing $W=\square _{\Gamma}E_\vvv$ of disk bundles according to a tree $\Gamma$. Here, $E_\vvv$ denotes the total space of a linear disk bundle over the sphere $S^p$ (resp. $S^q$) of rank $q$ (resp. $p$) associated to a vertex $\vvv$ of $\Gamma $, and two disk bundles are plumbed together if the respective vertices are connected by an edge of $\Gamma $. Let $\xi $ be a vector bundle of rank $q$ over a $p$-dimensional closed Riemannian manifold $B$ of positive Ricci curvature for which the total space of the associated sphere bundle is spin. Let $E$ denote the total space of the disk bundle associated to $\xi$. Consider the plumbing of $E$ and $W$ \wrt \ some disk bundle of $W$.

\begin{theoremB}\label{theorem B} Let $M$ be the boundary of the plumbing  $E\square W$. If $p(M)=1\in H^*(M;\Q )$ and $\{p,q\}= \{4s, 4t\}$, $1\leq s\leq t<2s$, then $M$ admits infinitely many different geometries of positive Ricci curvature.
\end{theoremB}

The conclusion also holds true in the more general situation where $W$ is a plumbing of disk bundles $E_\vvv \to B_\vvv$  according to a tree, $W$ is spin, and each base manifold $B_\vvv$ admits a core metric. If the base manifold $B$ of $\xi $ admits a core metric then the theorem also follows from \cite{B20,R23} (see Remark \ref{remark core infty}).

In the proofs of Theorem \ref{AB*} and the two theorems above, geometries of positive Ricci curvature are distinguished using the basic index difference for spin manifolds. Alternatively, one could utilize Kreck-Stolz invariants. Both approaches require that $M$ is $4k+3$-dimensional and that its rational Pontryagin classes vanish.

In dimension $n=4k+1\geq 5$ eta invariants can be used to distinguish geometries of positive Ricci curvature on non-simply connected manifolds in favorable situations (see, for example, \cite{DG21,W21,D22,GW22,G24}).

Theorem \ref{CD*} describes a large class of manifolds for which eta invariants detect infinitely many different geometries of positive Ricci curvature. Here, we like to point out the following special cases.

\begin{theoremC}\label{theorem c} Let $B$ be a closed connected manifold of dimension $2k+1\geq 3$ with a $\Zp 2$-action with isolated fixed points and an invariant $Ric >0$ metric. Then the total space $S(TB)$ of the sphere tangent bundle has a smooth free $\Zp 2$-action such that the quotient manifold admits infinitely many different geometries of positive Ricci curvature.
\end{theoremC}

\begin{theoremD}\label{theorem d} Let $\xi $ be a $\Zp 2$-equivariant vector bundle of rank $p$ over a $p$-dimensional closed manifold $B$ and let $p=2k+1\geq 3$. Suppose $\Zp 2$ acts with isolated fixed points. Suppose $B$ is $2$-connected and admits an invariant metric of positive Ricci curvature. 
Then the total space $S(\xi )$ of the associated sphere bundle has a smooth free $\Zp 2$-action such that the quotient manifold admits infinitely many different geometries of positive Ricci curvature.
\end{theoremD}

The conclusion also holds true if one replaces the bundle $\xi $ by certain $\Zp 2$-equivariant plumbings (see Theorem \ref{special CD* 2}). Plumbings of this kind have been used in \cite{DG21} to show that for any $k\geq 1$ there are homotopy $\R P^{4k+1}$s which carry infinitely many different geometries of positive Ricci curvature. The construction of these geometries is a crucial ingredient in the proof of Theorem \ref{CD*}. In \cite{D25} a similar approach is used to exhibit examples of even-dimensional manifolds which admit infinitely many different geometries of positive Ricci curvature.
We remark that the proofs of Theorems C and D only show the conclusion for {\em some} free $\Zp 2$-action but not necessarily for the action induced by the {\em given} action.

\medskip

The proofs of Theorem \ref{AB*} and Theorem \ref{CD*} consist of three parts, which are of geometric, topological, and index-theoretical nature. The reasoning differs depending on the dimension, either being of the form $4k+1$ or $4k+3$. We give a rough outline.

The main geometrical ingredient in the proofs is the construction of $scal>0$ metrics on even-dimensional plumbings which restrict to $Ric>0$ metrics on the boundary. Metrics of this kind were (up to deformation within positive scalar curvature) first constructed by Wraith in his work \cite{W11} on the moduli space of positive Ricci curvature metrics on homotopy spheres (see also \cite{W95,W97,W98}). We will adapt a slightly different approach, due to Reiser \cite{R23}, to construct metrics with the above properties on the generalized plumbings considered in this paper. These plumbings are of the form $W\square V$ where $V$ is a manifold with a nice metric coordinate function (see Definition \ref{nice coordinate function}) and $W$ is a plumbing of disk bundles (see Theorem \ref{plumbing theorem} for details). If the dimension of the boundary $\partial (W\square V)$ is of the form $4k+1$, we consider $\Zp 2$-equivariant plumbings with free action on the boundary and construct such metrics on $W\square V$ which are in addition invariant under the action (see Theorem \ref{equivariant plumbing theorem} for details).

On the topological side, we want to find generalized plumbings as above with (equivariantly) diffeomorphic boundaries.
To do so, we construct plumbings $W\square V$ for which $\partial (W\square V)$ is diffeomorphic to the connected sum $\partial W\sharp \partial V$ and $\partial W$ is a sphere (see Lemma \ref{lemma 2 connected sum} and the proofs of Theorem \ref{AB*} and Theorem \ref{CD*} for details).

The geometric and topological constructions yield $Ric >0$ metrics on a closed manifold $M$ which is diffeomorphic to the boundary of infinitely many generalized plumbings. If $M$ is $4k+1$-dimensional, then $M$ comes with a free $\Zp 2$-action and invariant $Ric >0$ metrics. The metrics on $M$ (in dimension $4k+3$) and on the quotient $\overline M :=M/\Zp 2$ (in dimension $4k+1$) give rise to elements in $\mathcal{M}_{Ric>0}(M)$ and $\mathcal{M}_{Ric>0}(\overline M)$, respectively.

In the last part of the proofs, we use the basic index difference of Gromov-Lawson (see \cite{LM89} and references therein) to show that the elements in the moduli space $\mathcal{M}_{Ric>0}(M)$ of the $4k+3$-dimensional manifold $M$ represent infinitely many connected components, i.e., the corresponding $Ric >0$ metrics on $M$ represent infinitely many different geometries of positive Ricci curvature. In the $4k+1$-dimensional case, we use eta invariants to show that the $Ric >0$ metrics on $\overline M$ represent infinitely many different geometries of positive Ricci curvature. For the study of moduli spaces of metrics of positive scalar curvature this approach goes back to work of Atiyah-Patodi-Singer \cite{APSII75} and Botvinnik-Gilkey \cite{BG95}. 

Finally, we like to remark that by combining the above reasoning with \cite{RW25} (resp. \cite{R23,R24}) one can find examples of manifolds for which the moduli space of metrics of positive intermediate Ricci curvature  (resp. moduli space of core metrics) has infinitely many connected components.

The paper is structured as follows. In the next section we briefly discuss basic properties of spaces of metrics, their moduli spaces, and geometries of positive Ricci curvature.
In Section \ref{section topological plumbing} we recall the plumbing construction for disk bundles, generalizations thereof, and a connected sum construction via plumbing. Section \ref{section geometric plumbing} is dedicated to the construction of $Ric >0$ metrics on the boundary of generalized plumbings. These plumbings are of the form $W\square V$, where $V$ has a nice metric coordinate function (see Definition \ref{nice coordinate function}) and $W$ is a plumbing of disk bundles. The metrics in question extend to $scal >0$ metrics on $W\square V$, which are of product type near the boundary. The section also contains some relevant geometric background. In Section \ref{Equivariant plumbing} we consider $\Zp 2$-equivariant plumbings and extend the aforementioned topological and geometric properties of plumbings to the equivariant setting. Section \ref{index section} reviews those aspects of (equivariant) index theory which will be relevant for the detection of different connected components of the moduli space.

Theorem \ref{AB*}, which is our main theorem for $4k+3$-dimensional manifolds, as well as Theorems A and B and other special cases are proved in Section \ref{section 4k+3}. The proofs of Theorem \ref{CD*}, Theorems C and D and related results for $4k+1$-dimensional manifolds are given in Section \ref{section 4k+1}.

\bigskip
\noindent
{\em Acknowledgements.} I like to thank Philipp Reiser for helpful comments and for a suggestion how to treat a step in the proof of Theorem \ref{plumbing theorem}. I am also greatful to Michael Wiemeler for useful remarks. This work was supported in part by the SNSF-Project 200020E\_193062 and the DFG-Priority programme  {\em Geometry at infinity} (SPP 2026).


\pagebreak
\section{Moduli spaces and geometries of positive Ricci curvature}\label{moduli}

Let $M$ be a closed $n$-dimensional manifold. Consider the space $\mathcal{R}(M)$ of Riemannian metrics on $M$, endowed with the smooth topology of uniform convergence of all derivatives. The group $\text{Diff}(M)$ of diffeomorphisms of $M$ acts on $\mathcal{R}(M)$ by pulling back metrics. The {\em moduli space of metrics} $\mathcal{M}(M)$ is defined as the orbit space $\mathcal{R}(M)/\text{Diff}(M)$, equipped with the quotient topology.

Note that two metrics $g_0, g_1\in 
\mathcal{R}(M)$ yield the same element in $\mathcal{M}(M)$ if and only if $(M,g_0)$ and $(M,g_1)$ are isometric. The isotropy group of the action of $\text{Diff}(M)$ on $\mathcal{R}(M)$ at $g$ is equal to the isometry group of $(M,g)$.
Note also that if $M$ and $N$ are of the same diffeomorphism type, then a metric on $M$ defines a well-defined element in the moduli space of $N$.

Next, consider a condition $\cal C$ for metrics which is preserved under the pullback of metrics under diffeomorphisms (for example, $\cal C$ could be a lower bound on scalar or Ricci curvature). 
Let $\mathcal{R}_{\cal C}(M)\subset \mathcal{R}(M)$ denote the subspace of metrics satisfying condition $\cal C$ and let $\mathcal{M}_{\cal C}(M):=\mathcal{R}_{\cal C}(M)/\text{Diff}(M)$ be the respective moduli space. 

Whereas the space of Riemannian metrics $\mathcal{R}(M)$ is a convex cone and, hence, contractible, the space $\mathcal{R}_{\cal C}(M)$ may have non-trivial topology. For example, if $M$ is spin and if condition $\cal C$ implies positive scalar curvature,  then the corresponding space $\mathcal{R}_{{\cal C}}(M)$, if not empty, may carry some non-trivial topology, which comes from the topology of an orbit of the action of $\text{Diff}(M)$ (this goes back to Hitchin, see \cite[IV \S 7]{LM89} and references therein). Note, however, that this kind of topological information disappears once one passes to the moduli space.

At present, the only known non-trivial topological information about the moduli space $\mathcal{M}_{Ric>0}(M)$ of metrics of positive Ricci curvature is related to the number of its connected components. This is in sharp contrast to the case of positive scalar curvature or nonnegative Ricci curvature where manifolds have been constructed for which the moduli space of $scal>0$ or $Ric \geq 0$ metrics has infinitely many non-trivial homotopy or homology groups, see \cite{BHSW10} or \cite{TW22}, respectively (see also \cite{TW15} and references therein for a survey of some of the earlier results in this direction).  We will use the following terminology.

\begin{definition}\label{different geometries} Two metrics on $M$ of positive Ricci curvature are said to represent {\bf \em different geometries of positive Ricci curvature} if the corresponding elements in the moduli space belong to different connected components of $\mathcal{M}_{Ric>0}(M)$.
\end{definition}
Note that $\mathcal{M}_{Ric>0}(M)$, being the quotient of a locally path-connected space, is locally path-connected as well. Hence, connected components are path components.

Clearly, two $Ric>0$ metrics which belong to the same connected component of $\mathcal{R}_{Ric>0}(M)$ represent the same geometry of positive Ricci curvature. Conversely, if $g_0$ and $g_1$ are metrics on $M$ which represent the same geometry of positive Ricci curvature, then there is a path inside $\mathcal{R}_{Ric>0}(M)$ connecting $g_0$ with $F^*(g_1)$ for some diffeomorphism $F$. This follows from Ebin's slice theorem \cite{E70} which implies that a path in $\mathcal{M}_{Ric>0}(M)$ can always be lifted to $\mathcal{R}_{Ric>0}(M)$ (see \cite[Prop. VII.6]{B75}, \cite[Lemma 4.4]{Wi21}). 

For later reference, we note that if $\cal G$ is a subgroup of $\text{Diff}(M)$ of finite index, then $M$ carries infinitely many different geometries of positive Ricci curvature if and only if the quotient space $\mathcal{R}_{Ric>0}(M)/{\cal G}$ has infinitely many connected components, i.e., if and only if $\pi_0(\mathcal{R}_{Ric>0}(M)/{\cal G})$ is infinite. This applies, for example, if $M$ is spin and ${\cal G}$ is the subgroup of spin preserving diffeomorphisms or, if $M$ is $Spin^c$ with $H^2(M;\Z )$ finite and ${\cal G}$ is the subgroup of diffeomorphisms preserving the $Spin^c$-structure.

In Sections 7--8 we will describe large classes of manifolds which admit infinitely many different geometries of positive Ricci curvature.

\section{Topological Plumbing}\label{section topological plumbing}

In this section we first discuss plumbing of disk bundles following Milnor \cite{M59,M60} (see also \cite{HM68,B72}) and then extend the plumbing construction to a more general setting. At the end of the section, we recall a basic connected sum construction via plumbing.

\subsection{Plumbing of disk bundles}\label{subsection plumbing disk bundles} 
Let $\xi$ be an oriented real Riemannian vector bundle of rank $p\geq 1$ over an oriented compact connected base manifold $B^q$ of dimension $q\geq 1$. Let $E\to B$ and $\partial E\to B$ denote the associated linear disk and sphere bundle, respectively.
If not stated otherwise, we will assume that $B$ has empty boundary.  Note that $\partial E$ is the boundary of $E$ in this situation.

Let $\overline \varphi  :D^q\hookrightarrow B$ be an embedding of the closed unit disk $D^q\subset \R ^q$ and let $\varphi : D^q\times D^p\hookrightarrow E$ be a coordinate function of the bundle covering $\overline \varphi $, i.e., $E_{\vert \overline \varphi (D^q)}\to  \overline \varphi (D^q)\times D^p$, $\varphi (x,v)\mapsto (\overline \varphi (x),v)$, is a local trivialization of the linear disk bundle $E\to B$. We will call $\varphi $ an \emph{orientation preserving coordinate function} if $\varphi $ is orientation preserving when restricted to fibers and if $\overline \varphi $ is orientation preserving.

It follows from the disk theorem that any two orientation preserving coordinate functions $\varphi _i$, $i=0,1$, are isotopic and the isotopy extends to an ambient isotopy, i.e., there exists a smooth family of coordinate functions $\varphi _t$, $t\in[0,1]$, connecting $\varphi _0$ and $\varphi _1$ and there is a smooth family of diffeomorphisms $f_t:E\to E$ with $f_0=\mathrm{id}$ and $\varphi _t=f_t\circ \varphi _0$.

Next, consider a rank $p$ vector bundle $\xi _1$ over $B_1^q$ as before and another oriented real Riemannian vector bundle $\xi _2$ of rank $q$ over an oriented compact connected base manifold $B_2^p$. Let $\varphi _1: D^q\times D^p\hookrightarrow E_1$ be an orientation preserving coordinate function as before and let $\varphi _2: D^p\times D^q\hookrightarrow E_2$ be an orientation preserving coordinate function of the disk bundle associated to $\xi _2$.

In this situation one can plumb the two disk bundles together using cross identification. The resulting space is obtained from the disjoint union $E_1\coprod E_2$ by identifying $\varphi _1(x,y)$ with $\varphi _2(y,x)$, $(x,y)\in D^q \times D^p$, and straightening out the angles. More generally, one could make the identification using any diffeomorphism $D^q\times D^p\to D^p\times D^q$. If not stated otherwise, we will always use the diffeomorphism $I:D^q\times D^p\to D^p\times D^q$, $(x,y)\mapsto(y,x)$.

The plumbing of the two disk bundles $E_i\to B_i$, $i=1,2$, is a $(p+q)$-dimensional manifold, which will be denoted by $E_1\square E_2$. Its boundary is given by
$$\partial (E_1\square E_2)=\partial E_1\setminus \varphi _1(\overset \circ {D^q}\times S^{p-1})\bigcup _{S^{q-1}\times S^{p-1}}\partial E_2\setminus \varphi _2( \overset \circ {D^p}\times S^{q-1}),$$
where $\varphi _1(x,y)$ and $\varphi _2(y,x)$ are identified for $(x,y)\in S^{q-1}\times S^{p-1}$.

Note that the orientations on $E_1\square E_2$ induced by the inclusions $E_i\hookrightarrow E_1\square E_2$, $i=1,2$, agree if and only if $pq$ is even.
Note also that the oriented diffeomorphism type of $E_1\square E_2$ does not depend on the particular choice of trivializations, since orientation preserving coordinate functions are ambient isotopic.

The plumbing construction can be iterated and thereby generalized to a plumbing according to a connected graph $\Gamma$ where each vertex $\vvv$ of $\Gamma $ corresponds to a disk bundle $E_\vvv\to B_\vvv$ and two disk bundles are plumbed together for each edge connecting the corresponding vertices (if $p\neq q$ one needs to assume that the dimensions match up). The plumbing is a $(p+q)$-dimensional manifold with boundary, denoted by $\square _{\Gamma}E_\vvv$. If $\Gamma$ is a rooted tree, then the plumbing inherits a well-defined orientation from the inclusion of the disk bundle at the root. As before, the oriented diffeomorphism type of $\square _{\Gamma}E_\vvv$ does not depend on the particular choice of trivializations.

By shrinking the fibers in the disk bundles, one sees that $\square _{\Gamma}E_\vvv$ is homotopy equivalent to the space obtained by attaching each $B_\vvv$ to the corresponding vertex $\vvv$ of $\Gamma $. If $\Gamma $ is a tree, then $\square _{\Gamma}E_\vvv\simeq \bigvee _\vvv B_\vvv$. For example, if the plumbing is carried out according to a graph $\Gamma$ which is a straight line with $2l$ vertices, denoted by $1, \ldots , 2l$, then $\square _{\Gamma}E_\vvv=E_1 \square \ldots \square E_{2l}\simeq B_1\vee  \ldots \vee B_{2l}$.

{\bf In the following, we will always assume that $\boldsymbol{\Gamma }$ is a tree and that $\boldsymbol {p,q\geq 3}$.} We note that in this case the plumbing $\square _{\Gamma}E_\vvv$ and its boundary are both simply connected if all base manifolds $B_\vvv$ are spheres.

For an oriented real Riemannian vector bundle of rank $p$ over the sphere $S^q$ it will be useful to describe the associated disk bundle $E\to S^q$ in terms of a characteristic map $\phi :S^{q-1}\to SO(p)$.

We write $S^q$ as the union of the upper and lower hemisphere, $S^q=D^q_+\cup D^q_-$, and identify  $D^q_{\pm}$ with the unit disk $D^q$. Let $\varphi ^{\pm }: D^q\times D^p\hookrightarrow E$ be orientation preserving coordinate functions of $E_{\vert D^q_{\pm}}$ and let $\Phi:=(\varphi ^-)^{-1}\circ \varphi ^+_{\vert S^{q-1}\times D^p}: S^{q-1}\times D^p\to S^{q-1}\times D^p$. Then  the characteristic map $\phi :S^{q-1}\to SO(p)$ is defined by
$$\Phi (y,x)=(y, \phi (y)(x)).$$
Conversely, $E\to S^q$ can be obtained, up to isomorphism, from $\phi $ by identifying $D^q\times D^p$ with $D^q\times D^p$ along $S^{q-1}\times D^p$ via $\Phi$. The latter will be denoted by $D^q\times D^p\cup_\Phi D^q\times D^p$. Similarly, the sphere bundle $\partial E\to S^q$ is isomorphic to $D^q\times S^{p-1}\cup_\Phi D^q\times S^{p-1}$. Recall that the characteristic map is uniquely determined up to homotopy by the disk bundle and that isomorphism types of rank $p$ oriented linear disk bundles over $S^q$ are in one-to-one correspondence with elements in $\pi _{q-1}(SO(p))$ (see \cite{S51} for details).

Next consider the plumbing $E_1\square E_2$ of two disk bundles over spheres. Let $\varphi _1^{\pm }: D^q\times D^p\hookrightarrow E_1$ be orientation preserving coordinate functions as before, $\Phi_1:=(\varphi _1^-)^{-1}\circ \varphi _1^+: S^{q-1}\times D^p\to S^{q-1}\times D^p$,  $(x,y)\mapsto (x, \phi _1(x)(y))$, and identify $E_1$ with $D^q\times D^p\cup_{\Phi _1} D^q\times D^p$. Also let $\varphi _2^{\pm}: D^p\times D^q\hookrightarrow E_2$ be orientation preserving coordinate functions of ${E_2}_{\vert D^p_{\pm }}$, $\Phi_2:=(\varphi _2^-)^{-1}\circ {\varphi _2^+}_{\vert S^{p-1}\times D^q}$, and identify $E_2$ with $D^p\times D^q\cup_{\Phi _2} D^p\times D^q$.

Then the boundary of the plumbing \wrt \ $\varphi _1^-$ and $\varphi _2^+$ is
$$\partial (E_1\square E_2)=\partial E_1\setminus \varphi _1^-(\overset \circ {D^q}\times S^{p-1})\bigcup _{S^{q-1}\times S^{p-1}}\partial E_2\setminus \varphi _2^+(\overset \circ {D^p}\times S^{q-1}),$$
where $\varphi _1^-(x,y)$ is identified with $\varphi _2^+(y,x)$ for $(x,y)\in  S^{q-1}\times S^{p-1}$. Using the identifications above, the boundary $\partial (E_1\square E_2)$ can be described by
\begin{equation}\label{eq gluing 1}D^q\times S^{p-1}\bigcup _{\Phi _2\circ I\circ\Phi_1}D^p\times S^{q-1},\end{equation} 
where $(x,y)\in  S^{q-1}\times S^{p-1}$ is identified with $\Phi _2\circ I\circ\Phi_1(x,y)\in  S^{p-1}\times S^{q-1}$.

Note that if one of the bundles is trivial, say $E_1\to S^q$, then the diffeomorphism $D^q\times S^{p-1}\to D^q\times S^{p-1}$, $(x,y)\mapsto (\phi _2(y)(x),y)$, induces an identification of $\partial (E_1\square E_2)$ with $D^q\times S^{p-1}\bigcup _{ I}D^p\times S^{q-1}=S^{p+q-1}$.

Milnor \cite{M59,M60} showed that if one of the bundles, say $E_1\to S^q$, has a nowhere vanishing section, then $\partial (E_1\square E_2)$ admits a Morse function with two critical points and, hence, is a homotopy sphere. Note that $E_1\to S^q$ has a nowhere vanishing section if and only if the homotopy class of  its characteristic map $\phi _1$ is in the image of $\pi _{q-1}(SO(p-1))\to \pi _{q-1}(SO(p))$.

Milnor also showed that if one restricts to bundles which admit nowhere vanishing sections, then the plumbing construction yields a bilinear map
\begin{equation}\label{milnor pairing 1}\beta : \pi _{q-1}(SO(p-1))\otimes \pi _{p-1}(SO(q-1))\to \Theta _{p+q-1}, \quad (\theta _1, \theta _2) \mapsto \partial (E_1\square E_2).\end{equation}

Here, $E_i$ is the disk bundle with nowhere vanishing section defined by $\theta _i$ and $\Theta_{k}$ denotes the group of diffeomorphism classes of oriented $k$-dimensional homotopy spheres (for more information on the Milnor pairing $\beta $ see \cite{M60} and Section \ref{section 4k+3}).

For later reference, we point out the following generalization of (\ref{eq gluing 1}) to $2l$ disk bundles $E_i\to B_i$ over spheres with characteristic maps $\phi _i$.

\begin{lemma}\label{lemma plumbing over spheres}
Let $E_1 \square \ldots \square E_{2l}$ be the plumbing of disk bundles over spheres according to a straight line, $B_1=S^q,\ldots , B_{2l}=S^p$. We identify $E_{2i+1}$ with $D^q\times D^p\cup_{\Phi _{2i+1}} D^q\times D^p$ and $E_{2i}$ with $D^p\times D^q\cup_{\Phi _{2i}} D^p\times D^q$. Then
the boundary $\partial (E_1 \square \ldots \square E_{2l})$ is diffeomorphic to the manifold obtained by gluing $D^q\times S^{p-1}$ and $D^p\times S^{q-1}$ along their boundary via the diffeomorphism
\begin{equation}\label{eq gluing 2}\Phi_{2l}\circ I\circ \Phi_{2l-1}\circ \ldots \circ I\circ \Phi _2\circ I\circ\Phi_1 :S^{q-1}\times S^{p-1}\to S^{p-1}\times S^{q-1}.\end{equation}\qed 
\end{lemma}

\subsection{Generalized plumbing}\label{subsection generalized plumbing}
For later use, we extend the plumbing construction to a more general setting. Let $V^{p+q}$ be a smooth manifold with boundary. A map $\psi: D^p\times D^q\to V$ will be called a {\em neat embedding} if $\psi$ extends to a neat embedding $\overset \circ {D_r^p}\times D^q\hookrightarrow V$ for some $r>1$.\footnote{Recall that an embedding of a manifold into $V$ with image $A^a\subset V$ is neat if $\partial A=A\cap \partial V$ and $A$ is covered by charts $(\phi,U)$ of $V$ such that $A\cap U=\phi ^{-1}(\R ^a)$. In particular, $A$ intersects $\partial V$ transversally.} Here, $\overset \circ {D_r^p}\subset \R ^p$ denotes the open disk of radius $r$.

Let $\partial \psi:= \psi _{\vert D^p\times S^{q-1}}$. Note that $\psi$ maps $D^p\times \overset \circ {D^q}$ into the interior of $V$ and $\partial \psi$ maps $D^p\times S^{q-1}$ into the boundary (i.e., $D^p\times S^{q-1}=\psi ^{-1}(\partial V)$).

Suppose that $\psi: D^p\times D^q\hookrightarrow V$ is as above, $Q^{p+q}$ is another smooth manifold with boundary, and $\zeta : D^q\times D^p\hookrightarrow Q$ is a neat embedding. Then one can define the {\em generalized plumbing} $Q\square V$ \wrt \ $(\psi, \zeta)$ by taking the disjoint union $Q\coprod V$, identifying $\zeta (y,x)$ with $\psi (x,y)$, $(x,y)\in D^p \times D^q$, and straightening out the angles.

The boundary $\partial (Q\square V)$ of the generalized plumbing is obtained by gluing $\partial Q \setminus\zeta(\overset \circ {D^q}\times S^{p-1})$ and\linebreak $\partial V \setminus \psi (\overset \circ {D^p}\times S^{q-1})$ along their boundary via $\psi \circ I\circ \zeta^{-1}$, 
$$\partial (Q\square V)=\partial Q \setminus\zeta(\overset \circ {D^q}\times S^{p-1}) \bigcup_{\psi \circ I\circ \zeta^{-1}}\partial V \setminus \psi (\overset \circ {D^p}\times S^{q-1}).$$

Note that the diffeomorphism type of $\partial (Q\square V)$ does not change if $\partial \zeta:=\zeta_{\vert D^q\times S^{p-1}}:D^q\times S^{p-1}\hookrightarrow \partial Q$ is moved within its isotopy class.

Next, we discuss a way to construct the connected sum via plumbing (see \cite[Prop. 2.6]{CW17}, \cite[Prop. 3.3]{R24} and references therein for generalizations).

Let $\psi: D^p\times D^{q}\hookrightarrow V$ be a neat embedding as before and let $j: D^q\times S^{p-1}\hookrightarrow S^{p+q-1}\sharp \partial Q $ be induced by the inclusion $D^q\times S^{p-1}\hookrightarrow (D^q\times S^{p-1}\cup _I D^p\times S^{q-1})=S^{p+q-1}$ into the first component. We will assume that the connected sum is taken \wrt \ an open ball $U$ in the interior of $D^p\times S^{q-1}$ and an open ball $U^\prime \subset \partial Q$. In this situation,
$$(S^{p+q-1}\sharp \partial Q )\setminus j(\overset \circ {D^q}\times S^{p-1}) \bigcup_{\partial \psi \circ I\circ  j^{-1}} \partial V \setminus \psi (\overset \circ {D^p}\times S^{q-1}) =  (D^p\times S^{q-1})\sharp \partial Q  \; \cup \;  \partial V \setminus \psi (\overset \circ {D^p}\times S^{q-1})$$
$$=\partial Q \setminus U^\prime \; \cup \; (D^p\times S^{q-1})\setminus U \; \cup \; \partial V \setminus \psi (\overset \circ {D^p}\times S^{q-1}) =\partial Q \setminus U^\prime\; \cup  \; \partial V \setminus \psi (U)=\partial Q\sharp \partial V.$$

\begin{lemma}\label{lemma 1 connected sum}
Let $\psi$ and $\zeta$ be neat embeddings as above. Assume that $\partial \zeta:D^q\times S^{p-1}\hookrightarrow \partial Q\overset \cong \to S^{p+q-1}\sharp \partial Q $ is isotopic to the embedding $j: D^q\times S^{p-1}\hookrightarrow S^{p+q-1}\sharp \partial Q $. Then $\partial (Q \square V)$ is diffeomorphic to $\partial Q\sharp \partial V$ by a diffeomorphism which is the identity on $\partial V\setminus \psi (D^p\times S^{q-1})$.
\end{lemma}

\noindent
\begin{proof} 
To construct the diffeomorphism one first moves $\partial\zeta $ to $j$ via an ambient isotopy and then identifies $(S^{p+q-1}\sharp \partial Q )\setminus j(\overset \circ {D^q}\times S^{p-1}) \bigcup_{\psi \circ I\circ  j^{-1}} \partial V \setminus \psi (\overset \circ {D^p}\times S^{q-1})  $ with $\partial Q\sharp \partial V$ as indicated above.
\end{proof}

\begin{lemma}\label{lemma 2 connected sum} Let $V_i$ be a manifold with neat embedding $\psi_i: D^p\times D^q\hookrightarrow V_i$, $i=1,2$. Let $\cal D$ be the plumbing $(S^q\times D^p)\square (S^p\times D^q)$ of trivial disk bundles and let $\varphi _i: D^q\times D^p\hookrightarrow S^q\times D^p\subset {\cal D}$, $i=1,2$, be disjoint coordinate functions of $S^q\times D^p\to S^q$. Consider the plumbing ${\cal P}$ obtained by plumbing $V_1$ and $V_2$ to $\cal D$ \wrt \ $(\psi _1,\varphi _1)$ and  $(\psi _2,\varphi _2)$. Then $\partial {\cal P}\cong \partial V_1\sharp \partial V_2$. Moreover, the diffeomorphism can be chosen to be the identity on $\partial V_i\setminus \psi _i(D^p\times S^{q-1})$, $i=1,2$.
\end{lemma}

\noindent
\begin{proof} We identify $\partial {\cal D}=S^{p+q-1}$ with the connected sum of two copies, $\partial {\cal D}\cong \partial{\cal D}_1\sharp \partial{\cal D}_2$ (the connected sum will be taken \wrt \ open balls in the second component of $\partial{\cal D}_i=D^q\times S^{p-1}\cup _I D^p\times S^{q-1}$, $i=1,2$). Note that the embeddings $\partial \varphi_i:D^q\times S^{p-1}\hookrightarrow \partial {\cal D}\cong \partial{\cal D}_1\sharp \partial{\cal D}_2$, $i=1,2$, are isotopic to the embeddings $j_i:D^q\times S^{p-1}\hookrightarrow  \partial{\cal D}_1\sharp \partial{\cal D}_2$ induced by the inclusion $D^q\times S^{p-1}\hookrightarrow D^q\times S^{p-1}\cup _I D^p\times S^{q-1}=\partial{\cal D}_i$.

Arguing as in Lemma  \ref{lemma 1 connected sum}, it follows that there is a diffeomorphism $\partial {\cal P}\overset \cong \to 
\partial V_1\sharp \partial V_2$ which is the identity on $\partial V_i\setminus \psi _i(D^p\times S^{q-1})$, $i=1,2$.
\end{proof}

\section{Geometric plumbing}\label{section geometric plumbing}

In this section we construct metrics of positive scalar curvature ($scal >0$) on certain generalized plumbings. These metrics are of product type near the boundary and restrict to metrics of positive Ricci curvature ($Ric >0$)  on the boundary. Metrics on plumbings with these properties (up to a deformation within $scal >0$) were first constructed by Wraith \cite{W11}. We will follow a slightly different approach due to Reiser \cite{R23} and adapt it to the generalized plumbings considered in this paper (see Theorem \ref{plumbing theorem}).

\subsection{Preliminaries}\label{preliminaries}

We first fix some notations. Let $ds^2_p$ denote the metric of the round unit $p$-dimensional sphere and let $S^p(r)$ denote the round sphere $(S^p, r^2ds^2_p)$ of radius $r>0$.

The closed geodesic ball of radius $R\in(0,\pi  r)$ in $S^p(r)$ will be denoted by $D^p_R(r)$. The metric of $D^p_R(r)$ can be described by the warped product metric $dt^2+r^2\sin ^2 (t/r) ds^2_{p-1}$, $t\in [0,R]$. Note that $\partial  D^p_R(r)= S^{p-1}(r \sin (R/r))$. If $R/r$ is smaller than $\pi /2$, then $D^p_R(r)$ is a spherical cap properly contained in a hemisphere and has convex boundary.

The geodesic ball $D^p_R(r)$ will also be denoted by $B_\epsilon ^p(\rho)$, where $\epsilon := R/r\in (0,\pi )$ and $\rho :=r \sin (R/r)$ is the radius of the sphere $\partial  D^p_R(r)$. Equivalently, $B_\epsilon ^p(\rho)$ is equal to $D^p_R(r)$ for $r=\rho/\sin (\epsilon)$ and $R=\epsilon \rho /\sin (\epsilon)$.

Next, we recall the second fundamental form $\mathrm{I\!I}$
of the boundary of a compact $(n+1)$-dimensional Riemannian manifold $(V,g_V )$. At a boundary point $p\in M:=\partial V$, the form $\mathrm{I\!I}_p$ is defined by
\begin{equation}\label{formula 2nd}\mathrm{I\!I}_p: T_pM\times T_pM\to \R,\quad (X,Y)\mapsto -g_V(\nabla  _{\underline{X}} \nu,\underline {Y}) _p.\end{equation}
Here, $\underline{X},\underline{Y}$ denote extensions of $X,Y$ to vector fields on a neighborhood of $p\in V$ which restrict to vector fields on the boundary, $\nu$ is the unit {\em inward} pointing normal vector field, and $\nabla $ is the Levi-Civita connection of $V$. The expression $-g_V(\nabla  _{\underline{X}} \nu,\underline {Y})= g_V( \nu,\nabla  _{\underline{X}} \underline {Y} )$ is symmetric in $\underline{X},\underline{Y}$ and independent of the chosen extensions. By the Koszul formula, $\mathrm{I\!I}$ at $p$ is given by
\begin{equation}\label{koszul}\mathrm{I\!I}(X,Y)=-\frac 1 2 \left(\nu g_V(\underline{X},\underline{Y})-g_V([\nu ,\underline{X}],\underline{Y})-g_V([\nu ,\underline{Y}],\underline{X}) \right).\end{equation}

The second fundamental form $\mathrm{I\!I}_p$ is symmetric and bilinear and can be represented by a diagonal $(n\times n)$ matrix $\hat{\mathrm{I\!I}}_p$ \wrt \ a suitable orthonormal basis of $T_pM$. The diagonal entries of $\hat {\mathrm{I\!I}}_p$ are the principal curvatures at $p$. Hence, $\mathrm{I\!I}_p $ is positive definite if and only if all principal curvatures are positive. If this is the case for all $p\in M$, then, using the first variation formula, it follows that $V$ has convex boundary in the sense that every two points of $V$ can be joined by a minimal geodesic whose interior lies in the interior of $V$. If, in addition, $V$ has positive Ricci curvature, then by Myers' argument the fundamental group of $V$ is finite.

Note that for the rescaled metric $\lambda ^2 g_V$, $\lambda >0$, the second fundamental form at $p$ is represented by $\frac 1 \lambda \hat {\mathrm{I\!I}}_p$.
The second fundamental form has a nice explicit description in terms of a collar of $V$. Suppose $(-\epsilon,0]\times M\hookrightarrow V$ is a collar on which $g_V$ takes the form $dt^2 +g(t)$ for a smooth family $g(t)$, $t\in  (-\epsilon,0]$, of metrics on $M$. Then $\mathrm{I\!I}$ is equal to $\frac 1 2 g^\prime (0)$. Note that if $\hat {\mathrm{I\!I}}$ represents $\mathrm{I\!I}$ at $(0,x)\in (-\epsilon,0]\times M$ and $f: (-\epsilon,0]\to \R $ is smooth, then $\frac {f^\prime(0)}{f(0)}\mathrm{I}_n+\hat {\mathrm{I\!I}}$ represents the second fundamental form of $M$ at this point for the metric $dt^2 +f^2(t)g(t)$ (where $\mathrm{I}_n$ denotes the $(n\times n)$ identity matrix). For example, the second fundamental form of the boundary of $D^p_R(r)$ is represented by $\frac 1 r \cot (\frac R r)ds^2_{p-1}$.
In particular, the second fundamental form of the boundary of a geodesic ball $B_\epsilon ^p(\rho)$ which is properly contained in a hemisphere (i.e., $\epsilon <\pi /2$) is positive definite.

The construction of $Ric>0$ metrics in Section \ref{generalized geometric plumbing} will rely (as in \cite{B19,R24}) on the following gluing result of Perelman. Let $V_1$ and $V_2$ be $(n+1)$-dimensional Riemannian manifolds and let  $M_1\subset V_1$ and $M_2\subset V_2$ be isometric boundary components with second fundamental forms $\mathrm{I\!I}_{M_1}$ and $\mathrm{I\!I}_{M_2}$, respectively.
\begin{theorem}[\cite{P97}]\label{Perelman gluing} Suppose $V_1$ and $V_2$ are of positive Ricci curvature, $\phi :M_1\to M_2$ is an isometry, and $\mathrm{I\!I}_{M_1}+\phi^*(\mathrm{I\!I}_{M_2})$ is positive semi-definite. Then $V_1\cup_\phi V_2$ admits a $Ric>0$ metric which is equal to the given metrics on $V_1$ and $V_2$ away from a small neighborhood of the boundary components.\qed
\end{theorem}

If $V_1$ is isometric to $(M\times (-\epsilon ,0],dt^2 + \kappa _-(t)g_M)$ near $M_1$, $V_2$ is isometric to $(M\times ( 0,\epsilon], dt^2 + \kappa _+(t)g_M)$ near $M_2$, and $\kappa _-(0)=\kappa _+(0)$, then the condition on the second fundamental forms is satisfied if $\kappa ^\prime_-(0)\geq   \kappa ^\prime_+(0)$. 
For example, the $C^0$-metric, obtained by identifying the boundary of two geodesic balls $B^p_{\epsilon_1}(\rho)$ and $B^p_{\epsilon_2}(\rho)$, can be smoothed out to yield a $Ric>0$ metric on $B^p_{\epsilon_1}(\rho)\cup _{S^{p-1}(\rho)}B^p_{\epsilon_2}(\rho)$ if $\cos (\epsilon _1)+\cos (\epsilon _2)\geq 0$. Moreover, the metric can be chosen to agree with the given metric outside of a small neighborhood of $S^{p-1}(\rho)$.

In \cite{P97} it is assumed that the sum $\mathrm{I\!I}_{M_1}+\phi^*(\mathrm{I\!I}_{M_2})$ is positive definite. However, the conclusion still holds under the weaker assumption above (see, for example, \cite[p. 3402]{R23}).

\begin{remark}\label{GL mean curvature remark}
If $V_1$, $V_2$ are of positive scalar curvature and the sum $\frac 1 n\tr (\mathrm{I\!I}_{M_1})+\frac 1 n\tr (\phi^*(\mathrm{I\!I}_{M_2}))$ of mean curvatures is nonnegative, then the metrics extend in the above sense to a $scal>0$ metric on $V_1\cup_\phi V_2$ (see \cite[\S 5]{GL80I}, \cite[p. 705]{G18}).
\end{remark}

Next we recall some relevant constructions of metrics of positive Ricci curvature. Let $G$ be a compact Lie group, $P\to B$ a principal $G$-bundle, $F$ a compact manifold with $G$-action, and $P\times _G F\to B$ the associated bundle. Let $g_B$ be a metric on $B$, $g_F$ a $G$-invariant metric on $F$, and let $P$ be equipped with a principal connection. Then, by Vilms' theorem \cite[Thm. 9.59]{B87}, there exists a unique metric $g$ on $P\times _G F$ such that  $P\times _G F\to B$ is a Riemannian submersion with totally geodesic fibers isometric to $(F,g_F)$ and horizontal distribution associated to the connection. The canonical variation $g_t$, $t>0$, is obtained by replacing $g$ by $t g$ in vertical direction. For $t<1$ this corresponds to shrinking the fibers of $P\times _G F \to B$. Suppose $B$ is compact. Then one has (see \cite[Prop. 9.70]{B87}):
\begin{proposition}\label{Proposition shrinking} If $(F,g_F)$ has positive scalar curvature, then $(P\times _G F,g_t)$ has positive scalar curvature for $t$ sufficiently small. If $(F,g_F)$ and $(B,g_B)$ have positive Ricci curvature, then $(P\times _G F,g_t)$ has positive Ricci curvature for $t$ sufficiently small.\qed
\end{proposition}

We will use the proposition for $P\to B$ a principal orthogonal bundle and for $F$ a round sphere or a geodesic ball in a round sphere.

Next, we recall from \cite{GY86} (see also \cite[Thm. 1.10]{W02}) that a metric of positive Ricci curvature can be deformed to a metric of positive Ricci curvature which looks around a given point like a small geodesic ball in the round unit sphere.

\begin{proposition}\label{Proposition local form} A metric $g_B$ on $B$ of positive Ricci curvature can be deformed in a neighborhood of a given point to a metric of positive Ricci curvature which has constant sectional curvature $1$ near the point.\qed
\end{proposition}
Hence, after deforming the metric on $B^q$ locally, we may assume that the metric has positive Ricci curvature and that $D^q_\delta (1)$ can be embedded isometrically into $B$. Here, $\delta >0$ depends on the topology of $B$ and may be very small. After rescaling, we may assume that the metric on $B$ has positive Ricci curvature and that there is an isometric embedding of the geodesic ball $D^q_R(N)$ into $B$ where $R/N=\delta$.

\begin{remark}\label{Remark local form}
For $B$ a $\Zp 2$-manifold and $pt\in B$ a fixed point, the argument in \cite{GY86,W02} also shows the following: An invariant metric $g_B$ on $B$ of positive Ricci curvature can be deformed in a neighborhood of $pt$ to an invariant metric of positive Ricci curvature which has constant sectional curvature $1$ near pt.
\end{remark}

\begin{example}\label{shrinking example} Let $P\to B^q$ be a principal $O(p)$-bundle, $p\geq 3$, and suppose the base manifold $B$ is compact and admits a metric of positive Ricci curvature. By Proposition \ref{Proposition local form} we can equip $B$ with a metric $g_B$ of positive Ricci curvature which comes with an isometric embedding $\overline \varphi :D^q_R(N)\hookrightarrow B$ where $R/N$ may be very small. We fix a principal connection on $P\to B$ which is trivial over $\overline \varphi (D^q_R(N))$.

Let $E:=P\times _{O(p)}D^p\to B$ and $\partial E:=P\times _{O(p)}S^{p-1}\to B$ be the associated disk and sphere bundles, respectively. We equip $D^p$  with the metric $g_{B_{\epsilon} ^p(1)}$ of a spherical cap.

Consider the associated Riemannian submersions $E\to B$ and $\partial E\to B$ with totally geodesic fibers and their canonical variations.
By shrinking the fibers sufficiently (see Proposition \ref{Proposition shrinking}), we see that there exists $\kappa >0$ such that for every $\rho <\kappa$ there is a metric $g_E$ on $E$ of positive Ricci curvature, which restricts to a metric $g_{\partial E}$ on $\partial E$ of positive Ricci curvature, and an isometric embedding $\varphi :D^q_R(N)\times B_{\epsilon} ^p(\rho)\hookrightarrow (E,g_E)$ covering $\overline\varphi $.

Moreover, the boundary $\partial E\subset E$ has positive mean curvature and nonnegative second fundamental form if $\epsilon <\pi /2$ and is totally geodesic if $\epsilon =\pi /2$. 
\end{example}
For $\epsilon <\pi /2$ the principal curvatures of the boundary of a totally geodesic fiber $\partial F_b\subset F_b$, $b\in B$, are positive whereas the second fundamental form vanishes on horizontal vectors. In particular, the mean curvature of $\partial E\subset E$ is positive. For $\epsilon =\pi /2$ the boundary $\partial E$ is totally geodesic by symmetry and, hence, has vanishing second fundamental form and vanishing mean curvature.

For later reference, we point out the following extensions of Example \ref{shrinking example} to the equivariant setting.
\begin{remark}\label{equivariant extension example and remark}  Let $P\to B^q$ be a $\Zp 2$-equivariant principal $O(p)$-bundle, $p\geq 3$, and suppose the base manifold $B$ is compact and admits an invariant metric of positive Ricci curvature. Suppose the induced $\Zp 2$-action on $E:=P\times _{O(p)}D^p$ has an isolated fixed point. Then there exist $R,N>0$ and a positive constant $\kappa$ such that for every $\rho <\kappa$ there is an invariant metric $g_E$ on $E$ of positive Ricci curvature, which restricts to a metric $g_{\partial E}$ on $\partial E$ of positive Ricci curvature, and an isometric equivariant embedding $\varphi :D^q_R(N)\times B_{\epsilon} ^p(\rho)\hookrightarrow (E,g_E)$.

Moreover, the boundary $\partial E\subset E$ has positive mean curvature and nonnegative second fundamental form if $\epsilon <\pi /2$ and is totally geodesic if $\epsilon =\pi /2$. 
\end{remark}

Here, $\Zp 2$ acts isometrically on $D^q_R(N)\times B_{\epsilon} ^p(\rho)$ by $\pm (\id, \id)$ and $\varphi $ is an equivariant coordinate function of $E\to B$ at the isolated fixed point. The proof uses Remark \ref{Remark local form} and shrinking of the fibers. 

For further reference, we will also recall the notion of a core metric, as introduced in \cite{B19}. A Riemannian metric $g$ on $M^n$ is a \emph{core metric} if there is an embedding $\iota : D^n\hookrightarrow M$ such that $g$ has  positive Ricci curvature on $M\setminus \iota (D^n)$,  the boundary of $M\setminus \iota (D^n)$ has positive definite second fundamental form, and the metric restricted to its boundary is round.
For example, the metric of a round sphere $S^n$, $n\geq 2$, is a core metric (consider the inclusion of a geodesic ball which properly contains a hemisphere).

\subsection{Generalized geometric plumbing}\label{generalized geometric plumbing}
Next, we want to equip certain generalized plumbings with a $scal>0$ metric which restricts to a $Ric>0$ metric on the boundary such that the metric is of product type near the boundary. The generalized plumbings will be of the form $(\square _{\Gamma}E_\vvv)\square V$ where $\Gamma $ is a tree and $E_\vvv$ denotes the total space of a disk bundle over the sphere $S^p$ (resp. $S^q$) of rank $q$ (resp. $p$). The plumbing of $\square _{\Gamma}E_\vvv$ and $V$ is taken \wrt \  a neat embedding into $V$ and a coordinate function of some disk bundle $E_\vvv\to B_\vvv$.

We will restrict to manifolds $V$ which admit a family of special metrics with properties given in Definition \ref{nice coordinate function} below. In particular, we assume that $V$ contains a subspace which metrically looks like a disk bundle when restricted to a coordinate function $\varphi$ as described in Example \ref{shrinking example} and that the metrics yield metrics of positive Ricci curvature on $\partial V$. The key step in the construction is then to show that for a disk bundle $E_\vvv\to B_\vvv$ the plumbing $E_\vvv\square V$ also has such a special family of metrics. By iterating this construction, one arrives at the desired metric on $(\square _{\Gamma}E_\vvv)\square V$.

\begin{definition} \label{nice coordinate function}
Let $V^{p+q}$ be a smooth manifold with boundary. A neat embedding $\psi: D^p\times D^q\hookrightarrow V$ is called a {\bf \em nice metric coordinate function} if there exists $R,N>0$ and a positive constant $\kappa$ such that for every $\rho <\kappa$ there exists a metric $g_V$ on $V$ with the following properties:

The metric $g_V$ has $scal>0$ on $V$,  has $Ric>0$ on $\partial V$, the mean curvature of $\partial V$ is nonnegative, and $\psi$ defines an isometric embedding $D^p_R(N)\times B_{\pi /2} ^q(\rho)\hookrightarrow (V,g_V)$.
\end{definition}

The condition on the mean curvature will be used to deform $g_V$ into a metric which is of product type near the boundary.
The next remark follows from Example \ref{shrinking example}.
\begin{remark}\label{nice coordinate function remark} The total space $E$ of any disk bundle of rank $\geq 3$ over a compact base manifold of positive Ricci curvature has a nice metric coordinate function (where $R/N, \kappa$ depend on the bundle and maybe very small).
\end{remark}

In this situation one has the following result for $p,q\geq 3$ (compare \cite{W95,W97,W98,W11,R23}).

\begin{theorem}\label{plumbing theorem} Let $V$ be a smooth manifold with nice metric coordinate function $\psi: D^p\times D^{q}\hookrightarrow V$, let $W$ be a plumbing of disk bundles over spheres according to a tree, and let $W\square V$ be the plumbing \wrt \ $\psi$ and a coordinate function of a disk bundle $E\to S^q$ of $W$. Then $W\square V$ has a metric of positive scalar curvature which is of product type near the boundary and which restricts to a metric of positive Ricci curvature on the boundary $\partial (W\square V)$.
\end{theorem}
The existence statement for the $Ric>0$ metric on the boundary was first shown by Wraith if $p=q$ \cite{W95,W98} or if the tree is a straight line \cite[Thm. 7.1]{W95}. The extension to arbitrary trees with $p\neq q$ is due to Reiser \cite{R23}. The statement on the scalar curvature was first shown  in a slightly different setting by Wraith \cite{W11}, up to a deformation of the boundary metric within positive scalar curvature. The techniques in \cite{W98,W11,R23} can be combined to show the conclusion of the theorem, up to such a deformation, which is sufficient for the applications of this paper (see the discussion in \cite[\S 3.2]{RW25}).

We will give a slightly different construction following the method of \cite{R23}. This construction can be used to show that the theorem above also holds true in the more general situation where $W$ is a plumbing of disk bundles $E_\vvv \to B_\vvv$  according to a tree provided each base manifold $B_\vvv$ admits a core metric. For the purpose of this paper it suffices to consider bundles over spheres as in the theorem above.

We will give a sketch of the proof and refer to Section \ref{proof plumbing theorem} for details.

\noindent
\begin{proof}[Sketch of the proof of Theorem \ref{plumbing theorem}] We will first discuss the construction of the metric for the plumbing $E\square V$ of a disk bundle $E\to S^q$ and $V$. The construction can then be iterated to yield a metric on $W\square V$ which satisfies all properties except for the product structure near the boundary, which will be taken care of at the end of the proof.

We decompose the base $S^q$ of the bundle $E$ as a union
$$S^q=B^q_-\cup Z_1\cup Z_2\cup Z_3\cup B^q_+,$$ where $B^q_{-}$ is a closed ball of radius $a_1>0$, $Z_i$ is a cylinder of the form $Z_i= I_i\times S^{q-1}$, $I_i=[a_i,b_i]$, $a_1<b_1=a_2<b_2=a_3<b_3$, and $B^q_{+}=D^q$ is another closed ball.

Let $P\to S^q$ denote the principal $SO(p)$-bundle associated to $E\to S^q$. We fix a trivialization of $P$ over $S^q\setminus B^q_-$ and a principal connection which is trivial over $S^q\setminus B^q_-$. The restriction of $P$ to $S^q\setminus B^q_-$ will be identified with a principal product bundle with trivial principal connection.

\medskip
\noindent
The metric on $E$ will be constructed from metrics on $E_{\vert Z_3}$ and $E_{\vert S^q\setminus Z_3}$. The latter is obtained as follows:

\medskip
\noindent
\underline {Metric $g_B$ on $S^q\setminus Z_3$:} We choose a metric $\tilde g_B=dt^2 +k^2(t)ds^2_{q-1}$ such that $(B^q_{-}\cup Z_1\cup Z_2,\tilde g_B)$ is a geodesic ball in an elliptic paraboloid with boundary the unit sphere $S^{q-1}(1)$ and $0<k^\prime (b_2)<\frac 1 2$. Using Proposition \ref{Proposition local form}, we change the metric $\tilde g_B$ locally at points $p_i$ in the interior of $Z_1$ to obtain a $Ric>0$ metric  $g_B$ such that its restriction to a neighborhood of each $p_i$ is isometric to a small spherical cap $D^q_{\epsilon_i} (1)$ in the unit sphere.
The equip $B^q_{+}$ with the metric (again denoted by $g_B$) of a hemisphere in $S^q(\rho)$, $\rho >0$.

\medskip
\noindent
\underline {Metric $g_+$ on $E_{\vert B^q_+}$:} On $E_{\vert B^q_+}=B^q_+\times D^p=D^q\times D^p$ we choose $g_+$ such that $(E_{\vert B^q_+},g_+)=B_{\pi /2} ^q(\rho) \times D^p_R(N)$. Here, $R$, $N$, and $\rho$ are the parameters given by the nice metric coordinate function $\psi$.

\medskip
\noindent
\underline {Metric $g_{-\cup 1}$ on $E_{\vert B^q_-\cup Z_1}$:} Using Vilms' theorem for the principal connection above and fiber $F=B^p_{\pi /2}(r)$, $r>0$, we obtain a metric $g_{-\cup 1}$ on $E_{\vert B^q_-\cup Z_1}$ such that $(E_{\vert B^q_-\cup Z_1},g_{-\cup 1})\to (B^q_-\cup Z_1,g_B)$ is a Riemannian submersion with totally geodesic fibers isometric to $B^p_{\pi /2}(r)$.
By construction, there is an isometric neat embedding $\psi_i: D^q_{\epsilon_i} (1) \times B^p_{\pi /2}(r)\hookrightarrow E$ onto a trivialization of $E$ at $p_i$, which can be used for subsequent plumbings.

\medskip
\noindent
\underline {Metric $g_2$ on $E_{\vert Z_2}$:}
We equip the fiber of $E_{\vert Z_2}=Z_2\times D^p$ over $(t,x)\in Z_2=I_2\times S^{q-1}$ with a spherical cap metric $g(t)$ such that $(D^p, g(t))=B^p_{\epsilon (t)}(r)$, where
$\epsilon (t):I_2=[a_2,b_2]\to \R _{>0}$ is a smooth, monotonously decreasing function with $\epsilon (a_2)=\pi /2$, which is constant near the endpoints ($\epsilon (b_2)<\pi /2$ will be determined later, see (\ref{bc 1 sketch})).
The resulting metric on $E_{\vert Z_2}$ is of the form $g_2=dt^2 +k^2(t)ds^2_{q-1}+g(t)$.

\medskip
The metrics  $g_{-\cup 1}$and $g_2$ define a smooth metric on $E_{\vert B^q_-\cup Z_1\cup Z_2}$ for which $E_{\vert B^q_-\cup Z_1\cup Z_2}\to B^q_-\cup Z_1\cup Z_2$ is a Riemannian submersion. 
After shrinking the fibers sufficiently, i.e., taking $r$ sufficiently small,  this metric has positive Ricci curvature on $(\partial E)_{\vert B^q_-\cup Z_1\cup Z_2}$ and on $E_{\vert B^q_-\cup Z_1}$ (see Proposition \ref{Proposition shrinking}). Moreover, it can be shown (using \cite[Cor. 9.37]{B87}) that $(E_{\vert Z_2},g_2)$ has positive scalar curvature for $r$ sufficiently small. The metric on $E_{\vert S^q\setminus Z_3}$ is given by this metric together with the metric $g_+$ on $E_{\vert B^q_+}$.

\medskip
\noindent
\underline {Metric $\bar g_{f,h}$ on $E_{\vert Z_3}$:}
Next, we will define a metric on $E_{\vert Z_3}=Z_3\times D^p=[a_3,b_3]\times S^{q-1}\times D^p$, which is of positive scalar curvature, restricts to a $Ric>0$ metric on $\partial E_{\vert Z_3}$, and which is compatible with the metrics above, after possibly rescaling. This is the crucial step in the construction and is essentially due to Reiser \cite{R23,R25}.
The metric on $(\partial E)_{\vert Z_3}$ will be a doubly warped product metric of the form $$g_{f,h}=dt^2+h^2(t)ds^2_{q-1} + f^2(t)ds^2_{p-1},$$
which
depends on two smooth functions $f,h:[a_3,b_3]\to \R _{>0}$ to be determined later. The extension of $g_{f,h}$ to $E_{\vert Z_3}$ will be obtained by the following construction.

We identify $[a_3,b_3]\times D^p$ with a subspace $X$ of the Riemannian cylinder $\R \times S^p (\beta N)$, $\beta >0$, using 
$$\Phi: [a_3,b_3]\times D^p\to \R \times S^p, \quad (t,(s,x))\mapsto (\varphi ( t),(F(\varphi ( t))s,x)), \quad s\in [0,1], x\in S^{p-1},$$ where
$\varphi (t):[a_3, b_3]\to [a_3, \tilde b_3 ]$, $t\mapsto \tilde t $,  is a diffeomorphism
with $\varphi (a_3)=a_3$ and $F(\tilde t):[a_3, \tilde b_3]  \to (0, \beta N \pi)$ is
 a smooth map such that the curve $t \to (\varphi (t), F(\varphi (t))$ in $\R ^2$ is parametrized by arc length and such that $$\frac {f(t)}{\beta N}=\sin \left (\frac {F(\varphi (t))}{\beta N}\right)\text{ for } a_3\leq t\leq b_3.$$ 
Let $X:= \{ (\tilde t, (s, x))\in \R \times S^p \, \mid \, \tilde t\in [a_3,\tilde b_3],\, s\leq F(\tilde t), \, x\in S^{p-1}\}$ be the image of $\Phi$ and let $g_X$ denote the metric on $X$  induced by the inclusion $X\subset \R \times S^p (\beta N)$.

Now let $\bar g_{f,h}$ be the metric on $E_{\vert Z_3}=[a_3,b_3]\times S^{q-1}\times D^p$ defined by requiring that

\begin{itemize}
\item $\bar g_{f,h}$ restricted to $[a_3,b_3]\times \{pt\} \times D^p$ is equal to the pullback of $g_X$ under $\Phi$ \, $\forall \, pt\in S^{q-1}$,
\item $\bar g_{f,h}$ restricted to $\{t\}\times S^{q-1}\times \{pt\}$ is equal to $h^2(t)ds^2_{q-1}$  \, $\forall \, (t,pt)\in [a_3,b_3]\times D^p$, and
\item the tangent spaces of $S^{q-1}$ and $[a_3,b_3]\times D^p$ at every point of $[a_3,b_3]\times S^{q-1}\times D^p$ are orthogonal with respect to $\bar g_{f,h}$.
\end{itemize}

Note that $\bar g_{f,h}$ restricts to $g_{f,h}$ on the boundary. A direct computation shows that $(E_{\vert Z_3},\bar g_{f,h})$ has positive scalar curvature.

\medskip
In the following we will consider the metrics $\alpha ^2 g_{-\cup 1}$, $\alpha ^2 g_2$, $\bar g_{f,h}$, and $\beta ^2 g_+$ on the corresponding pieces of $E$ as well as $\beta ^2 g_V$ on $V$, where $\alpha ,\beta >0$ will be determined later. The task is to choose $f,h$ and parameters above such that the metrics can be glued together (maybe after a subsequent smoothing) to yield a metric on $E\square V$ which is of positive scalar curvature, which restricts to a metric of positive Ricci curvature on the boundary $\partial (E\square V)$, and for which the boundary has nonnegative mean curvature. To ensure this, we will require, in particular, that $\bar g_{f,h}$ agrees with $\alpha ^2 g_2$ and $\beta ^2 g_+$ at the respective boundary part of $E_{\vert Z_3}$ and that Perelman's conditions on the second fundamental forms (see Theorem \ref{Perelman gluing}) are fulfilled at the boundary components of $(\partial E)_{\vert Z_3}$. These boundary conditions are satisfied if
\begin{equation}\label{bc 1 sketch}\epsilon (b_2)=\arcsin \left (\frac {f(a_3)} {\beta N}\right), \quad h(a_3)= \alpha , \quad h^\prime (a_3)\leq  k^\prime (b_2), \quad f(a_3)=\alpha r, \quad \text{and } \quad f^\prime (a_3)= 0,\end{equation}
 as well as
\begin{equation}\label{bc 2 sketch} h(b_3)= \beta \rho , \quad h^\prime (b_3)= 0, \quad f(b_3)=\beta  N\sin (R/N), \text{ and }f^\prime (b_3)= \cos (R/N).\end{equation}

In \cite{R23} Reiser constructs smooth functions $f,h:[a_3,b_3]\to \R _{>0}$ which satisfy conditions (\ref{bc 1 sketch}) and (\ref{bc 2 sketch}) such that, in addition, the  metric $g_{f,h}$ on $(\partial E)_{\vert Z_3}$ has positive Ricci curvature. We will review some of the properties of these functions which will be also important for subsequent mean curvature estimates.

\medskip
\noindent
\underline {Functions $f,h$:}
The smooth functions $f,h:[a_3,b_3]\to \R _{>0}$ are constructed from smooth pieces $f_l,h_l$ (resp. $f_r, h_r$) which are defined on $[a_3,t_1]$ (resp. $[t_1,b_3]$), where $t_1$ and $b_3$ maybe very large, and have the following properties (see Section \ref{proof plumbing theorem} for details):
\begin{itemize}
\item $f_l,h_l$ satisfy the left boundary conditions (\ref{bc 1 sketch}),  $f_r,h_r$ satisfy the right boundary conditions (\ref{bc 2 sketch}),
\item  the metrics $dt^2+h_l^2(t)ds^2_{q-1} + f_l^2(t)ds^2_{p-1}$ and $dt^2+h_r^2(t)ds^2_{q-1} + f_r^2(t)ds^2_{p-1}$, defined on their respective part of $(\partial E)_{\vert Z_3}$, have positive Ricci curvature, if $\alpha ,\rho$ are sufficiently small and $\beta$ is sufficiently large,
\item $h_l,h_r$ define $h$ which is smooth at $t_1$,
\item $f_l,f_r$ satisfy $f_l(t_1)=f_r(t_1)$, $f_l^\prime(t_1)=f_r^\prime(t_1)$, $f_l^{\prime \prime}(t_1)>0$, $f_r^{\prime \prime} (t_1)\leq 0$ and, hence, define a $C^1$-function which is smooth for $t\neq t_1$.
\end{itemize}

By Theorem \ref{Perelman gluing}, the functions $f_l$ and $f_r$ can be smoothed at $t=t_1$ 
to obtain a smooth function $f$ on $[a_3, b_3]$, such that the metric $g_{f,h}=dt^2+h^2(t)ds^2_{q-1} + f^2(t)ds^2_{p-1}$ on $(\partial E)_{\vert Z_3}=[a_3,b_3]\times S^{q-1}\times S^{p-1}$ has positive Ricci curvature. 

We will use $\Phi$ to identify $E_{\vert Z_3}=[a_3,b_3]\times S^{q-1}\times D^p$ with $X\times S^{q-1}\subset [a_3,\tilde b_3] \times S^p \times S^{q-1}$ and denote by $\hat g_{f,h}$ the metric on $X\times S^{q-1}$ induced by $\bar g_{f,h}$. Since $f_r,h_r$ satisfy (\ref{bc 2 sketch}), the fiber $F_{\tilde b_3}$ of $(X\times S^{q-1},\hat g_{f,h})\to [a_3,\tilde b_3]$ over $\tilde b_3$ is isometric to $\partial B^q_{\pi /2}(\beta \rho)\times D^p_{\beta R}(\beta N)$.

We will also identify $(E_{\vert B^q_+},\beta ^2g_+)$ with the  image of $D^p_{\beta R}(\beta N)  \times B_{\pi /2} ^q(\beta \rho)$ in $(V,\beta ^2 g_V)$ under the isometry given by the nice metric coordinate function $\psi$. 

Under these identifications, taking the union $E_{\vert Z_3}\cup E_{\vert B^q_+}$ corresponds to gluing $X\times S^{q-1}$ and $V$ by identifying $F_{\tilde b_3}$ isometrically with  $\psi(D^p_{\beta R}(\beta N)  \times \partial B_{\pi /2} ^q(\beta \rho))\subset V$.

The function $f_r$ can be chosen such that, in addition to the properties above, the curve $\tilde t \mapsto (\tilde t, F(\tilde t))$ induces a smoothing of the corners which arise in the plumbing construction for $E$ and $V$.  Moreover, the function $h_r$ can be chosen such that the metric  on the gluing of $X\times S^{q-1}$ and $V$ is also smooth near the gluing area (see Section \ref{proof plumbing theorem} for details).

This gives a metric on $E\square V$  which is smooth except for points with $t=a_3$. It follows from Theorem \ref{Perelman gluing} that the functions $f$ and $h$ can be smoothed at $t=a_3$ such that the new metric on $E\square V$ is of positive scalar curvature and restricts to a $Ric >0$ metric on $\partial (E\square V)$.

\medskip
\noindent
The construction above can be carried out such that the boundary of $E\square V$ has nonnegative mean curvature. For the boundary part belonging to $V$, this is the case, since $V$ admits a nice metric coordinate function. For the part over $B^q_-\cup Z_1$, the mean curvature vanishes, since the boundary is totally geodesic. It follows from the construction that for $r$ sufficiently small the mean curvature of $\partial E\subset E$ is nonnegative over $Z_2$. The analysis for the part over $Z_3$ is more involved. A computation shows, roughly speaking, the following:

The principal curvatures in direction $S^{q-1}$ contribute in the negative to the mean curvature. The contribution in "$t$-direction" is negative if and only if  $f^{\prime \prime}>0$. Moreover, these contributions become arbitrarily small for $\beta \gg 0$. On the other hand, the principal curvatures in direction $S^{p}$ are uniformly bounded from below by a positive constant (see Section \ref{proof plumbing theorem} for precise statements and details). Hence, the mean curvature of $\partial E\subset E$ is positive over $Z_3$ provided $\beta $ is sufficiently large.

Altogether, we see that the metric $g_{E \square V}$ on $E \square V$ can be chosen such that the boundary of $E\square V$ has nonnegative mean curvature.

Since $\psi_i: D^q_{\epsilon_i} (1) \times B^p_{\pi /2}(r)\hookrightarrow (E\square V, g_{E \square V})$ is, up to a scaling factor $\alpha $, an isometric embedding and $r>0$ can be chosen to be arbitrarily small, $E\square V$ has a nice metric coordinate function $\psi _i$ which can be used for a subsequent plumbing.

By iterating this process, one obtains a $scal>0$ metric $g_{W\square V}$ on $W\square V$ which restricts to a $Ric>0$ metric $g_{\partial (W\square V)}$ on $\partial (W\square V)$ such that the boundary has nonnegative mean curvature.

\medskip
Finally, we add a metric color $(\partial (W\square V)\times [0,\delta], g_{\partial (W\square V)} +ds^2)$  to $(W\square V,g_{W\square V})$ along $\partial (W\square V)=\partial (W\square V)\times \{0\}$ to obtain a $C^0$-metric on the union. By Remark \ref{GL mean curvature remark}, this metric can be deformed near the gluing area to yield a smooth metric of positive scalar curvature on $(\partial (W\square V)\times [0,\delta])\cup W\square V\cong W\square V$. The resulting metric has positive Ricci curvature at the boundary and is of product type near the boundary.
\end{proof}

\begin{remark}\label{remark plumbing nice coordinate function} The proof shows that if $W$ is a plumbing of disk bundles over spheres according to a tree and $V$ has a nice metric coordinate function, then $W\square V$ also has a nice metric coordinate function.
\end{remark}

\section{Equivariant plumbing}\label{Equivariant plumbing}
In dimension $4k+1$ we will use eta invariants to distinguish geometries of positive Ricci curvature. Since these invariants are trivial for simply connected manifolds, we will consider manifolds with non-trivial fundamental group. The manifolds under consideration arise as quotient of the boundary of manifolds with nice equivariant metric coordinate function (see Definition \ref{nice equivariant coordinate function} below) by a free $\Zp 2$-action. For the construction of $Ric >0$ metrics on these quotients, we will need to extend the results from the last sections to the equivariant setting.

First, we discuss equivariant plumbing of disk bundles.
Let $\xi _1$ be a $\Zp 2$-equivariant oriented real Riemannian vector bundle of rank $p$ over a compact connected oriented base manifold $B_1^q$. Assume $\Zp 2$ acts on $\xi _1$ with isolated fixed points. Let $E_1\to B_1$ and $\partial E_1\to B_1$ be the associated linear disk and sphere bundle, respectively. Note that $\Zp 2$ acts on $B_1$ with isolated fixed points and that at a fixed point there is an equivariant orientation preserving coordinate function $\varphi _1: D^q\times D^p\hookrightarrow E_1$ where $\Zp 2$ acts on $D^q\times D^p$ by $\pm (\id, \id)$. Also note that $\Zp 2$ acts freely on $\partial E_1$.

\begin{example}\label{equivariant bundle example}
Let $\Zp 2$ act linearly on $S^q=D^q_+\cup D^q_-$ with two fixed points (we assume the non-trivial element of $\Zp 2$ acts by $-\id $ on $D^q_{\pm}$). Let $\xi$ be a rank $p$ oriented vector bundle over $S^q$ given by a characteristic map $\phi:S^{q-1}\to SO(p)$. Assume $\phi (-x)=\phi (x)$ for all $x\in S^{q-1}$. Then the $\Zp 2$-action can be lifted to the disk bundle associated to $\xi$ by letting the non-trivial element of $\Zp 2$ act on $D^q_{\pm}\times D^p$ by $(-\id ,-\id)$.
\end{example}

Next, consider another $\Zp 2$-equivariant oriented real Riemannian vector bundle $\xi _2$ of rank $q$ over an oriented compact connected base manifold $B_2^p$. Again we assume that $\Zp 2$ acts with isolated fixed points. Let  $E_2\to B_2$ denote the associated disk bundle and let  $\varphi _2: D^p\times D^q\hookrightarrow E_2$ be an equivariant orientation preserving coordinate function at a fixed point.

In this situation one can plumb the two disk bundles together using the equivariant coordinate functions and cross identification. The resulting space $E_1\square E_2$ is obtained from the disjoint union $E_1\coprod E_2$ by identifying $\varphi _1(x,y)$ with $\varphi _2(y,x)$, $(x,y)\in D^q \times D^p$, and straightening out the angles equivariantly. Note that the induced $\Zp 2$-action on $E_1\square E_2$ has isolated fixed points and is free on the boundary $\partial (E_1\square E_2)$.

The plumbing construction can be iterated and thereby generalized to a plumbing of $\Zp 2$-equivariant disk bundles at isolated fixed points according to a suitable graph.
Note that the equivariant diffeomorphism type of the plumbing may depend on the choice of the equivariant coordinate functions.

Next, we consider the generalized plumbing construction described in \S \ref{subsection generalized plumbing} in the equivariant context. Let $V^{p+q}$ be a smooth $\Zp 2$-manifold with neat equivariant embedding $\psi: D^p\times D^q \hookrightarrow V$ where, as before, $\Zp 2$ acts on $D^p\times D^q$ by $\pm (\id, \id)$. Suppose that $Q^{p+q}$ is another smooth $\Zp 2$-manifold with boundary and $\zeta : D^q\times D^p\hookrightarrow Q$ a neat  equivariant embedding. Then one can define the equivariant plumbing $Q\square V$ \wrt \ $(\psi, \zeta)$ by taking the disjoint union $Q\coprod V$ and identifying $\zeta (y,x)$ with $\psi (x,y)$, $(x,y)\in D^p \times D^q$, and straightening out the angles equivariantly.

We will also need to adapt the geometric plumbing discussed in \S \ref{generalized geometric plumbing} to the equivariant setting. The next definition corresponds to Definition \ref{nice coordinate function}.

\begin{definition} \label{nice equivariant coordinate function}
Let $V^{p+q}$ be a smooth $\Zp 2$-manifold with boundary. A neat equivariant embedding\linebreak $\psi: D^p\times D^q\hookrightarrow V$ is called a {\bf \em nice equivariant metric coordinate function} if there exists $R,N>0$ and a positive constant $\kappa$ such that for every $\rho <\kappa$ there exists an invariant metric $g_V$ on $V$ with the following properties:

The metric $g_V$ has $scal>0$ on $V$,  has $Ric>0$ on $\partial V$, the mean curvature of $\partial V$ is nonnegative, and $\psi$ defines an isometric embedding $D^p_R(N)\times B_{\pi /2} ^q(\rho)\hookrightarrow (V,g_V)$.
\end{definition}

As in the non-equivariant situation, nice equivariant metric coordinate functions occur naturally for disk bundles (see Remark \ref{equivariant extension example and remark}).

\begin{remark}\label{nice equivariant coordinate function remark} Let $E\to B^p$ be a $\Zp 2$-equivariant disk bundle of rank $q\geq 3$. Suppose the $\Zp 2$-action on $E$ has an isolated fixed point. Suppose $B$ is compact and has an invariant metric of positive Ricci curvature. Let $\psi: D^p\times D^q\hookrightarrow E$ be an equivariant coordinate function of $E$ at the isolated fixed point. Then $\psi $ is a nice equivariant metric coordinate function (where $R/N, \kappa$ depend on the bundle and maybe very small).
\end{remark}

The proof of Theorem \ref{plumbing theorem} can be adapted to show the following equivariant version which is sufficient for the purposes of this paper  (see \cite[\S 7.2]{DG21} for some details). 

Let $E_i\to S^{n_i}$, $i=1,\ldots , k$, be equivariant disk bundles of rank $r_i$, as in Example \ref{equivariant bundle example}, with $\{r_i,n_i\}=\{p,q\}$. As before, we assume $p,q\geq 3$.

Suppose that $W:=E_1\square \ldots \square E_k$ is the plumbing at isolated fixed points of these disk bundles according to a straight line (note that the equivariant plumbing construction is restricted to a straight line since $\Zp 2$ acts on each disk bundle above with two isolated fixed points). Consider an equivariant coordinate function $\varphi _k$ of $E_k$ at the isolated fixed point which has not been used for the plumbing. 
Then one has the following equivariant version of Theorem \ref{plumbing theorem}.
\begin{theorem}\label{equivariant plumbing theorem} Assume $V$ has a nice equivariant metric coordinate function $\psi: D^p\times D^{q}\hookrightarrow V$. Let $W\square V$ be the equivariant plumbing \wrt \ $\psi$ and the equivariant coordinate function $\varphi _k$. Then $W\square V$ has an invariant metric of positive scalar curvature which is of product type near the boundary and which restricts to an invariant metric of positive Ricci curvature on the boundary $\partial (W\square V)$.
\end{theorem}

\begin{proof} The proof is essentially the same as the one for Theorem \ref{plumbing theorem} up to a few minor changes. It suffices to consider the construction of the metric on $E_k\square V$. We will assume that $E_k$ has rank $p$ (the same reasoning applies if $E$ is of rank $q$). We follow the notation in the proof of Theorem \ref{plumbing theorem} and let $E:=E_k\to S^q$.

We decompose the base sphere as a union $S^q=B^q_-\cup Z_1\cup Z_2\cup Z_3\cup B^q_+$ such that the $\Zp 2$-fixed points are the centers of the balls $B^q_{\pm }$ ($\Zp 2$ acts by $\pm \id$ on the parallels $S^{q-1}$). In the proof of Theorem \ref{plumbing theorem} we make the following changes: Let $U$ be a small invariant ball around the center of $B^q_-$. We fix a trivialization of the equivariant principal $SO(p)$-bundle $P$ over $(S^q\setminus B^q_-)\cup U$ and fix an invariant connection which is trivial over $(S^q\setminus B^q_-)\cup U$. We modify the metric $\tilde g_B$ on $U$ such that $U$ is equivariantly isometric to a small spherical cap $D^q_\epsilon$ (this can be done inside positive sectional curvature). However, we do not make any modifications of $\tilde g_B$ at other points $p_i$ since there are no further fixed points in $S^q\setminus B^q_+$.

With these modifications all metric constructions in the proof of Theorem \ref{plumbing theorem} are invariant under the $\Zp 2$-action and, hence, yield a metric with the properties stated in Theorem \ref{equivariant plumbing theorem}.
\end{proof}

Let $\varphi _1$ denote an equivariant coordinate function of $E_1$ at the isolated fixed point which has not been used for the plumbing.
\begin{remark}\label{remark equivariant plumbing nice coordinate function} It follows from the proof that $\varphi _1$ is a nice equivariant metric coordinate function of $W\square V$.
\end{remark}

Note that if $\Zp 2$ acts freely on $\partial V$, then the $\Zp 2$-action on $\partial (W\square V)$ is also free. 
In Section \ref{section 4k+1} we will use Theorem \ref{equivariant plumbing theorem} to exhibit $\Zp 2$-quotients which admit infinitely many different geometries of positive Ricci curvature.

\section{Index theory for $Spin^c$-manifolds}\label{index section}
In this section we briefly recall index formulas for Dirac operators, eta invariants as well as equivariant refinements. We will restrict the discussion to $Spin^c$-structures which include spin structures as a special case.
We first review the construction of the group $Spin^c(k)$ and its representations following \cite{ABS64,
G85,
LM89}.
\subsection{The group $Spin^c(k)$}
Let $V$ be a $k$-dimensional real vector space with scalar product $\langle \; , \; \rangle $, ${\cal T}(V)=\R \oplus V \oplus V\otimes V \oplus \ldots$ its tensor algebra, and $Cl(V):={\cal T}(V)/{\cal I}(V)$ the Clifford algebra. Here, ${\cal I}(V)$ is the two-sided ideal in ${\cal T}(V)$ generated by $v\otimes w + w\otimes v + 2\langle v,w\rangle $, $v,w\in V$.

The endomorphism $-\id _V$ induces an endomorphism $\alpha :Cl(V)\to Cl(V)$ of the Clifford algebra and a $\Zp 2$-grading $Cl(V)=Cl^0(V)\oplus Cl^1(V)$ which corresponds to the decomposition into $\pm $-eigenspaces of $\alpha $.

The linear involution on ${\cal T}(V)$, defined by $v_1\otimes \ldots \otimes v_k\mapsto v_k\otimes \ldots \otimes v_1$, leaves the ideal ${\cal I}(V)$ invariant and descends to the \emph{transpose} $(\; )^t :Cl(V) \to Cl(V)$, which is linear, grading preserving, and satisfies $(x \cdot y)^t=y^t\cdot x^t$.

Let $(e_1,\ldots,e _k)$ be an orthonormal basis (ONB) of $V$. The Clifford algebra $Cl(V)$ can be identified with the $\R $-algebra generated by the unit $1$ and the symbols $e^\prime _i$, $i=1,\ldots ,k$, subject to the relations ${e^\prime _i} ^2=-1$ $\forall \, i$ and $e^\prime _i e^\prime _j+e^\prime _j e^\prime _i=0$ $\forall \, i\neq j$ (the identification is induced by $e_i\mapsto e^\prime _i$, $i=1,\ldots ,k$).

The volume element $\omega:=e_1\cdot \ldots \cdot e _k$ satisfies $\omega ^2=(-1)^{\frac {k(k+1)} 2 + k}$ and $e _i\omega =(-1)^{k-1} \omega  e _i$ $\forall \, i$. Hence, $\omega x = x \omega $ $\forall \, x\in Cl(V)$ if $k$ is odd and $\omega x =\alpha (x)  \omega $ $\forall \, x\in Cl(V)$ if $k$ is even (see [LM89], Prop. 3.3, p. 22).

For $k=2l$, let $\alpha _l:=i^l\omega $. Note that $\alpha _l^2=1$, $\alpha _l x=\alpha (x)  \alpha _l$ $\forall \, x\in Cl(V)$, and $\alpha _l$ only depends on the orientation given by the ONB $(e _1,\ldots,e _{2l})$ (see [LM89], p. 34).

Let $pin(V):=\{x\in Cl(V)\, \mid \, x=v_1\cdot \ldots \cdot v_n,\,  v_i\in V, \, \vert \vert v_i\vert \vert =1 \, \forall \, i, n\in\N \}$. It follows that $(pin(V), \cdot)$ is a closed subgroup of the group of units of $(Cl(V),\cdot)$ and is, in fact, a compact Lie group.

Any vector $v\in V$ of length $\vert\vert v\vert \vert =1$ acts on $V$ by $y\mapsto v\cdot y\cdot v^t$ as the reflection along the hyperplane $v^\perp\subset V$. The map $\rho (x) (y):=x\cdot y\cdot x^t$ defines a surjective homomorphism $pin(V)\to O(V)$ of Lie groups which can be seen to be a twofold covering with kernel equal to $\{\pm 1\}$, i.e., one has a short exact sequence of compact Lie groups
$0\to \Zp 2\to pin(V)\overset{\rho} \to O(V)\to 0$.

Let $Spin(V):= pin(V)\cap Cl^0(V)$. The elements of $Spin(V)$ act by an even number of reflections on $V$ and one obtains a short exact sequence
 $$0\to \Zp 2\to Spin(V)\overset{\rho} \to SO(V)\to 0.$$
 We will need the Lie group $Spin^c(V)$ which is defined by
 $$Spin^c(V):=Spin(V)\times _{\Zp 2}U(1),$$
 where $\Zp 2$ acts by $\pm (1,1)$. Note that $Spin^c(V)$ is contained in the complex Clifford algebra $\C l(V):=Cl(V)\otimes \C $ and that there is a short exact sequence
$$1\to \Z_2 \to Spin^c (V) \xrightarrow {\rho \times (\; )^2} SO(V)\times U(1)\to 1,$$
where $(\; )^2:U(1)\to U(1)$, $\lambda \mapsto \lambda ^2$ (see \cite{ABS64} for details).

If $V$ is equal to $\R ^k$ with standard scalar product and standard ONB $(e _1,\ldots,e _k)$, we write $Cl_k$, $\C l_k$, $Spin(k)$, and $Spin^c(k)$ for $Cl(V)$, $\C l(V)$, $Spin(V)$, and $Spin^c(V)$, respectively.

The complex Clifford algebra $\C l_{2l}$ has up to equivalence one irreducible representation $\triangle _{2l}$. One can show that the left regular representation of $\C l_{2l}$ decomposes into $2^l$ representations of (complex) dimension $2^l$, each equivalent to $\triangle _{2l}$. Hence, we may assume that $\C l_{2l}$ acts on $\triangle _{2l}$ by Clifford multiplication.

The restriction of $\triangle _{2l}$ to $\C l^0_{2l}$ decomposes into two inequivalent irreducible $\C l^0_{2l}$-representations $\triangle ^\pm _{2l}$ of dimension $2^{l-1}$ which are the $\pm $-eigenspaces of $(\alpha _l \cdot ):\triangle _{2l} \to \triangle _{2l}$ and are interchanged under Clifford multiplication with a nonzero element of $\C ^{2l}\subset \C l_{2l}$.

The complex Clifford algebra $\C l_{2l+1}$ has two inequivalent irreducible representations $\triangle _{2l+1}^{\pm }$, both of dimension $2^l$. The representation $\triangle _{2l+1}^+$ can be defined by taking the $\C l_{2l}$-representation $\triangle _{2l}$ and letting $e_{2l+1}$ act by Clifford multiplication with $-\alpha _l$. The representation $\triangle _{2l+1}^-$ is obtained from $\triangle _{2l+1}^+$  by composing the action of each $e_{i}$ with $-\id$. The two representations can be distinguised by the action of $(e_{1}\cdot \ldots \cdot e_{2l+1})$ via Clifford multiplication which is given by multiplication with $\pm i^{-l-1}$ on $\triangle _{2l+1}^{\pm }$. Both representations are equivalent as  $\C l^0_{2l+1}$-representation and are equivalent to $\triangle _{2l}$ over $\C l_{2l}$  (see \cite[p. 256--257]{G85} for details).
 
\subsection{$Spin^c$-manifolds}\label{section spinc manifolds} In this subsection we briefly review basic properties of $Spin^c$-structures (for more details see, for example, \cite{ABS64,
LM89}). The definition of a $Spin^c$-structure depends on topological and metric data. If one strips away the metric information, one obtains the notion of a {\em topological} $Spin^c$-structure. Conversely, a topological $Spin^c$-structure together with appropriate metric data determines a $Spin^c$-structure. Moreover, once the metric data is fixed, the equivalence classes of topological $Spin^c$-structures correspond one-to-one to the equivalence classes of $Spin^c$-structures (see, for example, \cite[\S 3.2]{DG21} for details). To keep the exposition light, we will call a topological $Spin^c$-structure also a $Spin^c$-structure when the context is clear.

Let $\xi$ be an oriented Riemannian vector bundle of rank $k$ over a topological space $X$ and $P_{SO(k)}\to X$ the associated principal bundle of oriented orthonormal frames.  A $Spin^c$-structure for $\xi$ or for $P_{SO(k)}\to X$ consists of a principal $U(1)$-bundle $P_{U(1)}\to X$ and a principal $Spin^c (k)$-bundle $P\to X$ together with a $Spin^c (k)$-equivariant bundle map $P\longrightarrow P_{SO(k)} \times P_{U(1)}$ over $X$ where $Spin^c (k)$ acts on $P_{SO(k)} \times P_{U(1)}$ via $(\rho \times (\; )^2):Spin^c (k)\to SO(k)\times U(1)$.

It is known that $\xi$ admits a 
 $Spin^c$-structure with associated principal bundle $P_{U(1)}\to X$ if and only if $w_2(\xi)$ is the mod $2$ reduction of the first Chern class of $P_{U(1)}\to X$. Moreover, for a fixed choice of $P_{U(1)}\to X$, the set of inequivalent $Spin^c$-structures, if not empty, is parametrized by $H^1(X;\Zp 2)$.
 
A $Spin^c$-structure on an oriented Riemannian manifold $X$ is by definition a $Spin^c$-structure of its tangent bundle. The first Chern class of the associated principal $U(1)$-bundle will be denoted by $c$ and will be referred to as the first Chern class of the $Spin^c$-structure.

If $X$ has non-empty boundary, then a (topological) $Spin^c$-structure on $X$ induces a (topological) $Spin^c$-structure on $\partial X$.

We note that an oriented closed manifold homotopy equivalent to $\R P^k$ always admits a topological $Spin^c$-structure. Moreover, there are precisely two such structures up to equivalence.

Next, we recall that for a Hermitian vector bundle the underlying real Riemannian vector bundle inherits a $Spin^c$-structure. The associated principal $U(1)$-bundle is associated to the determinant bundle of the Hermitian vector bundle (see \cite[\S 3]{ABS64} for details). In particular, an almost Hermitian manifold is a  $Spin^c$-manifold.

In the presence of a group action, one can consider invariant $Spin^c$-structures and invariant almost complex structures. We will only need the special case of smooth $\Zp 2$-actions.

Let $M$ be a $k$-dimensional Riemannian manifold with isometric $\Zp 2$-action and let $P\to M$ be a $\Zp 2$-invariant $Spin^c$-structure. Note that the $\Zp 2$-action on $P$ induces the action on the frame bundle $P_{SO(k)}\to M$ and induces a $\Zp 2$-action on $P_{U(1)}$.

The $\Zp 2$-action on $P$ covers the action on $P_{SO(k)} \times P_{U(1)}$. Another lift of the latter action can be obtained as follows. Let $\tau :P\to P$ denote the action of the non-trivial element of $\Zp 2$ on $P$ and let $\delta $ denote the non-trivial deck transformation of the twofold covering $P\to P_{SO(k)} \times P_{U(1)}$. Note that $\delta $ is given by the principal action of $\pm (1,-1)\in Spin^c(k)$ on $P$. One can verify directly that the map $\tilde \tau :=\tau\circ \delta$ defines a $\Zp 2$-action on $P$ which lifts the action on $P_{SO(k)} \times P_{U(1)}$. Moreover, the action given by $\tilde \tau $ is inequivalent to the one given by $\tau$. In this way one obtains a new invariant $Spin^c$-structure which is not equivalent to the original invariant structure.

If the $\Zp 2$-action on $M$ is free, then an invariant $Spin^c$-structure induces a 
$Spin^c$-structure on the quotient manifold $M/\Zp 2$. Moreover, one has a one-to-one correspondence between equivalence classes of invariant $Spin^c$-structures on $M$ and equivalence classes of $Spin^c$-structures on $M/\Zp 2$. The same holds true if one considers topological $Spin^c$-structures.

For example, the two inequivalent topological $Spin^c$-structures on $\R P^{2l+1}$ (for the standard orientation) correspond to two inequivalent invariant topological $Spin^c$-structures on the sphere $S^{2l+1}$ (for the standard orientation) where $\Zp 2$ acts by $\pm \mathrm{id}$ on $S^{2l+1}$. Note that the standard complex structure on $D^{2l+2}\subset  \C ^{l+1}$ is invariant \wrt \ this action and induces one of the two invariant topological $Spin^c$-structures on $D^{2l+2}$. Composing the action with the deck transformation $\delta$ as above, one obtains the other one. The induced structures on the boundary are precisely the two invariant inequivalent topological $Spin^c$-structures on $S^{2l+1}$. In particular, we have the following lemma which we point out for further reference.

\begin{lemma}\label{lemma spinc on disk}
An invariant topological $Spin^c$-structure on $S^{2l+1}$ is induced by an invariant topological $Spin^c$-structure on $D^{2l+2}$, which is unique up to equivalence.\qed
\end{lemma}

The lemma will be used in Section \ref{section 4k+1} to extend a given invariant $Spin^c$-structure over a plumbing construction.

\subsection{Dirac operators} 
We recall the Dirac operator of a $Spin^c$-manifold (see for example \cite[D.9]{LM89} and \cite[\S 3.2]{DG21} for details).

For $M$ a $k$-dimensional oriented Riemannian manifold with $Spin^c$-structure $P\to M$, one has  the spinor bundle $S(M)=P\times _{Spin^c(k)}\triangle \to M$ where $\triangle$ denotes the restriction of $\triangle _{2l}$ or $\triangle _{2l+1}^{\pm}$ to $Spin^c(k)$ for $k=2l$ or $k=2l+1$, respectively. The Levi-Civita connection of $M$ together with a chosen unitary connection $\nabla ^c $ on $P_{U(1)}\to M$ determine a connection $\nabla$ on $S(M)$. In this situation there is an associated $Spin^c$-Dirac operator $D_M$, acting on sections of $S(M)$,
\begin{equation}\label{spinc Dirac operator}
D_M:\Gamma (S(M))\to \Gamma (S(M)\otimes T^*M)\to  \Gamma (S(M)\otimes TM)\to \Gamma (S(M)),\end{equation}
where the first map is the connection $\nabla$, the second map uses the isomorphism given by the metric $g$, and the last map is induced from Clifford multiplication.

If $k$ is even, the decomposition $\triangle _k=\triangle _k^+\oplus \triangle _k^-$ as $\C l_k^0$-representation yields a decomposition $S(M)=S^+(M)\oplus S^-(M)$ and operators $D_M^{\pm}:\Gamma (S^{\pm}(M))\to \Gamma (S^{\mp}(M))$.

Given a Hermitian complex vector bundle $E\to M$ with a Hermitian connection $\nabla ^E$, one can consider the twisted Dirac operator ${\Diractwist M E}:\Gamma(S(M)\otimes E)\to \Gamma(S(M)\otimes E)$
essentially defined as before.

If $M$ is closed, then the operator above is elliptic of first order and its index can be computed from the symbol via the Atiyah-Singer index theorem. In cohomological terms one obtains the following formula which we state for further reference (see \cite[\S 5]{ASIII68}, \cite[Thm. D.15]{LM89}):
\begin{equation}\label{a hat equation twisted}\ind \, D_{M,E}^{+} = \langle\hat  {\cal A}(TM)e^{c/2}ch(E),[M]\rangle,\end{equation}
where $\hat  {\cal A}(TM)$ denotes the $\hat  {\cal A}$-series for the Pontryagin classes of the tangent bundle, $c$ is the first Chern class of the $Spin^c$-structure, $ch(E)$ is the Chern character of $E$, and $\langle \quad , [M]\rangle $ denotes evaluation on the fundamental cycle. If $c$ is a torsion class and the operator is untwisted (i.e., $E$ is a trivial complex vector bundle of rank one), then the formula simplifies to
\begin{equation}\label{a hat equation}\ind \, D_M^{+} = \langle\hat  {\cal A}(TM),[M]\rangle=:\hat A (M).\end{equation}

\subsection{The APS index theorem}
Next, we recall the Atiyah-Patodi-Singer index theorem for $Spin^c$-manifolds with boundary. Our exposition is based on \cite{APSI75,APSII75} and follows closely \cite[\S 3.2]{DG21}.

We first recall the eta invariant of a self-adjoint elliptic operator which measures the spectral asymmetry of the operator. The eta function is defined by the series
\begin{equation}\label{EQ: eta definition}
\eta(z):=\sum_\lambda \frac{\sign(\lambda)}{\vert\lambda \vert^{z}}, \quad \text{ for } z\in\C ,
\end{equation}
where the sum is over all the non-zero eigenvalues $\lambda$ of the operator (taken with multiplicity). Recall from \cite[\S 2]{APSI75} that this series converges absolutely for real part $\Re (z)\gg 0$ and extends in a unique way to a meromorphic function on $\C$, also denoted by $\eta$, which is holomorphic at $z=0$. The eta invariant of the operator is defined by $\eta(0)$.

Let $(W,g_W)$ be an even-dimensional compact $Spin^c$-manifold with non-empty boundary $\partial W=M$ such that the metric $g_W$ is of product type near the boundary. Let $g_M$ denote the induced metric on $M$. We fix a connection $\nabla ^c $ on the associated principal $U(1)$-bundle over $W$ which is constant in the normal direction near the boundary.

Recall that the $Spin^c$-Dirac operator (\ref{spinc Dirac operator}) on $W$ restricts to an operator 
$\Gamma (S^+(W))\to \Gamma (S^-(W))$ which will be denoted by  ${\mathfrak D}^+$. The restriction of $S^\pm (W)$ to the boundary $M$ can be identified with the spinor bundle $S(M)$. The restriction of ${\mathfrak D}^+$ can be identified with the $Spin^c$-Dirac operator $D_M:\Gamma (S(M))\to \Gamma (S(M))$ on $(M,g_M)$ \wrt \ the induced $Spin^c$-structure and connection on $M$. In a suitable collar of $M$ the operator ${\mathfrak D}^+$ takes the form ${\mathfrak D}^+=\sigma (\frac \partial {\partial u}+D_M)$ where $u$ is the normal coordinate and $\sigma=\sigma _{{\mathfrak D}^+}(du)$ is the bundle isomorphism given by the symbol $\sigma _{{\mathfrak D}^+}$ of ${\mathfrak D}^+$.

After imposing the Atiyah-Patodi-Singer boundary condition, the operator ${\mathfrak D}^+$ has finite dimensional kernel and will be denoted by $D_W^+$. Similarly, its formal adjoint, subject to the adjoint APS boundary condition, has finite dimensional kernel and will be denoted by $(D_W^+)^*$.

The index of $D_W^+$ is defined as $\ind \, D_W^+:= \dim\ker \, D_W^+ - \dim\ker \,  (D_W^+)^* \in \Z $. By the APS index theorem (see \cite{APSI75,APSII75}):
\begin{equation}\label{EQ: APS}
\ind \, D_W^+ = \left ( \int _W \hat{{\cal A}}(W,g_W)e^{\frac{1}{2}c} \right ) - \frac{h(D_M,g_M) + \eta(D_M,g_M)}{2},
\end{equation}
where $\hat{{\cal A}}(W,g_W)$ is the $\hat {\cal A}$-series evaluated on the Pontryagin forms $p_i(W,g_W)$, $c$ is the first Chern form of
$\nabla ^c $, $h(D_M,g_M)$ denotes the dimension of the kernel of
$D_M$, and $\eta(D_M,g_M)$ is the eta invariant of the self-adjoint operator $D_M$. Note that $\eta (M,g_M):=\frac{h(D_M,g_M) + \eta(D_M,g_M)}{2}$ only depends on data of the boundary $M$, namely the metric, the connection, and the  $Spin^c$-structure of $M$.

For $E\to W$ a Hermitian complex vector bundle with a Hermitian connection $\nabla ^E$ which is constant in the normal direction near the boundary, the operator can be twisted by $E$. The twisted operator ${\DiractwistAPS W E}$ has finite index and the APS index formula \eqref{EQ: APS} takes the form
\begin{equation}\label{EQ: APS twisted}
\ind \, {\DiractwistAPS W E}  = \left ( \int _W \hat{{\cal A}}(W,g_W) e^{\frac{1}{2}c} ch(E,\nabla ^E) \right ) - \frac{h_{E}(D_M,g_M) + \eta_E(D_M,g_M)}{2},
\end{equation}
where $ch(E,\nabla ^E)$ is the Chern character for the Chern forms of $(E,\nabla ^E)$, $h_{E}(D_M,g_M)$ is the dimension of the kernel of the twisted $Spin^c$-Dirac operator ${\Diractwist M {E\vert_M}}$ on $(M,g_M)$, and $\eta_E(D_M,g_M)$ is the eta invariant of ${\Diractwist M {E\vert_M}}$. Let $\eta _E(M,g_M):=\frac{h_{E}(D_M,g_M) + \eta_E(D_M,g_M)}{2}$. Formula (\ref{EQ: APS twisted}) may be extended to virtual bundles. Note that $\eta _E(M,g_M)=-\ind \, {\DiractwistAPS W E} \in \Z $ if the term involving the integration vanishes, e.g., if $E=E_1-E_2$ where $E_1$ and $E_2$ are flat bundles of equal rank.

\subsection{Eta invariants on connected components of  $\mathcal{R}_{scal>0}$}\label{eta section}

In the presense of positive scalar curvature, certain terms in formulas (\ref{a hat equation twisted})--(\ref{a hat equation}) and  (\ref{EQ: APS})--(\ref{EQ: APS twisted}) may vanish due to the Schr\"odinger-Lichnerowicz argument. More precisely, one has (see \cite{L63}, \cite{ASIII68}, \cite[\S 3]{APSII75}, \cite[p. 398, II.8.17 and Appendix D]{LM89}, \cite[\S 3.2]{DG21}).

\begin{theorem}\label{Dirac injective}
\begin{enumerate}
\item Let $(M,g_M)$ be a closed $Spin^c$-manifold and $\nabla ^c $ a connection on the associated principal $U(1)$-bundle. If $(M,g_M)$ has positive scalar curvature and $\nabla ^c $ is flat,
 then the $Spin^c$-Dirac operator $D_M$ is injective. In particular, $\hat A (M)=0$. Corresponding statements hold true if the operator is twisted with a flat bundle $E$.
\item Let $(W,g_W)$ be a compact even-dimensional $Spin^c$-manifold with boundary $M$. Suppose $g_W$ is of product type near the boundary and $\nabla ^c $ is a connection of the associated principal $U(1)$-bundle which is constant in the normal direction near the boundary. If $(W,g_W)$ has positive scalar curvature and $\nabla ^c $ is flat,
 then the index of the Dirac-operator $D_W^+$ and $h(D_M,g_M)$ both vanish.  Corresponding statements hold true if the operator is twisted with a flat bundle $E$.\qed
\end{enumerate}
\end{theorem}

As pointed out in \cite[p. 417]{APSII75}, Theorem \ref{Dirac injective} can be used to define invariants which are constant on connected components of $\mathcal{R}_{scal>0}(M)$.

Let $M$ be a closed odd-dimensional manifold. Suppose $M$ has a $Spin^c$-structure with first Chern class a torsion class. We fix a flat connection on the associated principal $U(1)$-bundle. Let $E$ be a zero-dimensional virtual bundle over $M$ with flat connection.

\begin{corollary}\label{eta sspinc constant} Let $g_0$ and $g_1$ be two metrics of positive scalar curvature on $M$. Let $\eta _0$ (resp. $\eta _1$) denote the eta invariant of the Dirac operator $\Gamma (S(M)\otimes E)\to \Gamma (S(M)\otimes E)$ defined by the connections above and the metric $g_0$ (resp. $g_1$). 

If the metrics $g_0$ and $g_1$ are in the same connected component of $\mathcal{R}_{scal>0}(M)$, then $\eta _0=\eta _1$.
\end{corollary}

\noindent
\begin{proof}  Let $g_0$ and $g_1$ be $scal>0$ metrics on $M$ which are in the same connected component. Then there is a smooth family $g_t$, $t\in [0,1]$, of $scal>0$ metrics on $M$ connecting $g_0$ and $g_1$ which is constant near the boundary of $[0,1]$. By stretching the interval sufficiently, one obtains a $scal>0$ metric $g_W$ on $W:=M\times [0,L]$, $L\gg 0$, which is of product type near the boundary and restricts to $g_0$ and $g_1$ at the boundary (see \cite[Lemma 3]{GL80II}).

We pull back the principal $U(1)$-bundle and its flat connection to $W$ and consider the $Spin^c$-structure on $W$ defined by the $Spin^c$-structure on $M$ the metric $g_W$ and the principal $U(1)$-bundle. We pull back the virtual bundle $E$ and its flat connection to $W$ and consider the associated twisted Dirac operator $\DiractwistAPS W E$ on $W$. Since $g_W$ has positive scalar curvature, it follows from Theorem \ref{Dirac injective}.2, the assumption on $E$, and (\ref{EQ: APS twisted}) that
$$ 0=\ind \, {\DiractwistAPS W E}  = \left ( \int _W \hat{{\cal A}}(W,g_W)  e^{\frac{1}{2}c} ch(E,\nabla ^E) \right ) - \frac{h_{E}(D_{\partial W},g_{\partial W}) + \eta_E(D_{\partial W},g_{\partial W})}{2}=-\frac {\eta_E(D_{\partial W},g_{\partial W})}{2}.$$
Since $\eta_E(D_{\partial W},g_{\partial W})=\eta _1-\eta _0$, the corollary follows.
\end{proof}

\subsection{Fixed point formula} For the computation of eta invariants we will use Donnelly's fixed point formula \cite{D78}.
We first recall a basic relationship between eta invariants of a twisted operator and equivariant untwisted eta invariants. We restrict to the case of interest (for the general setting see for example \cite[\S 3.3]{DG21} and references therein).

Let $M\to \overline M$ be a twofold covering map of connected closed manifolds and $\alpha :\Zp 2\to U(m)$ a unitary representation. We denote the associated flat vector bundle $M \times _{\Zp 2} \C ^m\to \overline M$ by $\alpha$ as well.

Let ${\cal D}_{\overline M}$ be a self-adjoint elliptic operator on $\overline M$ and ${\cal D}_M$ a $\Zp 2$-equivariant self-adjoint elliptic operator on $M$ covering ${\cal D}_{\overline M}$ (e.g., ${\cal D}_{\overline M}$ is a $Spin^c$-Dirac operator on $\overline M$ and ${\cal D}_M$ is its equivariant lift to $M$).

Note that each eigenspace of the equivariant operator ${\cal D}_M$ is a finite-dimensional $\Zp 2$-representation. The eta invariant of ${\cal D}_M$ refines to an equivariant eta invariant $\eta _{{\cal D}_M}:\Zp 2 \to \R$ (see \cite[\S 2]{APS73}, \cite[(2.13) on p. 412]{APSII75}, \cite{D78} for details).

Let ${\cal D}_{\overline M,\alpha}$ denote the operator ${\cal D}_{\overline M}$ twisted with $\alpha $ and $\eta _\alpha $ the eta invariant of ${\cal D}_{\overline M,\alpha}$. In this situation one has the following formula (see \cite[Thm. 3.4]{D78}, \cite[p. 390]{BGS97}):
\begin{equation}\label{EQ: Donnelly covering}
\eta _\alpha = \frac 1 2 \left ( \eta _{{\cal D}_M} + \eta _{{\cal D}_M}(\tau)\cdot \chi _\alpha (\tau)\right ),
\end{equation}
where $\chi _\alpha:\Zp 2\to\C$ denotes the character of the representation $\alpha$ and $\tau$ is the non-trivial element of $\Zp 2$.
Formula (\ref{EQ: Donnelly covering}) extends to virtual representations.

For $\alpha $ the non-trivial one-dimensional representation and  $\tilde \alpha :=\alpha -\mathbb {1}$ ($\mathbb {1}$ denotes the trivial complex one-dimensional representation), one obtains 
\begin{equation}\label{EQ: twisted eta}
\eta _{\tilde \alpha }= -\eta _{{\cal D}_M}(\tau).\end{equation}
 The equivariant eta invariant can be computed in situations where $M$ bounds a suitable manifold $W$.

Let $W$ be a compact manifold with boundary $M$. Suppose the $\Zp 2$-action extends smoothly to $W$. Let ${\cal D}_W$ be an equivariant elliptic operator on $W$. Suppose that on a suitable collar of $M$ the operator ${\cal D}_W$ is given by $\sigma (\frac \partial {\partial u}+{\cal D}_M)$ where $\sigma=\sigma _{{\cal D}_W}(du)$. Then Donnelly's equivariant extension of the APS index theorem takes the following form.

\begin{theorem}\label{Donnelly}
$\ind ({\cal D}_W)(\tau)+\frac 1 2 (h_\tau +\eta _{{\cal D}_M}(\tau))=\sum _{N\subset W^{\tau}}a_N$.\qed
\end{theorem}
Here,  $\ind ({\cal D}_W)(\tau)$ and $h_\tau$ denote the index of the equivariant operator ${\cal D}_W$ and the $\Zp 2$-representation $\ker {\cal D}_M$ evaluated at $\tau$, respectively, and $a_N$ is the local contribution in the Lefschetz-Fixed-Point formula of the equivariant symbol of ${\cal D}_W$ at the component $N\subset W^{\tau}$.

Note that the integral term in the non-equivariant APS index theorem depends on the metric whereas the right hand side of the formula in Theorem \ref{Donnelly} does not.

As an illustration, we will apply the formula for a twisted $Spin^c$-Dirac operator on real projective space  (see
 \cite[Prop. 3.5]{DG21} for details).

Let $W=D^{2n}$, $M=S^{2n-1}=\partial W$, and $M\to \overline M:=\R P^{2n-1}$ the universal covering map. The group $\Zp 2$ acts by $\pm \id$ on $W$ with one fixed point $pt$ and freely on $M$ with quotient $\overline M$.
We equip $W$, $M$, and $\overline M$ with the topological $Spin^c$-structures induced by the standard complex structure on $\C ^n$.

\begin{examples}\label{eta examples}
Let $g_W$ be an invariant metric on $W=D^{2n}$ which is of product type near the boundary and let $g_{\overline M}$ be the induced metric on $\overline M=\R P^{2n-1}$. Let ${\cal D}_W$ be the equivariant $Spin^c$-Dirac operator on $W$ defined by the complex structure on $\C ^n$, the metric $g_W$, and a fixed invariant flat unitary connection on the determinant line bundle. Then the local contribution $a_{pt}$ at the fixed point in Theorem \ref{Donnelly} is equal to $2^{-n}$ (see for example \cite{DG21}) .

Let $\eta _{\tilde \alpha}({\cal D}_{\overline M},g_{\overline M})$ denote the eta invariant of the $Spin^c$-Dirac operator ${\cal D}_{\overline M}$ on $\overline M$ twisted with the virtual representation $\tilde \alpha :=\alpha -\mathbb {1}$ where $\alpha :\Zp 2\to \C$ is the non-trivial representation. Then $\eta _{\tilde \alpha}(\overline M,g_{\overline M})=\frac{h_{\tilde \alpha}({\cal D}_{\overline M},g_{\overline M}) + \eta_{\tilde \alpha}({\cal D}_{\overline M},g_{\overline M})}{2}\equiv -2^{-n}\bmod \Z$.
If $g_W$ has positive scalar curvature, then $\eta _{\tilde \alpha}({\cal D}_{\overline M},g_{\overline M})= -2^{-n+1}$.
\end{examples}

\section{Proof of Theorems A and B and generalizations}\label{section 4k+3}
In this section we exhibit a large class of manifolds of dimension $4k+3\geq 7$ that carry infinitely many different geometries of positive Ricci curvature. These manifolds arise as the boundary of certain spin manifolds with nice metric coordinate function (see Theorem \ref{AB*}). Special cases of this result are Theorem A and Theorem B stated in the introduction. The proof will involve the topological and geometric properties of plumbings discussed in Section \ref{section topological plumbing} and Section \ref{section geometric plumbing} as well as the basic index difference which we will recall below.

The construction of different geometries of positive scalar curvature via plumbing goes back to Gromov, Lawson, and Carr. Their method applies to spin manifolds $M$ of dimension $4k+3\geq 7$ which admit a $scal>0$ metric (see \cite[IV, \S 7]{LM89} and references therein for details). Given two $scal >0$ metrics $g_0,g_1$ on $M$, Gromov and Lawson define an invariant of the pair $(g_0,g_1)$, the relativ index $i(g_0,g_1)\in \Z $, which vanishes if $g_0$ and $g_1$ can be connected by a path of $scal>0$ metrics.  We will only need this invariant in the following elementary situation:

Let $g_0$ and $g_1$ be $scal >0$ metrics on a closed $(4k+3)$-dimensional spin manifold $M$. Suppose for $i=0,1$ there exists a compact spin manifold $W_i$ with $\partial W_i=M$ and suppose $g_i$ extends to a $scal>0$ metric on $W_i$ which is of product type near the boundary. Then $i(g_0,g_1)=\hat A(W_1\cup _M(-W_2))$ and will be called the {\em basic index difference}.

Building on the work of Gromov and Lawson, the basic index difference has been used to show that for every closed spin manifold $M$ of positive scalar curvature and of dimension $4k+3\geq 7$ the space of $scal >0$ metrics has infinitely many components (see \cite[IV, \S 7]{LM89} for details). If, in addition, $p(M)=1\in H^*(M;\Q)$, i.e., all Pontryagin classes of $M$ are torsion classes, then the argument can be extended to show that $M$ carries infinitely many different geometries of positive scalar curvature, a result which was first shown by Kreck and Stolz using their $s$-invariant (see \cite[Cor. 2.15]{KS93}).

For spin manifolds of odd dimension $\geq 5$ with non-trivial fundamental group but unrestricted Pontryagin classes, many examples carrying infinitely many different geometries of positive scalar curvature are known (see the work of Botvinnik and Gilkey \cite{BG95} for first results in this context).

In \cite{W11} Wraith showed that every homotopy sphere $\Sigma$ of dimension $4k+3\geq 7$ which bounds a parallelizable manifold admits infinitely many different geometries of positive Ricci curvature. Wraith constructs $scal>0$ metrics on plumbings with boundary $\Sigma$ which restrict to $Ric>0$ metrics on the boundary (up to a deformation of the boundary metric within positive scalar curvature). He then uses Kreck-Stolz invariants to show that the $Ric>0$ metrics yield infinitely many different geometries of positive Ricci curvature on $\Sigma$ (see \cite{W11} for details).

Recently, Domange \cite{D24} has used the basic index difference to show that the conclusion also holds true if one considers other plumbings which are related to the Milnor pairing. We will briefly describe his work since it will be relevant for the proof of Theorem \ref{AB*}.

We first recall the Milnor pairing $\beta $ (see \cite{M59,M60} for details). For $\theta _1:S^{q-1}\to SO(p-1)$, let $E_1\to S^q$ be the disk bundle with characteristic map $\phi _1:=i\circ \theta _1:S^{q-1}\to SO(p-1)\hookrightarrow SO(p)$, where $i$ denotes the standard inclusion $SO(n)\hookrightarrow SO(n+1)$. The disk bundle has rank $p$ and a nowhere vanishing section. Its isomorphism type is determined by the homotopy class of $\theta _1$. Similarly, for $\theta _2:S^{p-1}\to SO(q-1)$, let $E_2\to S^p$ be the rank $q$ disk bundle with nowhere vanishing section defined by $\phi _2:=i\circ \theta _2:S^{p-1}\to SO(q-1)\hookrightarrow SO(q)$.

Milnor showed that the plumbing $W(\theta_1,\theta_2):=E_1\square E_2$ has as boundary a homotopy sphere $M(\theta_1,\theta_2):=\partial W(\theta_1,\theta_2)$  and that the map $(\theta_1,\theta_2)\mapsto M(\theta_1,\theta_2)$ defines a bilinear map
\begin{equation}\label{equation Milnor pairing 2}\beta : \pi _{q-1}(SO(p-1))\otimes \pi _{p-1}(SO(q-1))\to \Theta _{p+q-1}
\end{equation} into the group of diffeomorphism classes of oriented $(p+q-1)$-dimensional homotopy spheres.

Following Milnor, we consider the Pontryagin homomorphism $p_n:\pi _{4n-1}(SO(m))\to \Z $ which maps $[\phi ]\in \pi _{4n-1}(SO(m))$ to $p_n(\xi )\in H^{4n}(S^{4n};\Z )\cong \Z $ where $\xi$ is the vector bundle over $S^{4n}$ of rank $m$ defined by the characteristic map $\phi$. If $m< 2n$, then $p_n$ is trivial for dimensional reasons. If $m=2n$, then $p_n$ is trivial since $p_n$ is equal to the square of the Euler class of $\xi $.

Let $p=4s$, $q=4t$. In the following we will assume that  $1\leq s\leq t<2s$. In this range it is known that the Pontryagin homomorphisms $p_s:\pi _{4s-1}(SO(4t-1))\to \Z $ and $ p_t:\pi _{4t-1}(SO(4s-1))\to \Z $ are non-trivial.

Suppose $[\Sigma ]=\beta (\theta_1,\theta_2)= \beta(\theta_1^\prime,\theta _2^\prime )$, i.e., $M(\theta_1,\theta_2)=\partial W(\theta_1,\theta_2)$  and $M(\theta_1^\prime,\theta_2^\prime)=\partial W(\theta_1^\prime,\theta_2^\prime)$ can be identified with the homotopy sphere $\Sigma $ by orientation preserving diffeomorphisms.\footnote{Abusing notation we use the same notation for a map and its homotopy class.}
We fix such identifications and consider the oriented closed manifold $Y=W(\theta_1,\theta_2)\cup _{\Sigma } -W(\theta_1^\prime ,\theta_2^\prime )$ obtained by gluing $W(\theta_1,\theta_2)$ and $-W(\theta_1^\prime ,\theta_2^\prime )$ along their boundary \wrt \ the chosen identifications. Note that $Y$ is a spin manifold of dimension $4(s+t)=p+q$ with vanishing signature and that the Pontryagin classes of $Y$ are trivial except maybe in degree $4s, 4t$, and $4(s+t)$. A computation shows that the $\hat A$-genus of $Y$ takes the form
$$\hat A(Y)=c(s,t)(p_s(\theta_1)p_t(\theta_2)- p_s(\theta_1^\prime) p_t(\theta_2^\prime)),$$
where $c(s,t)$ is a non-zero constant which only depends on $s$ and $t$ (see \cite{M59,M60,D24} for details). In particular, $\hat A(Y)$ does not depend on the chosen identifications.

Since the group of oriented homotopy spheres $\Theta_{k}$ is finite for $k\neq 4$ and the Pontryagin homomorphisms $p_s$ and $p_t$ are both non-trivial, it follows that for any $[\Sigma ]$ in the image of the Milnor pairing there are $(\theta_1^i, \theta _2^i)\in \beta ^{-1}([\Sigma ])$, $i\in \N $, such that the manifolds $Y_{(i,j)}:=W(\theta_1^i,\theta_2^i)\cup _{\Sigma } -W(\theta_1^j ,\theta_2^j )$, $i,j\in \N $,  have non-vanishing $\hat A$-genus if $i\neq j$.

We are now in the position to prove the main theorem of this section.

\begin{theorem}\label{AB*} Let $V$ be a smooth spin manifold with boundary which has a nice metric coordinate function $\psi: D^p\times D^q\hookrightarrow V$ with $\{p,q\}= \{4s, 4t\}$, $1\leq s\leq t<2s$. Suppose all Pontryagin classes of $\partial V$ are torsion classes, i.e., $p(\partial V)=1\in H^*(\partial V;\Q )$. Then $\partial V$ admits infinitely many different geometries of positive Ricci curvature. 
\end{theorem}

\begin{proof} We assume that $(p,q)=(4s,4t)$ (the case $(p,q)=(4t,4s)$ is analogous). 

If $(\theta_1, \theta_2)$ is in the kernel of the Milnor pairing $\beta $
and $E_1\to S^{4s}$ and $E_2\to S^{4t}$ are the disk bundles associated to $\phi_1=i\circ \theta_1$ and $\phi_2=i\circ \theta_2$, respectively, then the  boundary of $E_1\square E_2$ is diffeomorphic to the standard sphere. In this situation we will identify $\partial (E_1\square E_2)$ with $S^{4s+4t-1}$ via a fixed orientation preserving diffeomorphism (the choice of the diffeomorphism will not affect the subsequent reasoning).

We fix elements $(\theta_1^i, \theta_2^i)$, $i\in \N$, in the kernel of
$\beta : \pi _{4s-1}(SO(4t-1))\otimes \pi _{4t-1}(SO(4s-1))\to \Theta _{4(s+t)-1}$,
such that the $\hat A$-genus of
$$Y_{(i,j)}:=W_i\cup _{S^{4s+4t-1}} -W_j$$ is non-zero if $i\neq j$.

Here, $W_i$ denotes the plumbing $W(\theta_1^i,\theta_2^i)=E_1^i\square E_2^i$ where $E_1^i\to S^{4s}$ and $E_2^i\to S^{4t}$ are the disk bundles associated to $i\circ \theta_1^i$ and $i\circ \theta_2^i$, respectively, and the boundary $\partial W_i$ has been identified with $S^{4s+4t-1}$.
Note that $W_i$, $W_j$ and $Y_{(i,j)}$, being oriented and $2$-connected, have unique spin structures.

We will use the plumbing construction considered in Lemma \ref{lemma 2 connected sum}. Let $\cal D$ be the plumbing  of trivial disk bundles,  ${\cal D}:=(S^{4t}\times D^{4s})\square (S^{4s}\times D^{4t})$, and let $\varphi _1,\varphi _2: D^{4t}\times D^{4s}\hookrightarrow S^{4t}\times D^{4s}\subset {\cal D}$ denote disjoint orientation preserving coordinate functions. Let $\psi_1:=\psi: D^{4s}\times D^{4t}\hookrightarrow V$ and let $\psi _2: D^{4s}\times D^{4t}\hookrightarrow W_i$ be defined by an orientation preserving coordinate function of $E_1^i$.

Consider the plumbing ${\cal P}_i$ obtained by plumbing $V$ and $W_i$ to $\cal D$ \wrt \ $(\psi _1,\varphi _1)$ and $(\psi _2,\varphi _2)$.
By Theorem \ref{plumbing theorem}, ${\cal P}_i$ has a metric $g_i$ of positive scalar curvature which is of product type near the boundary and which restricts to a metric $\partial g_i$ of positive Ricci curvature on the boundary. By Lemma \ref{lemma 2 connected sum}, the boundary $\partial {\cal P}_i$ is diffeomorphic to $\partial V\sharp \partial W_i=\partial V\sharp S^{4s+4t-1}$. Below we will use the basic index difference to show that the metrics $\partial g_i$, $i\in \N $, define different geometries of positive Ricci curvature on $\partial V$.

Recall from the proof of Lemma \ref{lemma 2 connected sum} that the diffeomorphism $\partial V\sharp \partial W_i\cong \partial {\cal P}_i$ is induced by an isotopy of $\partial {\cal D}$ (which is independent of $V$ and $W_i$) and that the diffeomorphism may be chosen to be the identity on $\partial V\setminus \psi _1(D^{4s}\times S^{4t-1})$ and $\partial W_i\setminus \psi _2(D^{4s}\times S^{4t-1})$.

It follows from the proof of Theorem \ref{plumbing theorem} and the above that the pullback of $\partial g_i$ under the diffeomorphism $ \partial V\sharp \partial W_i\to \partial {\cal P}_i$ is isotopic as $scal>0$ metric to a metric $\partial g_V\sharp \partial g_{W_i}$ with the following properties:

\begin{itemize}
\item $\partial g_V\sharp \partial g_{W_i}$ is the connected sum of $scal >0 $ metrics $\partial g_V$ and $\partial g_{W_i}$, where
\item $\partial g_V$ is a metric on $\partial V$, which can be chosen to be independent of $i$ as long as $i$ belongs to a finite index set, and
\item $\partial g_{W_i}$ is a metric on $\partial W_i$, which extends to a $scal >0 $ metric $g_{W_i}$ on $W_i$, and which is of product type near the boundary.
\end{itemize}

Let $Z_i:=(\partial V \times [0,1])\sharp_\partial W_i$ where the boundary connected sum is taken \wrt \  $\partial V \times \{1\}$ and $\partial W_i$. The product metric $\partial g_V +ds^2$ on $\partial V \times [0,1]$ and the metric $g_{W_i}$ on $W_i$ define a $scal >0$ metric $\hat g_i$ on $Z_i$ (unique up to isotopy) which extends $\partial g_V\sharp \partial g_{W_i}$, is of product type near the boundary, and restricts to $\partial g_V +ds^2$ on $\partial V \times [0,1/2]$.

Now suppose $\partial g_i$ and $\partial g_j$ represent elements in the same connected component of ${\cal M}_{Ric>0}(\partial V)$. Then $\partial g_V\sharp \partial g_{W_i}$ and $\partial g_V\sharp \partial g_{W_j}$ represent elements which are connected by a path in ${\cal M}_{scal>0}(\partial V)$. It follows from Ebin's slice theorem that the path can be lifted to a path of $scal>0$ metrics connecting $\partial g_V\sharp \partial g_{W_i}$ with $F^*(\partial g_V\sharp \partial g_{W_j})$ for some diffeomorphism $F:\partial V\sharp \partial W_i\to \partial V\sharp \partial W_j$ (see Section \ref{moduli}). 

For the proof of Theorem \ref{AB*}, we may assume that $F$ preserves the spin structure since the subgroup of diffeomorphisms preserving the spin structure has finite index in the full diffeomorphism group of $\partial V$  (see Section \ref{moduli}). Since $\partial W_i$ and $\partial W_j$ are spheres, it follows from the disk theorem that $F$ is isotopic to a diffeomorphism, which respects the connected sum decompositions, and is the identity on the second component.

Hence, after replacing $F$ by this diffeomorphism, we may assume that there exists a path of $scal>0$\linebreak metrics, constant near the end points, connecting $\partial g_V\sharp \partial g_{W_i}$ with $F^*(\partial g_V\sharp \partial g_{W_j})$ for a spin preserving diffeomorphism $F=F_{\partial V}\sharp F_{(i,j)}:\partial V\sharp \partial W_i\to \partial V\sharp \partial W_j$. Here, $F_{\partial V}$ is a diffeomorphism of $\partial V$ which is the identity on the disk used in the construction of the connected sum and $F_{(i,j)}$ is the identity map (recall that $\partial W_i=S^{4s+4t-1}=\partial W_j$).
We now adapt the reasoning given in \cite[Thm. 7.7]{LM89}. Let 
$$Z_{(i,j)}:=Z_i\cup  -Z_j$$
 be the closed spin manifold obtained by gluing $(\partial V \times \{0\})\subset Z_i$ to $-(\partial V \times \{0\})\subset -Z_j$ via  the identity and by gluing $(\partial V \times \{1\})\sharp \partial W_i$ to $-(\partial V \times \{1\})\sharp \partial W_j$ via $F=F_{\partial V}\sharp F_{(i,j)}$.
 
Note that $Z_{(i,j)}$ is diffeomorphic to the connected sum of the mapping torus $M_{F_{\partial V}}$ of $F_{\partial V}:\partial V\to \partial V$ and the closed manifold $W_i\cup _{F_{(i,j)}} -W_j$. Hence, $\hat A(Z_{(i,j)})=\hat A (M_{F_{\partial V}})+\hat A(W_i\cup _{F_{(i,j)}} -W_j)$.

\bigskip
\noindent
\underline {Claim:} $\hat A(Z_{(i,j)})=0$.
\bigskip

\noindent
\underline {Proof:} Consider a path of $scal>0$ metrics, constant near the end points, connecting $\partial g_V\sharp \partial g_{W_i}$ with $F^*(\partial g_V\sharp \partial g_{W_j})$. The path defines a metric on the cylinder $(\partial V\sharp \partial W_i)\times [0,L]$ which has positive scalar curvature after stretching the interval $[0,L]$ (see \cite[Lemma 3]{GL80II}). This $scal >0$ metric together with the metric $\hat g_i$ on $Z_i=(\partial V \times [0,1])\sharp_\partial W_i)$ and the metric $\hat g_j$ on $Z_j=(\partial V \times [0,1])\sharp_\partial W_j$ define a $scal>0$ metric on $Z_{(i,j)}$. Since $Z_{(i,j)}$ is spin, it follows that $\hat A(Z_{(i,j)})=0$ (see Theorem \ref{Dirac injective}.1). \hfill $\Diamond$

\bigskip
\noindent
\underline {Claim:} $\hat A(M_{F_{\partial V}})=0$.
\bigskip

\noindent
\underline {Proof:} This follows from $p(\partial V)=1\in H^*(\partial V;\Q )$ by adapting the reasoning in \cite[p. 968]{M59}. We give a proof using differential forms. Since the real Pontryagin classes of $\partial V$ vanish, we can represent the real Pontryagin classes of $M_{F_{\partial V}}=(\partial V\times [0,1])/((x,0)\sim (F_{\partial V}(x),1))$ of degree $<4(s+t)$ by closed forms which vanish near the boundary of $\partial V\times [0,1]$. Hence, the diffeomorphism $F_{\partial V}$ is irrelevant for the computation of the mixed Pontryagin numbers of $M_{F_{\partial V}}$.

It follows that the mixed Pontryagin numbers of $M_{F_{\partial V}}$ and of the trivial mapping torus $\partial V\times S^1$ coincide (in fact, they all vanish). The same is true for the top Pontryagin number (the one which corresponds to the Pontryagin class of $M_{F_{\partial V}}$ of degree $4(s+t)$) since the signature is a cut and paste invariant. Hence, all Pontryagin numbers coincide and $\hat A(M_{F_{\partial V}})=\hat A(\partial V\times S^1)=0$. \hfill $\Diamond$

\bigskip
\noindent 
By the claims above, $\hat A(Y_{(i,j)})=\hat A(W_i\cup _{F_{(i,j)}} -W_j)=\hat A(Z_{(i,j)})-\hat A (M_{F_{\partial V}})=0$. Since $\hat A(Y_{(i,j)})$ is non-zero for $i\neq j$, it follows that $\partial g_i$ and $\partial g_j$ define different geometries of positive Ricci curvature on $\partial V$ for $i\neq j$.

\end{proof}

\begin{remark}\label{remark mapping torus}
Dropping the condition on the Pontryagin classes, the proof of Theorem \ref{AB*} still shows that the space of metrics of positive Ricci curvature ${\cal R}_{Ric>0}(\partial V)$ has infinitely many connected components.
The proof also shows that $\partial V$ admits infinitely many different geometries of positive Ricci curvature if one replaces the condition $p(\partial V)=1\in H^*(\partial V;\Q )$ by the (potentially) weaker assumption that the $\hat A$-genus of the mapping torus vanishes for every spin preserving diffeomorphism $f:\partial V\to \partial V$ which maps the $scal >0$ concordance class of $\partial g_V\sharp \partial g_{W_i}$ to the one of $\partial g_V\sharp \partial g_{W_j}$.
\end{remark}

The next theorems give applications of Theorem \ref{AB*} to certain disk bundles and plumbings.

\begin{theoremA}  Let $B$ be a $p$-dimensional closed Riemannian manifold of positive Ricci curvature, $\xi $ a vector bundle of rank $q$ over $B$, and $M$ the total space of the associated sphere bundle. Suppose $M$ is spin, all Pontryagin classes of $M$ are torsion classes, and $\{p,q\}= \{4s, 4t\}$, $1\leq s\leq t<2s$.  Then $M$ admits infinitely many different geometries of positive Ricci curvature.
\end{theoremA}

\begin{proof} Let $V$ be the total space of the disk bundle associated to $\xi$. Since $q\geq 4$ and since $w_i(TV)$, $i=1,2$, restricted to $M=\partial V$ vanishes, it follows from the Gysin sequence that $V$ is spin. By Remark \ref{nice coordinate function remark}, $V$ has a nice metric coordinate function $\psi: D^p\times D^q\hookrightarrow V$. Since $p(M)=1\in H^*(M;\Q )$, the statement follows from Theorem \ref{AB*}.\end{proof}

For a trivial bundle, Theorem A gives the following corollary.
\begin{corollary}\label{Cor 1 Thm A} A product manifold $B^p\times S^{q-1}$ admits infinitely many different geometries of positive Ricci curvature if $B$ is a closed spin manifold of positive Ricci curvature, $p(B)=1\in H^*(B;\Q )$, and $\{p,q\}= \{4s, 4t\}$, $1\leq s\leq t<2s$.\qed
\end{corollary}

Recall from the proof of Theorem \ref{AB*} that geometries are distinguished by using the basic index difference to show that corresponding metrics cannot be connected by $scal>0$ metrics. As a consequence, the geometries in the corollary cannot be constructed by taking product metrics since two such metrics can always be connected by $scal>0$ metrics using suitable shrinking of the factors in the product manifold.

To illustrate Theorem A further, we point out some special $7$-dimensional cases.

\begin{corollary}\label{Cor 2 Thm A} Let $\xi$ be a vector bundle of rank $4$ over a closed $4$-dimensional base manifold $B$. Then the total space $S(\xi)$ of the associated sphere bundle admits infinitely many different geometries of positive Ricci curvature in the following cases:
\begin{enumerate}
\item $e(\xi)\neq 0$, $S(\xi)$ is spin, and $B$ admits a metric of positive Ricci curvature (e.g., $B$ is equal to $S^4$, $S^2\times S^2$, $\pm \C P^2$, or a connected sum of copies of these).
\item $B$ is a spin manifold which admits a metric of positive Ricci curvature and $\xi$ is spin with $p_1(\xi)=0$.
\end{enumerate}
\end{corollary}

\begin{proof} Applying the Gysin sequence of $S(\xi)\to B$, one sees that $H^4(S(\xi);\Q)=0$ if $e(\xi)\neq 0$. Hence, the first statement follows from Theorem A. Note that the connected sums of $S^4$, $S^2\times S^2$, and $\pm \C P^2$ admit $Ric>0$ metrics by \cite{SY93}.

For the second statement, note that if $B$ is spin and admits a $scal>0$ metric, then $\hat A(B)=0$ (see Theorem \ref{Dirac injective}) and, hence,  $p_1(B)=0$. Since the tangent bundle of $S(\xi)$ is stably isomorphic to the pullback of $TB\oplus \xi$, all assumptions of Theorem A are satisfied and the statement follows.
\end{proof}

\begin{remark}\label{remark rp4}
Theorem A extends (essentially by the same proof) to the more general situation where the spin condition on $M$ is replaced by the condition that $M$ admits a $Spin^c$-structure with first Chern class a torsion class. By applying the extended version and arguing as above one sees, for example, that 
the total space $S(T\R P ^4)$ of the sphere tangent bundle of the real projective space $\R P^4$ admits infinitely many different geometries of positive Ricci curvature (note that $\R P^4$ has a metric of positive sectional curvature and the manifold $D(T \R P^4 )$ is almost complex with vanishing first Pontryagin class).
\end{remark}

We next consider the plumbing $W=\square _{\Gamma}E_\vvv$ of disk bundles according to a tree $\Gamma$. Here, $E_\vvv$ denotes the total space of a linear disk bundle over the sphere $S^p$ (resp. $S^q$) of rank $q$ (resp. $p$) associated to a vertex $\vvv$ of $\Gamma $ and two disk bundles are plumbed together if the respective vertices are connected by an edge of $\Gamma $. Let $\xi $ be a vector bundle of rank $q$ over a $p$-dimensional closed Riemannian manifold $B$ of positive Ricci curvature such that the total space of the associated sphere bundle is spin. Let $E$ denote the total space of the disk bundle associated to $\xi$. Consider the plumbing of $E$ and $W$ \wrt \ some disk bundle of $W$.

\begin{theoremB} Let $M$ be the boundary of the plumbing  $E\square W$. If $p(M)=1\in H^*(M;\Q )$ and $\{p,q\}= \{4s, 4t\}$, $1\leq s\leq t<2s$, then $M$ admits infinitely many different geometries of positive Ricci curvature.
\end{theoremB}

\begin{proof} By Remark \ref{remark plumbing nice coordinate function}, the plumbing $E\square W$  has a nice metric coordinate function. The plumbing $W$ is $3$-connected and, hence, spin. Note that $\partial E$ is assumed to be spin which is equivalent to $E$ being spin since  $q\geq 4$.  Applying Theorem \ref{AB*} to $V:=E\square W$ it follows that $M=\partial V= \partial (E\square W)$ admits infinitely many different geometries of positive Ricci curvature.\end{proof}

The condition on $p(M)$ is fulfilled if, for example, $E\square W$ is parallelizable or $M$ is a rational homology sphere. We note that for the plumbing of two disk bundles Theorem B gives back Domange's result \cite{D24}. 

\begin{remark}\label{remark core infty}
\begin{enumerate}
\item Theorem B also holds true if the base manifolds of the plumbing $W$ are core manifolds, $W$ is spin, $\{p,q\}= \{4s, 4t\}$, $1\leq s\leq t<2s$, and $p(M)=1\in H^*(M;\Q )$. 
\item  If the base manifold $B$ of $\xi $ admits a core metric then the theorem also follows from \cite[Theorem D]{B20} and \cite[Theorem B]{R23}.
\end{enumerate} 
\end{remark}

\section{Proof of Theorems C and D and generalizations}\label{section 4k+1}
In \cite{DG21} it is shown that in each dimension $4k+1\geq 5$ there are at least $2^{2k}$ homotopy real projectice spaces of pairwise distinct oriented diffeomorphism type which carry infinitely many different geometries of positive Ricci curvature. One way to construct these geometries is via equivariant geometric plumbing of disk tangent bundles $D(TS^{2k+1})$ using work of Wraith \cite{W11}.

In this section we use a similar but more general construction to exhibit large classes of manifolds in dimensions $4k+1\geq 5$ that carry infinitely many different geometries of positive Ricci curvature. These manifolds arise as $\Zp 2$-quotients of the boundary of manifolds with nice equivariant metric coordinate function. The main result of this section is Theorem \ref{CD*} below. Special cases are Theorems C and D from the introduction. The proofs will involve the topological, geometric, and equivariant properties of plumbings discussed in Sections 3--5. We will use eta invariants for $Spin^c$-Dirac operators (as discussed in Section \ref{index section}) to distinguish geometries of positive Ricci curvature. To do so, we need to consider manifolds with non-trivial fundamental group but do not need to impose restrictions on the Pontryagin classes. 

The results are to a certain extend similar in spirit to the ones from the previous section. However, the formulation of the statements is more involved.

The following conditions on a smooth $\Zp 2$-manifold $V$ and a neat embedding $\psi $ will be relevant.

\begin{condition}\label{conditions 4k+1}
\begin{enumerate}
\item $\Zp 2$ acts without fixed points on $M=\partial V$.
\item $\psi: D^p\times D^q \hookrightarrow V$ is a neat equivariant embedding where $\Zp 2$ acts on $D^p\times D^q$ by $\pm (\id,\id)$ and $p=q=2k+1\geq 3$.
\item $V$ has an invariant $Spin^c$-structure with flat associated principal $U(1)$-bundle.
\end{enumerate}
\end{condition}
Note that the flatness condition is equivalent to the condition that the first (integral) Chern class of the principal $U(1)$-bundle is torsion. The second condition implies that the action has at least one isolated fixed point.

\begin{theorem}\label{CD*} Let $V$ be a smooth $\Zp 2$-manifold with boundary $M$. Suppose $V$ has a nice  equivariant metric coordinate function $\psi: D^p\times D^q \hookrightarrow V$ and the three conditions given in \ref{conditions 4k+1} are satisfied. 

Then $M$ has a smooth free $\Zp 2$-action such that the quotient manifold $M/\Zp 2$ admits infinitely many different geometries of positive Ricci curvature.
\end{theorem}
Recall from Definition \ref{nice equivariant coordinate function} that $\psi$ is a nice equivariant metric coordinate function if there exists $R,N>0$ and a positive constant $\kappa$ such that for every $\rho <\kappa$ there exists an invariant metric $g_V$ on $V$ such that $g_V$ has $scal>0$ on $V$,  has $Ric>0$ on $\partial V$, the mean curvature of $\partial V$ is nonnegative, and such that $\psi$ defines an isometric embedding $D^p_R(N)\times B_{\pi /2} ^q(\rho)\hookrightarrow (V,g_V)$.
 
The proof of Theorem \ref{CD*} will be given at the end of this section. The idea of the proof may be summarized as follows: We will consider equivariant plumbings of $V$ with suitable plumbings $W$. The plumbings come with invariant $Spin^c$-structures which extend (up to a deck transformation) the structure on $V$. We show that there are infinitely many inequivalent such plumbings $W$ for which $\partial (W\square V)$ is diffeomorphic to $M=\partial V$. We want to arrange that $M$ and $\partial (W\square V)$ are also equivariantly diffeomorphic  for some fixed free $\Zp 2$-action on $M$. To see that this is possible, we will use Haefliger's isotopy theorem \cite{H61} which we state now for later reference.
\begin{theorem}\label{theorem haefliger} Let $X^n$, $Y^m$ be smooth manifolds, $X$ closed, connected, and let $f:X\to Y$ be a continuous map such that $f_*:\pi _i(X)\to \pi _i(Y)$ is an isomorphism for $i\leq r$ and surjective for $i=r+1$. Then any two embeddings $X\hookrightarrow Y$ which are homotopic to $f$ are ambient isotopic if $m>2n-r$ and $2m > 3(n+1)$.\qed
\end{theorem}
Using Haefliger's isotopy theorem, we show that there are infinitely many inequivalent choices of $W$ such that $M$ and $\partial (W\square V)$ are equivariantly diffeomorphic for some fixed free $\Zp 2$-action on $M$. By Proposition \ref{equivariant plumbing theorem}, $\partial (W\square V)$ carries an invariant $Ric>0$ metric which in turn yields an invariant $Ric>0$ metric on $M$. Computation of suitable eta invariants then shows that these $Ric>0$ metrics on $M$ define infinitely many different geometries of positive Ricci curvature on the quotient manifold $M/\Zp 2$.

We remark that the present proof of Theorem \ref{CD*} only shows the conclusion for {\em some} free $\Zp 2$-action on $M$. We expect that the proof can be modified to show that the conclusion holds true for the $\Zp 2$-action induced by the {\em given} action on $V$.

Before we give the details, we point out some applications. In the first application, we consider a $\Zp 2$-equivariant vector bundle $\xi$ of rank $2k+1\geq 3$ over a closed connected $\Zp 2$-manifold $B$ of dimension $2k+1$.

\begin{theorem}\label{special CD* 1} Suppose $B$ has an invariant $Ric >0$ metric, the total space $V$ of the disk bundle of $\xi $ has an invariant $Spin^c$-structure with flat associated principal $U(1)$-bundle, and the $\Zp 2$-action on $V$ has an isolated fixed point and no fixed points on the boundary $M:=\partial V$.

Then $M$ has a smooth free $\Zp 2$-action such that the quotient manifold $M/\Zp 2$ admits infinitely many different geometries of positive Ricci curvature.
\end{theorem}

\begin{proof} We first note that the isolated fixed point is in the zero section $B\subset V$ and that we can choose an equivariant coordinate function $\psi $ of the disk bundle at the fixed point with the properties stated in Condition \ref{conditions 4k+1}.2.

We next note that by Remark \ref{nice equivariant coordinate function remark} $\psi$ is a nice equivariant metric coordinate function. Since $V$ satisfies Conditions \ref{conditions 4k+1}.1 and \ref{conditions 4k+1}.3, the theorem follows from Theorem \ref{CD*}.
\end{proof}

As special cases, we obtain the following two results which have been stated in the introduction.
\begin{theoremC}\label{theorem C} Let $B$ be a closed connected manifold of dimension $2k+1\geq 3$ with a $\Zp 2$-action with isolated fixed points and an invariant $Ric >0$ metric.
Then the total space $S(TB)$ of the sphere tangent bundle has a smooth free $\Zp 2$-action such that the quotient manifold admits infinitely many different geometries of positive Ricci curvature.
\end{theoremC}

\begin{proof} Let $\pi :V\to B$ denote the disk tangent bundle together with the $\Zp 2$-action induced by differentials. Note that $\Zp 2$ acts on $V$ with isolated fixed point and the action has no fixed points on the boundary $S(TB)=\partial V$. We can identify equivariantly the tangent bundle $TV$ with $\pi ^*(TB\otimes \C)$ using the Levi-Civita connection of $B$. Thus, $V$ has an invariant almost complex structure for which the first Chern class is a torsion class. In particular, the principal $U(1)$-bundle of the induced invariant $Spin^c$-structure on $V$ is flat. Now the statement follows from Theorem \ref{special CD* 1}.
\end{proof}

\begin{theoremD}\label{theorem D}  Let $\xi $ be a $\Zp 2$-equivariant vector bundle of rank $p$ over a $p$-dimensional closed manifold $B$ and let $p=2k+1\geq 3$. Suppose $\Zp 2$ acts with isolated fixed points. Suppose $B$ is $2$-connected and admits an invariant metric of positive Ricci curvature. 
Then the total space $S(\xi )$ of the associated sphere bundle has a smooth free $\Zp 2$-action such that the quotient manifold admits infinitely many different geometries of positive Ricci curvature.
\end{theoremD}

\begin{proof} Let $V$ denote the total space of the associated disk bundle. Note that the $\Zp 2$-action on $V$ has isolated fixed point and no fixed points on the boundary.

Let $X$ be the manifold obtained from $V$ by cutting out a small invariant open disk at each fixed point. Hence, $X$ is a smooth manifold with boundary diffeomorphic to the union of $\partial V$ and a sphere $S^{4k+1}$ for each fixed point. Note that the $\Zp 2$-action restricts to a free orientation preserving action on $X$.

Since $B$ is $2$-connected, the manifold $X$ is also $2$-connected. It follows that the cohomology of the quotient $\overline X:=X/\Zp 2$ can be identified in degree $<3$ with the one of $\R P^\infty$. In particular, the coefficient homomorphism $\Zp  2=H^2(\overline X;\Z) \to H^2(\overline X;\Zp 2)=\Zp 2$ is an isomorphism. Since $\overline X$ is orientable, this implies that $\overline X$ admits a $Spin^c$-structure. Hence, $X$ admits a $\Zp 2$-invariant $Spin^c$-structure. Note that the invariant structure can be extended to each disk (see Lemma \ref{lemma spinc on disk}). Thus, $V$ admits an invariant $Spin^c$-structure and, since $V$ is $2$-connected, the associated principal $U(1)$-bundle has a flat connection. Now the statement follows from Theorem \ref{special CD* 1}.
\end{proof}

Next, assume $V=E_1\square \ldots \square E_k$ is a $\Zp 2$-equivariant plumbing of disk bundles $E_i\to S^{2k+1}$ of rank $2k+1\geq 3$ according to a straight line. We assume that $\Zp 2$ acts linearly on each disk bundle with isolated fixed points and that the plumbing is carried out \wrt \ equivariant coordinate functions at fixed points (as in Section \ref{Equivariant plumbing}, see also the paragraph before Theorem \ref{equivariant plumbing theorem}).

\begin{theorem}\label{special CD* 2} Let $V=E_1\square \ldots \square E_k$ be as above. Then $M=\partial V$ admits a smooth free $\Zp 2$-action such that the quotient manifold $M/\Zp 2$ admits infinitely many different geometries of positive Ricci curvature.
\end{theorem}

\begin{proof} We first note that Conditions \ref{conditions 4k+1}.1 and \ref{conditions 4k+1}.2 are satisfied since $\Zp 2$ acts on each disk bundle with isolated fixed points. Let $\psi: D^{2k+1}\times D^{2k+1} \hookrightarrow V$ be a neat equivariant embedding given by an equivariant coordinate function of the disk bundle $E_k$ at the isolated $\Zp 2$-fixed point which hasn't been used in the plumbing. By Remark \ref{nice equivariant coordinate function remark} and by the proof of Theorem \ref{equivariant plumbing theorem}, $\psi$ is a nice equivariant metric coordinate function.

Let $X$ be the manifold obtained from $V$ by cutting out a small invariant open disk at each fixed point. Hence, $X$ is a smooth manifold with boundary diffeomorphic to the union of $\partial V$ and a sphere $S^{4k+1}$ for each fixed point.
Next, note that $X$ is $2$-connected. Arguing as in the proof of Theorem D, we see that $X$ admits a $\Zp 2$-invariant $Spin^c$-structure which extends to $V$. Hence, $V$ also satisfies Condition \ref{conditions 4k+1}.3, i.e., $V$ admits an invariant $Spin^c$-structure and the associated principal $U(1)$-bundle has a flat connection. Thus, the statement follows from Theorem \ref{CD*}.
\end{proof}

We now come to the proof of Theorem \ref{CD*}.
In the proof we will consider plumbings of the form $W\square V$ where $W$ is given by a suitable plumbing of disk bundles. The plumbing of $W$ and $V$ will be carried out \wrt \ the nice equivariant metric coordinate function $\psi $ of $V$ and an equivariant coordinate function of a disk bundle occurring in $W$.

\begin{proof}[Proof of Theorem \ref{CD*}]

We consider a linear $\Zp 2$-action on $S^{2k+1}$ with two isolated fixed points and let $W_{8l}^{4k+2}$, $l\geq 1$, be obtained by plumbing equivariantly $8l$ copies of the disk tangent bundle $D(TS^{2k+1})$ at $\Zp 2$-fixed points according to a straight line with $8l$ vertices,
$$W_{8l}:= D(TS^{2k+1})\square D(TS^{2k+1})\square  \ldots \square D(TS^{2k+1}).$$

\noindent
\underline {Claim 1:} $\partial W_{8l}\cong S^{4k+1}$.
\bigskip

\noindent
\underline {Proof:} It is known that $\partial W_{8l}$ is (non-equivariantly) diffeomorphic to a homotopy sphere with vanishing Arf invariant and, hence, is diffeomorphic to the standard sphere (see, for example, \cite[\S 11.3]{HM68}). \hfill $\Diamond$

\bigskip
Consider the equivariant plumbing $W_{8l}\square V$. The plumbing will be carried out  \wrt \ $\psi$ and an equivariant orientation preserving coordinate function $\varphi _{8l}^-: D^{2k+1}\times D^{2k+1}\hookrightarrow W_{8l}$  of the last plumbing block $D(TS^{2k+1})\subset W_{8l}$ at the $\Zp 2$-fixed point which hasn't been used in the plumbing $W_{8l}$.

\bigskip
\noindent
\underline {Claim 2:} $\partial (W_{8l}\square V)\cong \partial V$.
\bigskip

\noindent
\underline {Proof:} It suffices to show that the embedding $\partial \varphi _{8l}^-: D^{2k+1}\times S^{2k}\hookrightarrow \partial W_{8l}\cong S^{4k+1}\sharp \partial W_{8l}$ is (non-equivariantly) isotopic to the inclusion $D^{2k+1}\times S^{2k}\hookrightarrow (D^{2k+1}\times S^{2k}\cup _I D^{2k+1}\times S^{2k})\sharp \partial W_{8l}$ (see Lemma \ref{lemma 1 connected sum}).

To see this, we first decompose the total space $\partial D(TS^{2k+1})$ of the sphere tangent bundle \wrt\ coordinate functions as $D^{2k+1}\times S^{2k}\cup_{\Phi } D^{2k+1}\times S^{2k}$. Since $S^{2k+1}$ has a nowhere vanishing vector field, the coordinate functions can be chosen such that $\Phi:S^{2k}\times S^{2k}\to S^{2k}\times S^{2k}$, $(x,y)\mapsto (x,\phi(x)(y))$,  is the identity on $S^{2k}\times \{pt\}$ for $pt:=(0,\ldots, 0,1)\in S^{2k}$.
 
Next, consider $D(TS^{2k+1})\square D(TS^{2k+1})$ and identify its boundary with $D^{2k+1}\times S^{2k}\cup_{\Phi \circ I \circ \Phi} D^{2k+1}\times S^{2k}$ (see Lemma \ref{lemma plumbing over spheres}). Note that the embedding $\iota : S^{2k} \hookrightarrow D^{2k+1}\times S^{2k}$, $x\mapsto (x, pt)$, into the first component of the union corresponds to the embedding $S^{2k} \hookrightarrow D^{2k+1}\times S^{2k}$, $x\mapsto (pt, \phi (pt)(x))$, into the second component. The latter embedding is isotopic to an embedding of a fiber in the product bundle $D^{2k+1}\times S^{2k}\to D^{2k+1}$ since $\phi (pt)\in SO(2k+1)$ and $SO(2k+1)$ is connected. Also note that the embedding $\iota: S^{2k} \hookrightarrow D^{2k+1}\times S^{2k}$ extends to $D^{2k+1}\hookrightarrow D^{2k+1}\times S^{2k}$, $x\mapsto (x, pt)$.
 
Let $\varphi _{2}^-: D^{2k+1}\times D^{2k+1}\hookrightarrow D(TS^{2k+1})\square D(TS^{2k+1})$ be an orientation preserving coordinate function of the second disk bundle. It follows from the above that
$$\partial \varphi _{2}^-: D^{2k+1}\times S^{2k}\hookrightarrow \partial\left(D(TS^{2k+1})\square D(TS^{2k+1})\right)$$ is ambiently isotopic to a tubular neighborhood of $\iota : S^{2k} \hookrightarrow D^{2k+1}\times S^{2k}$.

Let $\nu$ be the normal bundle of $D^{2k+1}\times \{pt\}\subset D^{2k+1}\times S^{2k}$ restricted to $S^{2k} \times \{pt\}$ and let $\mathbb {1}$ be the normal bundle of $S^{2k}\times \{pt\}\subset D^{2k+1}\times \{pt\}$. Then $\nu \oplus \mathbb {1}$ defines another tubular neighborhood of $S^{2k}\times \{ pt\} \subset D^{2k+1}\times S^{2k}$. By the tubular neighborhood theorem, we may assume that the two tubular neighborhoods agree up to a change of framing given by $\theta :=[S^{2k}\to SO(2k+1)]\in \pi _{2k}(SO(2k+1))$.

Finally, we consider the entire plumbing $W_{8l}$ and its boundary. Let $\varphi : S^{2k} \hookrightarrow \partial W_{8l}$ be the embedding induced by the embedding $S^{2k} \hookrightarrow D^{2k+1}\times S^{2k}$, $x\mapsto (x, pt)$, into the boundary of the first disk bundle of the plumbing. It follows from the above that $\partial \varphi _{8l}^-: D^{2k+1}\times S^{2k}\hookrightarrow \partial W_{8l}$ is ambiently isotopic to the tubular neighborhood of $\varphi $ defined by $\nu \oplus \mathbb {1}$ up to a change of framing given by $\theta ^{4l}$. Since $\pi _{2k}(SO(2k+1))$ is a torsion group with only torsion of order two \cite{K60}, we may assume that the change of framing is trivial. 

Hence, the embedding 
$\partial \varphi _{8l}^-$ is ambiently isotopic to the tubular neighborhood of $\varphi : S^{2k} \hookrightarrow \partial W_{8l}$ defined by $\nu \oplus \mathbb {1}$. The latter is ambiently isotopic  to the inclusion $D^{2k+1}\times S^{2k}\hookrightarrow (D^{2k+1}\times S^{2k}\cup _I D^{2k+1}\times S^{2k})\sharp \partial W_{8l}= S^{4k+1}\sharp \partial W_{8l}\cong  \partial W_{8l}$.
By Lemma \ref{lemma 1 connected sum}, $\partial (W_{8l}\square V)\cong \partial W_{8l}\sharp \partial V$. Since $\partial W_{8l}\cong S^{4k+1}$, the claim follows. \hfill $\Diamond$

\bigskip
\noindent
\underline {Claim 3:} $W_{8l}\square V$ admits a topological $\Zp 2$-invariant $Spin^c$-structure.
\bigskip

\noindent
\underline {Proof:} We first exhibit an invariant almost complex structure for $W_{8l}= D(TS^{2k+1})\square   \ldots \square D(TS^{2k+1})$. Note that the tangent bundle of  the total space of the disk tangent bundle $\pi : D(TS^{2k+1})\to S^{2k+1}$ can be identified equivariantly with $\pi ^*(TS^{2k+1}\otimes \C)$ using the Levi-Civita connection for the round metric. The identification yields an invariant complex structure (unique up to sign) on the tangent bundle $TD(TS^{2k+1})\to D(TS^{2k+1})$. We fix the almost complex structure, denoted by $J$, which is compatible with the orientation of $D(TS^{2k+1})$ (note that the almost complex structure $-J$ yields the opposite orientation).

We next note that the plumbing of disk tangent bundles via cross identification is compatible with the $\Zp 2$-action and that the almost complex structure on the first disk tangent bundle can be extended to $W_{8l}$ by choosing compatible invariant almost complex structures for the other disk tangent bundles (see \cite{DG21} for more details). Hence, $W_{8l}$ carries an invariant almost complex structure.

We fix the invariant topological $Spin^c$-structure on $W_{8l}$ which is induced by the almost complex structure above. Note that the restriction of the $Spin^c$-structure to the last disk tangent bundle is independent of $l$.

Recall that $V$ comes with a $\Zp 2$-invariant $Spin^c$-structure. The action may be changed by the deck transformation $\delta$ to obtain another inequivalent invariant $Spin^c$-structure  (see Section \ref{section spinc manifolds}). Recall also that $\partial (D^{2k+1}\times  D^{2k+1})=S^{4k+1}$, with a fixed orientation, has precisely two invariant $Spin^c$-structures up to equivalence.

In the following, we will assume that $\psi$ is orientation reversing (this can be arranged by a change of coordinates, if necessary). Among the two $Spin^c$-structures on $V$ above, we choose the one which, when restricted to $\psi (\partial (D^{2k+1}\times  D^{2k+1}))$, is compatible with the $Spin^c$-structure on $W_{8l}$ under the orientation preserving cross identification $\varphi _{8l}^-\circ I\circ \psi ^{-1}$ (see Lemma \ref{lemma spinc on disk}). By construction, the structures on $V$ and $W_{8l}$ are compatible and define a $\Zp 2$-invariant $Spin^c$-structure on $W_{8l}\square V$. \hfill $\Diamond$

\bigskip
\noindent
We note that the topological $\Zp 2$-invariant $Spin^c$-structure on
$W_{8l}\square V$, when restricted to $V$, is independent of $l$ and may differ from the structure given in Condition \ref{conditions 4k+1}.3 at most by the deck transformation $\delta$.
The induced structure on the boundary $\partial (W_{8l}\square V)$ descends to a topological $Spin^c$-structure on the quotient manifold $\partial (W_{8l}\square V)/\Zp 2$. 

By Theorem \ref{equivariant plumbing theorem}, $W_{8l}\square V$ has a $\Zp 2$-invariant metric of positive scalar curvature which is of product type near the boundary and which restricts to a metric $g_l$ of positive Ricci curvature on the boundary $\partial (W_{8l}\square V)$. Let $\overline g_l$ denote the induced metric on the quotient $\partial (W_{8l}\square V)/\Zp 2$.

The topological structures above define $Spin^c$-structures on $\partial (W_{8l}\square V)$ and $\partial (W_{8l}\square V)/\Zp 2$ \wrt \ $g_l$ and $\overline g_l$, respectively.
Let $\alpha _l$ denote the associated principal $U(1)$-bundle over the quotient. From Condition \ref{conditions 4k+1}.3 and the construction above follows that the first Chern class of $\alpha _l$ is torsion. We fix a flat connection on $\alpha _l$.

We next consider the $Spin^c$-Dirac operator on $(\partial (W_{8l}\square V)/\Zp 2, \overline g_l)$ twisted with the flat zero-dimensional bundle $\tilde \alpha _l:=\alpha _l- \mathbb {1}$. Let $\eta _l$ denote the eta invariant of this operator.

\bigskip
\noindent
\underline {Claim 4:} $\eta _l\neq \eta _{l^\prime}$ if $l\neq l^\prime$.

\bigskip
\noindent
\underline {Proof:} The eta invariant $\eta _l$ can be computed using Donnelly's formula (see Theorem \ref{Dirac injective}, equation (\ref{EQ: twisted eta}) and Theorem \ref{Donnelly}) which yields a sum of local contributions at the fixed point components,
$$-\frac 1 2\eta _l=\sum _{N\subset (W_{8l}\square V)^{\tau}}a_N=\sum _{N\subset W_{8l}^{\tau}}a_N + \sum _{N\subset (V\setminus \mathrm{im}\; \psi)^{\tau}}a_N,$$
where $\tau$ is the non-trivial element of $\Zp 2$.
Note that the sum $\sum _{N\subset (V\setminus  \mathrm{im}\; \psi)^{\tau}}a_N$ is independent of $W_{8l}$. In fact, it only depends on the topological invariant $Spin^c$-structure on $V$ which is independent of $l$. The fixed point manifold $W_{8l}^{\tau}$ consists of $2l+1$ isolated points and the contributions $a_N$, at the fixed points $N\subset W_{8l}^{\tau}$, are all equal and non-zero (the local contributions are equal to the ones in Examples \ref{eta examples}). Hence, $\eta _l\neq \eta _{l^\prime}$ if $l\neq l^\prime$.\hfill $\Diamond$

\bigskip
\noindent
Recall that $\Zp 2$ acts freely on $\partial W_{8l}$.
Since there are only finitely many oriented diffeomorphism types of homotopy $\R P^{4k+1}$s (see \cite{L71}), it follows that there is an infinite subset ${\cal L}\subset \N $ such that for $l\in {\cal L}$ the quotient $\partial W_{8l} /\Zp 2$ is diffeomorphic as an oriented manifold to a fixed homotopy $\R P^{4k+1}$, i.e., each $\partial W_{8l}$, $l\in {\cal L}$, is diffeomorphic to the standard sphere and $\partial W_{8l_1}$, $\partial W_{8l_2 }$, $l_1,l_2 \in {\cal L}$, are equivariantly diffeomorphic as oriented $\Zp 2$-manifolds.

\bigskip
\noindent
\underline {Claim 5:} There exists an infinite subset ${\cal L}^\prime \subset {\cal L}$ such that $\partial (W_{8l_1}\square V)$ and $\partial (W_{8l_2}\square V)$ are equivariantly diffeomorphic as oriented $\Zp 2$-manifolds for $l_1,l_2 \in {\cal L}^\prime$.

\bigskip
\noindent
\underline {Proof:} Recall that the boundary of the equivariant plumbing $W_{8l}\square V$ is given by  
$$\partial (W_{8l}
 {\square } V)=\partial W_{8l}\setminus \varphi _{8l}^- (\overset \circ {D^{2k+1}}\times S^{2k})
\bigcup_{S^{2k}\times S^{2k}} \partial V \setminus \psi (\overset \circ {D^{2k+1}}\times S^{2k}),$$
where $ \varphi _{8l}^- (x,y)$ and $\psi (y,x)$, $(x,y)\in S^{2k}\times S^{2k}$, are identified.

Consider $l_0,l_1,l_2\in {\cal L}$. Since the manifolds $\partial W_{8l_i}$ are equivariantly diffeomorphic, we can choose equivariant orientation preserving diffeomorphism $\tilde\Psi _i:\partial W_{8l_i}\to \partial W_{8l_0}$, $i=1,2$, and can identify $\partial (W_{8l_i}\square  V)$ equivariantly  with
$$\partial W_{8l_0}\setminus \tilde\Psi_i\circ \varphi _{8l_i}^- (\overset \circ {D^{2k+1}}\times S^{2k})
\bigcup_{S^{2k}\times S^{2k}} \partial V \setminus \psi (\overset \circ {D^{2k+1}}\times S^{2k})
,$$
where $\tilde\Psi_i\circ \varphi _{8l_i}^-(x,y)$ and $\psi (y,x)$, $(x,y)\in S^{2k}\times S^{2k}$, are identified.

Let $\varphi _1$ and $\varphi _2$ denote the restrictions of $\tilde\Psi_1\circ\varphi _{8l_1}^-$ and $\tilde \Psi_2\circ \varphi _{8l_2}^-$ to $\overset \circ {D^{2k+1}}\times S^{2k}$, respectively. To prove the claim, we will show that $\varphi _1$ and $\varphi _2$ are ambient isotopic as equivariant embeddings into $\partial W_{8l_0}$ for infinitely many choices of $l_1,l_2$.

The equivariant embedding $\varphi _i:\overset \circ {D^{2k+1}}\times S^{2k}\hookrightarrow   \partial W_{8l_0}$ induces an
embedding $\overline \varphi _i$ of $\R P^{2k}=\{0\}\times S^{2k}/{\pm \id}$ into  $\partial W_{8l_0}/\Zp 2=:P_l^{4k+1}$, $i=1,2$. Recall from Claim 1 that $P_l^{4k+1}$ is an oriented homotopy $\R P^{4k+1}$.

Note that the maps $\overline \varphi _1$ and $\overline \varphi _2$ are homotopic since both induce an isomorphism on fundamental groups. Note also that $\overline \varphi _i$ induces an isomorphism on homotopy groups $\pi _j$ of degree  $j<2k$, and a surjection $\pi _{2k}(\R P^{2k})\to \pi _{2k}(P_l^{4k+1})$. From Theorem \ref{theorem haefliger} it follows that the embeddings $\overline \varphi _1$ and $\overline \varphi _2$ are ambient isotopic.
Hence, $\varphi _1$ and $\varphi _2$ restricted to $\{0\}\times S^{2k}$ are ambient isotopic as equivariant embeddings into $\partial W_{8l_0}$.

Let $F_t$, $t\in [0,1]$, be a diffeotopy of equivariant diffeomorphisms of $\partial W_{8l_0}$ with $F_0=\id $ and $\varphi _1\vert_{\{0\}\times S^{2k}}=F_1\circ \varphi _2\vert_{\{0\}\times S^{2k}}$. By the equivariant tubular neighborhood theorem
 there exists an equivariant diffeotopy $G_t: \partial W_{8l_0} \to \partial W_{8l_0}$ with $G_0=\id $ such that $\varphi _1$ is equal to $G_1\circ F_1\circ \varphi _2:\overset \circ {D^{2k+1}}\times S^{2k}\hookrightarrow   \partial W_{8l_0}$ up to an equivariant framing. The homotopy classes of equivariant framings are parametrized by the set of homotopy classes of maps $\R P^ {2k}\to SO(2k+1)$ which, in turn, corresponds to the set of isomorphism classes of oriented vector bundles of rank $2k+1$ over the suspension $\Sigma \R P^ {2k}$. The latter is finite since the reduced $KO$-theory of $\Sigma \R P^ {2k}$ is finite (apply, for example, the Atiyah-Hirzebruch spectral sequence) and the Euler class of an oriented vector bundles of rank $2k+1$ is $2$-torsion.
 
Hence, there exists an infinite subset ${\cal L}^\prime \subset {\cal L}$ such that the restrictions of $\tilde\Psi_1\circ\varphi _{8l_1}^-$ and $\tilde\Psi_2\circ \varphi _{8l_2}^-$ to $\overset \circ {D^{2k+1}}\times S^{2k}$ are ambient isotopic as equivariant embeddings into $\partial W_{8l_0}$ for $l_1,l_2\in {\cal L}^\prime $. It follows that $\partial (W_{8l_1}\square  V)$ and
$\partial (W_{8l_2}\square V)$
are equivariantly diffeomorphic as oriented $\Zp 2$-manifolds for $l_1,l_2 \in {\cal L}^\prime$.  \hfill $\Diamond$

\bigskip
\noindent
From Claim 2 and Claim 5, we conclude that $M=\partial V$ has a free $\Zp 2$-action such that  for every $l\in {\cal L}^\prime$ $\overline M:=M/\Zp 2$ is orientation preserving diffeomorphic to $\partial (W_{8l}\square V)/\Zp 2$.
For each $l\in {\cal L}^\prime$, we fix such a diffeomorphism $\Psi_l:\overline M\overset \cong \to \partial (W_{8l}\square V)/\Zp 2$. 

Using the diffeomorphism $\Psi_l$, we can pull back to $\overline M$ the $Ric>0$ metric $\overline g_l$, the topological $Spin^c$-structure, and the virtual bundle $\tilde \alpha _l=\alpha _l -\mathbb {1}$. Let $\eta (\overline M)_l$ denote the eta invariant of the corresponding twisted $Spin^c$-Dirac operator for $(\overline M, \Psi_l ^*(\overline g_l))$. Then 
$\eta (\overline M)_l=\eta _l$.

Next, we choose an infinite subset ${\cal L}^{\prime \prime}\subset {\cal L}^\prime$, such that for all $l\in {\cal L}^{\prime \prime}$ the pullback bundles $\Psi_l ^*(\alpha _l)$ are isomorphic and the pullbacks of the $Spin^c$-structures on $\partial (W_{8l}\square V)/\Zp 2$ to $\overline M$  are equivalent as topological  $Spin^c$-structures.

Since $\eta (\overline M)_l\neq \eta (\overline M)_{l^\prime}$ for $l\neq l^\prime$, it follows from Corollary \ref{eta sspinc constant} that the $Ric>0$ metrics $\Psi_l ^*(\overline g_l)$ and $\Psi_{l^\prime} ^*(\overline g_{l^\prime})$ on $\overline M$ belong to different connected components of $\mathcal{R}_{Ric>0}(\overline M)$. Let ${\cal G}\subset \text{Diff}(\overline M)$ denote the subgroup of diffeomorphisms preserving the topological $Spin^c$-structure on $\overline M$. Since the eta invariant does not change under pullbacks of diffeomorphisms in  ${\cal G}$, it follows from the path lifting property for the map $\mathcal{R}_{Ric>0}(\overline M)\to \mathcal{R}_{Ric>0}(\overline M)/{\cal G}$ that the metrics $\Psi_l ^*(\overline g_l)$, $l\in {\cal L}^{\prime \prime}$, represent infinitely many different connected components in $\mathcal{R}_{Ric>0}(\overline M)/{\cal G}$. Since $\cal G$ has finite index in the full diffeomorphism group $\text{Diff}(\overline M)$, the same is true for the moduli space $\mathcal{M}_{Ric>0}(\overline M)=\mathcal{R}_{Ric>0}(\overline M)/\text{Diff}(\overline M)$ (see also Section \ref{moduli}). Hence, the metrics $\overline g_l$, $l\in {\cal L}^{\prime \prime}$,  represent infinitely many different geometries of positive Ricci curvature on $\overline M=M/\Zp 2$.
This finishes the proof of the theorem.
\end{proof}

\section{Proof of Theorem \ref{plumbing theorem}}\label{proof plumbing theorem}
Let $V$ be a smooth manifold with nice metric coordinate function $\psi: D^p\times D^{q}\hookrightarrow V$, let $W$ be a plumbing of disk bundles over spheres according to a tree, and let $W\square V$ be the plumbing \wrt \ $\psi$ and a coordinate function of a disk bundle $E\to S^q$ of $W$.

We want to show that $W\square V$ has a metric of positive scalar curvature which is of product type near the boundary and which restricts to a metric of positive Ricci curvature on the boundary $\partial (W\square V)$.

It suffices to discuss the construction of the metric for the plumbing $E\square V$ of a disk bundle $E\to S^q$ and $V$. The construction can be iterated to yield the metric on $W\square V$. More precisely, the outcome of the iteration given below will be a metric on $W\square V$ which satisfies all the properties stated in the theorem except for the product structure near the boundary. The latter property will be taken care of at the end of the proof.

We decompose the base $S^q$ of the bundle $E$ as a union
$$S^q=B^q_-\cup Z_1\cup Z_2\cup Z_3\cup B^q_+,$$ where $B^q_{-}$ is a closed ball of radius $a_1>0$, $Z_i$ is a cylinder of the form $Z_i= I_i\times S^{q-1}$, $I_i=[a_i,b_i]$, $a_1<b_1=a_2<b_2=a_3<b_3$, and $B^q_{+}=D^q$ is another closed ball.

Let $P\to S^q$ denote the principal $SO(p)$-bundle associated to $E\to S^q$. We fix a trivialization of $P$ over $S^q\setminus B^q_-$ and a principal connection on $P\to S^q$ which is trivial over $S^q\setminus B^q_-$. Hence, the twisting of the bundle is completely contained in the part over $B^q_-$. In the following, we will identify the restriction of $P$ to $S^q\setminus B^q_-$ with a principal product bundle $S^q\setminus B^q_- \times SO(p)\to S^q\setminus B^q_-$ with trivial principal connection and use corresponding identifications for the restrictions of $E$ and $\partial E$.

The plumbing $E\square V$ is obtained from the disjoint union $E\coprod V$ by first identifying $E_{\vert B^q_{+}}=B^q_{+}\times D^p = D^q\times D^p$ with $\psi(D^p\times D^q)\subset V$ and then straightening out the angles. The boundary of the plumbing is given by 
$$\partial (E\square V)=(\partial E)_{\vert B^q_-\cup Z_1\cup Z_2\cup Z_3}
\bigcup _{S^{q-1}\times S^{p-1}}
\partial V\setminus \psi(\overset \circ {D^p}\times S^{q-1}),$$
where $(\partial E)_{\vert \{b_3\}\times S^{q-1}}= S^{q-1}\times S^{p-1}$ and $\psi(S^{p-1}\times S^{q-1})\cong S^{p-1}\times S^{q-1}$ are identified via cross identification.

\medskip
\noindent
The metric on $E$ will be constructed using the following steps.

\medskip
\noindent
\subsection{The metric on the restriction of $E$ to $S^q\setminus Z_3$}

\subsubsection{Metric $g_B$ on $S^q\setminus Z_3$} We choose a metric $\tilde g_B$ on $B^q_{-}\cup Z_1\cup Z_2$ such that $(B^q_{-}\cup Z_1\cup Z_2,\tilde g_B)$ is a geodesic ball in an elliptic paraboloid (centered at the origin of $B^q_{-}$) and such that its boundary is isometric to the unit sphere $S^{q-1}(1)$. The metric is a warped product metric of the form
$$dt^2 +k^2(t)ds^2_{q-1}, \quad t\in [0,b_2],$$
with $k(b_2)=1$ and $k^\prime (b_2)>0$ small (later we will assume that $k^\prime (b_2)<1/2$). Note that $(B^q_{-}\cup Z_1\cup Z_2,\tilde g_B)$ has positive sectional curvature and its boundary has positive second fundamental form with principal curvatures going to zero for $b_2\to \infty$.  Next, we choose a finite number of disjoint points $p_i$ in the interior of $Z_1$ (the number depends on the number of plumbings which involve $E$). We then, using Proposition \ref{Proposition local form}, change the metric $\tilde g_B$ locally at the points $p_i$ to obtain a $Ric>0$ metric  $g_B$ such that its restriction to a neighborhood of each $p_i$ is isometric to a small spherical cap $D^q_{\epsilon_i} (1)$ in the unit sphere.

The equip $B^q_{+}$ with the metric (again denoted by $g_B$) of a hemisphere in $S^q(\rho)$, $\rho >0$.

\subsubsection{Metric $g_+$ on $E_{\vert B^q_+}$} On the restriction of $E$ to $B^q_+$, $E_{\vert B^q_+}=B^q_+\times D^p=D^q\times D^p$, we choose the metric $g_+$ such that $(E_{\vert B^q_+},g_+)$ is equal to $B_{\pi /2} ^q(\rho) \times D^p_R(N)$. Here, $R$, $N$, and $\rho$ are the parameters given by the nice metric coordinate function $\psi$. Recall from Definition \ref{nice coordinate function} that for every $\rho $ smaller than a constant, $V$ has a metric $g_V$ such that $B_{\pi /2} ^q(\rho) \times D^p_R(N)$ is isometric via cross identification to the image of $\psi$ in $(V,g_V)$ and such that $g_V$ has $scal >0$ on $V$, has $Ric>0$ on $\partial V$, and $\partial V$ has nonnegative mean curvature.

\subsubsection{Metric $g_{-\cup 1}$ on $E_{\vert B^q_-\cup Z_1}$} Using Vilms' theorem for the principal connection above and fiber $F=B^p_{\pi /2}(r)$, $r>0$, we obtain a metric $g_{-\cup 1}$ on $E_{\vert B^q_-\cup Z_1}$ such that $(E_{\vert B^q_-\cup Z_1},g_{-\cup 1})\to (B^q_-\cup Z_1,g_B)$ is a Riemannian submersion with totally geodesic fibers isometric to $B^p_{\pi /2}(r)$ . 

Note that $E_{\vert Z_1}$ equipped with this metric is isometric to $(Z_1,g_B)\times B^p_{\pi /2}(r)$ and has positive Ricci curvature. After shrinking the fibers sufficiently (i.e., making $r$ sufficiently small), we may assume that $E_{\vert B^q_-\cup Z_1}$ and $(\partial E)_{\vert B^q_-\cup Z_1}$ have positive Ricci curvature (see Proposition \ref{Proposition shrinking}).

By construction there is an isometric embedding $\psi_i: D^q_{\epsilon_i} (1) \times B^p_{\pi /2}(r)\hookrightarrow E$ onto a trivialization of $E$ at $p_i$. Note that $r$ can be chosen to be arbitrarily small. The embedding $\psi _i$ is a neat embedding and can be used for subsequent plumbings.

\subsubsection{Metric $g_2$ on $E_{\vert Z_2}$} Next, we consider the restriction of $E$ to $Z_2$, $E_{\vert Z_2}=Z_2\times D^p$. We equip the fiber $D^p$ of $E$ over $(t,x)\in Z_2=I_2\times S^{q-1}$ with a spherical cap metric $g(t)$ such that $(D^p, g(t))=B^p_{\epsilon (t)}(r)$. Here,
$$\epsilon (t):I_2=[a_2,b_2]\to \R _{>0}$$ is a smooth function which is constant near the endpoints and monotonously decreasing. More precisely, there are constants $\tau _1, \tau _2$ with $a_2< \tau _1 < \tau _2 <b_2$ such that $\epsilon (t)=\pi /2$ for $t\leq \tau _1$, $\epsilon (t)=\epsilon (b_2)$ for $t\geq \tau _2$, where $0<\epsilon (b_2)<\pi /2$ will be determined later (see (\ref{bc 1})),  and $\epsilon ^\prime (t)<0$ for $\tau _1 <t <\tau _2$. 

This determines a metric $g_2$ on $E_{\vert Z_2}$ which is of the form $$g_2=dt^2 +k^2(t)ds^2_{q-1}+g(t),\quad t\in [a_2,b_2].$$
Note that $((\partial E)_{\vert  Z_2},g_2)$ is isometric to $(Z_2,g_B)\times S^{p-1}(r)$ and the restriction of $(E_{\vert Z_2},g_2)$ to a suitable neighborhood of $E_{\vert \{b_2\}\times S^{q-1}}\subset E_{\vert Z_2}$ is isometric to $((\tau_2, b_2]\times S^{q-1}, g_B)\times B^p_{\epsilon (b_2)}(r)$.

The metrics $g_{-\cup 1}$ and $g_2$ define a smooth metric on $E_{\vert B^q_-\cup Z_1\cup Z_2}$ for which $E_{\vert B^q_-\cup Z_1\cup Z_2}\to B^q_-\cup Z_1\cup Z_2$ is a Riemannian submersion. The restriction to $(\partial E)_{\vert B^q_-\cup Z_1\cup Z_2}$ is a Riemannian submersion with totally geodesic fibers isometric to $S^{p-1}(r)$. After shrinking the fibers sufficiently (i.e., making $r$ sufficiently small) the metric on $(\partial E)_{\vert B^q_-\cup Z_1\cup Z_2}$ has positive Ricci curvature (see Proposition \ref{Proposition shrinking}). The fibers of $E$ over points of $Z_2$ are not all totally geodesic since they are not all isometric. Nevertheless, by shrinking the fibers sufficiently, the positive curvature of the fibers dominates the other curvatures and one obtains a metric of positive scalar curvature on $E_{\vert B^q_-\cup Z_1\cup Z_2}$ (this follows from \cite[Cor. 9.37 and \S 9.G]{B87}).

Hence, by taking $r$ sufficiently small, we may assume that $g_{-\cup 1}$ and $g_2$ have $scal>0$ on $E_{\vert B^q_-\cup Z_1\cup Z_2}$ and have $Ric>0$ on the boundary $(\partial E)_{\vert B^q_-\cup Z_1\cup Z_2}$.

\bigskip
In the next step we will define a metric on $E_{\vert Z_3}$ which restricts to a $Ric>0$ metric on $(\partial E)_{\vert Z_3}$ and which is compatible with the metrics above, after possibly rescaling. This is the crucial step in the proof, the construction is essentially due to Reiser \cite{R23,R25}.

\subsection{The metric $\bar g_{f,h}$ on the restriction of $E$ to $Z_3$} The metric on $E_{\vert Z_3}$ will be constructed from another metric by a smoothing (see \ref{subsubsection smoothing}). The latter metric  depends on two smooth functions $f,h:[a_3,b_3]\to \R _{>0}$ (to be specified later) and will be denoted by $\bar g_{f,h}$. Its restriction to the boundary 
$(\partial E)_{\vert Z_3}=[a_3,b_3]\times S^{q-1}\times S^{p-1}$ is a doubly warped product metric of the form $$g_{f,h}=dt^2+h^2(t)ds^2_{q-1} + f^2(t)ds^2_{p-1}.$$
The construction of $\bar g_{f,h}$ will be carried out in Subsections \ref{subsubsection gfh}--\ref{subsubsection functions}.

\medskip
To describe the metric $\bar g_{f,h}$, we will, following a suggestion of Reiser, first identify $[a_3,b_3]\times D^p$ with a subspace $X$ of the Riemannian cylinder $\R \times S^p (\beta N)$, where $\beta >0$ will be specified later. The metric of the round sphere $S^p (\beta N)$ will be written as a warped product metric $ds^2 + (\beta N)^2 \sin ^2(\frac s {\beta N}) ds^2_{p-1}$, $s\in [0, \beta N\pi ]$.

\subsubsection{Description of $X\subset \R \times S^p (\beta N)$} For $f:[a_3,b_3]\to \R _{>0}$ let $$\varphi (t):[a_3, b_3]\to [a_3, \tilde b_3 ],\quad t\mapsto \tilde t ,$$ be a diffeomorphism
with $\varphi (a_3)=a_3$
 and 
$$F(\tilde t):[a_3, \tilde b_3]  \to (0, \beta N \pi)$$ a smooth function such that the curve $t \to (\varphi (t), F(\varphi (t))$ in $\R ^2$ is parametrized by arc length and such that $$f(t)=\beta N\sin \left (\frac {F(\varphi (t))}{\beta N}\right)\text{ for } a_3\leq t\leq b_3.$$ Note that the two conditions determine $F$ and $\varphi $ uniquely. Now let
$$X:=\{ (\tilde t, (s, x))\in \R \times S^p \, \mid \, \tilde t\in [a_3,\tilde b_3],\, s\leq F(\tilde t), \, x\in S^{p-1}\}\subset \R \times S^p.$$ We equip $X$ with the metric $g_X$ induced by the inclusion of $X$ into the Riemannian cylinder $\R \times S^p (\beta N)$. Note that $(X,g_X)$ has $scal>0$ and nonnegative sectional curvature.

By construction, the fiber over $\tilde t=\varphi (t)$ of the projection $X\to [a_3,\tilde b_3]$ is a geodesic ball $D^p_{F(\tilde t)}(\beta N)$ with boundary isometric to a round sphere of radius $f(t)=\beta N \sin \left(\frac {F(\tilde t)} {\beta N}\right)$.

In the following we will identify $[a_3,b_3]\times D^p$ with $X$ using the diffeomorphism
$$\Phi :[a_3,b_3]\times D^p\to X, \quad (t,(s,x))\mapsto (\varphi ( t),(F(\varphi ( t))s,x)),$$ where $0< s\leq 1$, $x\in S^{p-1}$, are polar coordinates for $D^p$.

\subsubsection{Metric $\bar g_{f,h}$}\label{subsubsection gfh}  For $f,h:[a_3,b_3]\to \R _{>0}$ 
the metric $\bar g_{f,h}$ on $E_{\vert Z_3}=[a_3,b_3]\times S^{q-1}\times D^p$ is defined by requiring that

\begin{itemize}
\item $\bar g_{f,h}$ restricted to $[a_3,b_3]\times \{pt\} \times D^p$ is equal to the pullback of $g_X$ under $\Phi$ \, $\forall \, pt\in S^{q-1}$,
\item $\bar g_{f,h}$ restricted to $\{t\}\times S^{q-1}\times \{pt\}$ is equal to $h^2(t)ds^2_{q-1}$  \, $\forall \, (t,pt)\in [a_3,b_3]\times D^p$, and
\item the tangent spaces of $S^{q-1}$ and $[a_3,b_3]\times D^p$ at every point of $[a_3,b_3]\times S^{q-1}\times D^p$ are orthogonal with respect to $\bar g_{f,h}$.
\end{itemize}

Note that
$$\partial X:=\{ (\tilde t, (F(\tilde t), x))\, \mid \, \tilde t\in [a_3,\tilde b_3],\, x\in S^{p-1}\}=\{(\varphi (t), (F(\varphi( t)),x)) \, \mid \, t\in [a_3, b_3], \, x\in S^{p-1}\}
$$
equipped with the metric induced by $g_X$ is isometric to $([a_3,b_3]\times S^{p-1},dt^2+ f^2(t)ds^2_{p-1})$ via $\Phi _{\vert [a_3,b_3]\times S^{p-1}}$ and that $\bar g_{f,h}$ restricts to $g_{f,h}$ on $(\partial E)_{\vert Z_3}=[a_3,b_3]\times S^{q-1}\times S^{p-1}$. A direct computation shows that $(E_{\vert Z_3},\bar g_{f,h})$ has positive scalar curvature.

\medskip
In the following we will consider the metrics $\alpha ^2 g_{-\cup 1}$, $\alpha ^2 g_2$, $\bar g_{f,h}$, and $\beta ^2 g_+$ on the corresponding pieces of $E$  as well as $\beta ^2 g_V$ on $V$, where $\alpha ,\beta >0$ will be determined later. 

We want to construct smooth functions $f,h:[a_3,b_3]\to \R _{>0}$ such that $((\partial E)_{\vert Z_3},g_{f,h})$ has positive Ricci curvature, $\bar g_{f,h}$ agrees with $\alpha ^2 g_2$ and $\beta ^2 g_+$ at the respective boundary part of $E_{\vert Z_3}$, and Perelman's conditions on the second fundamental forms (see Theorem \ref{Perelman gluing}) are fulfilled at the boundary components of $(\partial E)_{\vert Z_3}$.

First note that the second fundamental forms of the left and right boundary components of $((\partial E)_{\vert Z_3},g_{f,h})$ are represented by

\begin{equation}\label{eq 2nd fundamental form Z3} \hat {\mathrm{I\!I}}_{a_3}=\left( \begin{smallmatrix} -\frac {h^\prime(a_3)} {h(a_3)}\mathrm{I}_{q-1} & 0\\ 0 &  -\frac {f^\prime(a_3)} {f(a_3)}\mathrm{I}_{p-1} \end{smallmatrix} \right)\quad \text{and} \quad \hat {\mathrm{I\!I}}_{b_3}=\left( \begin{smallmatrix} \frac {h^\prime(b_3)} {h(b_3)}\mathrm{I}_{q-1} & 0\\ 0 &  \frac {f^\prime(b_3)} {f(b_3)} \mathrm{I}_{{p-1}}\end{smallmatrix} \right),\end{equation}
respectively.

\subsubsection{Left boundary condition}
Recall that $(Z_2,g_B)$ is part of an elliptic paraboloid and that $E_{\vert \{b_2\}\times S^{q-1}}\subset E_{\vert Z_2}$ has a neighborhood which is isometric to $((\tau _2, b_2]\times S^{q-1}, dt^2 +k^2(t)ds^2_{q-1})\times B^p_{\epsilon (b_2)}(r)$ with $k(b_2)=1$ and $k^\prime (b_2)>0$ small (for technical reasons we will assume in the following that $k^\prime (b_2)<1/2$).

Hence, the second fundamental form  $\underline{\mathrm{I\!I}}_{b_2}$ of $\{b_2\}\times S^{q-1}(1)\subset Z_2$ is represented by $\hat {\underline{\mathrm{I\!I}}}_{b_2}=\frac {k^\prime (b_2)}{k(b_2)}\mathrm{I}_{q-1}$ and is positive definite.
It follows that the second fundamental form $\mathrm{I\!I}_{b_2}$
of $(\partial E)_{\vert \{b_2\}\times S^{q-1}}\subset (\partial E)_{\vert Z_2}$ is represented by the $(p+q-2)\times (p+q-2)$-block matrix $$\hat {\mathrm{I\!I}}_{b_2}=\left (\begin{matrix} \hat {\underline {\mathrm{I\!I}}}_{b_2} & 0\\0 & 0\end{matrix}\right)=\left (\begin{matrix} \frac {k^\prime (b_2)}{k(b_2)}\mathrm{I}_{q-1} & 0\\0 & 0\end{matrix}\right).$$ Note that if we replace the metric $g_2$ by $\alpha ^2g_2$, then $\hat {\underline{\mathrm{I\!I}}}_{b_2}$ and $\hat {\mathrm{I\!I}}_{b_2}$ have to be replaced by $\frac 1 \alpha \hat {\underline {\mathrm{I\!I}}}_{b_2}$ and $\frac 1 \alpha \hat {\mathrm{I\!I}}_{b_2}$, respectively.

In order to satisfy the left boundary conditions, we will choose $f$ and $h$ near $t=b_2=a_3$ such that $\bar g_{f,h}$ and $\alpha ^2g_2$ agree on $E_{\vert \{b_2\}\times S^{q-1}}$ and such that the condition $\frac 1 \alpha \hat{\mathrm{I\!I}}_{b_2}+\hat{\mathrm{I\!I}}_{a_3}\geq 0$ on the second fundamental forms given in Theorem \ref{Perelman gluing}  is satisfied.

The first condition means that $B^p_{\epsilon (b_2)}(\alpha r)=D^p_{F(a_3)}(\beta N)$ (i.e.,  $\epsilon (b_2)=\frac {F(a_3)} {\beta N}=\arcsin \left(\frac {f(a_3)} {\beta N}\right)$, $\alpha r=\beta N\sin \left (\frac {F(a_3)} {\beta N}\right)=f(a_3)$) and that $h(a_3)= \alpha$.

The second condition means that $h^\prime (a_3)\leq k^\prime (b_2) $ and $f^\prime (a_3)\leq 0$. 
In the following we will impose the conditions
\begin{equation}\label{bc 1}\epsilon (b_2)=\arcsin \left (\frac {f(a_3)} {\beta N}\right), \quad h(a_3)= \alpha , \quad h^\prime (a_3)\leq  \lambda, \quad f(a_3)=\alpha r, \quad \text{and } \quad f^\prime (a_3)= 0,\end{equation}
where $\lambda := k^\prime (b_2)<1/2$.

\subsubsection{Right boundary condition}
Recall that $(E_{\vert B^q_+},g_+)$ is equal to $B_{\pi /2} ^q(\rho) \times D^p_R(N)$ and isometric to the image of the nice metric coordinate function $\psi: D^p_R(N)\times B_{\pi /2} ^q(\rho)\hookrightarrow (V,g_V)$ via cross identification. The second fundamental form $\mathrm{I\!I}_{\psi}$ of $B_{\pi /2} ^q(\rho) \times \partial D^p_R(N)$ viewed as the boundary of $V\setminus \psi ({\overset \circ {D^p_R}(N)}\times B_{\pi /2} ^q(\rho))$ is represented by  $$\hat {\mathrm{I\!I}}_{\psi}=(\begin{smallmatrix}0 & 0\\0 & -\hat{\mathrm{I\!I}}_{R,N}\end{smallmatrix})$$ where $\hat {\mathrm{I\!I}}_{R,N}=\frac 1 N \cot (R/N) \mathrm{I}_{p-1}>0$ represents the second fundamental form of the boundary of $D^p_R(N)$ (\wrt \ the inward pointing normal vector).
 
We replace the metrics $g_+$ and $g_V$ by $\beta ^2g_+$ and $\beta ^2g_V$, respectively. Of course, the cross identification in the plumbing process remains an isometry for the rescaled metrics. 

In order to satisfy the right boundary condition given in Theorem \ref{Perelman gluing}, we need to choose $f$ and $h$ at $t=b_3$ such that, firstly, $\bar g_{f,h}$ and $\beta ^2g_+$ agree on $E_{\vert \{b_3\}\times S^{q-1}}$, i.e.,
$$D^p_{F(\tilde b_3)}(\beta N)\times S^{q-1}(h(b_3))=D^p_{\beta R}(\beta N)\times \partial B_{\pi /2} ^q(\beta \rho),$$ and, secondly,  the condition on the second fundamental forms is satisfied for $(\partial E)_{\vert \{b_3\}\times S^{q-1}}\subset (\partial E)_{\vert  Z_3}$ and for $\psi (\partial D^p_R(N)\times\partial B_{\pi /2} ^q(\rho))\subset \partial V\setminus \psi ({\overset \circ {D^p_R}(N)}\times \partial B_{\pi /2} ^q(\rho))$.
We will impose the following stronger conditions 
\begin{equation}\label{bc 2} h(b_3)= \beta \rho , \quad h^\prime (b_3)= 0, \quad f(b_3)=\beta  N\sin (R/N) \text{ and }f^\prime (b_3)= \cos (R/N).\end{equation}

\medskip
If $h$ and $f$ satisfy conditions (\ref{bc 1}) and (\ref{bc 2}), then the smooth metrics $\alpha ^2g_{-\cup 1}$,  $\alpha ^2g_2$, $\bar g_{f,h}$, and $\beta ^2g_V$ define a continuous metric on $V^\prime:=E \square V$.

In \cite{R23} Reiser constructs smooth functions $f,h:[a_3,b_3]\to \R _{>0}$ which satisfy conditions (\ref{bc 1}) and (\ref{bc 2}) such that, in addition, the  metric $g_{f,h}=dt^2+h^2(t)ds^2_{q-1} + f^2(t)ds^2_{p-1}$ on $(\partial E)_{\vert Z_3}=[a_3,b_3]\times S^{q-1}\times S^{p-1}$ has positive Ricci curvature. We will review this construction which will be also important for the mean curvature estimates below.

\subsubsection{Functions $f,h$}\label{subsubsection functions}
The smooth functions $f,h:[a_3,b_3]\to \R _{>0}$ will each be constructed from two pieces, defined on $[a_3,t_1]$ and  $[t_1,b_3]$, respectively ($t_1 $ will be determined later). When restricted to $[a_3,t_1]$ (resp. $[t_1,b_3]$) these pieces will be denoted by $f_l,h_l$ (resp. $f_r, h_r$).

The functions $f_l,h_l:[a_3,t_1]\to \R _{>0}$
are of the form 
$$h_l(t):=ah_0(t), \quad f_l(t):=bf_C(t), \quad a,b>0, \quad C\in (0,1),$$
with $h_0(t),f_C(t)$ defined as follows:
\begin{itemize}
\item $h_0:[a_3,\infty)\to \R_{>0}$ is determined by $h_0^\prime(t)=e^{-\frac 1 2 h_0^2(t)}$, $h_0(a_3)= \sqrt { -2 \ln (\lambda )}$, and
\item $f_C:[a_3,\infty)\to \R _{>0}$ is determined by $f^{\prime \prime}_C(t)=Ce^{-h_0^2(t)}f_C(t)$, $f_C(a_3)=1$, $f^\prime _C(a_3)=0$.
\end{itemize}

The functions $h_0,f_C$ have the following properties which we state for further reference:
$$h_0(t),h^\prime _0(t)>0, h^{\prime \prime}_0(t)=-h_0 (t)e^{-h_0^2(t)}<0, h^\prime _0(a_3)=\lambda,  h^{\prime \prime}_0(a_3)=-\lambda ^2\sqrt {-2 \ln \lambda },$$
$$f_C(t), f^\prime _C(t), f_C^{\prime \prime}(t)>0\text{ for } t>a_3, f_C^{\prime \prime}(a_3)=C\lambda ^2,$$
$$\frac {f_C^\prime(t)} {f_C(t)h_0(t)h_0^\prime(t)}\in [0,1], \lim _{t\to \infty} f_C(t) = \infty , \lim _{t\to \infty} f_C^\prime(t)  = \infty, \text{ and }\lim _{t\to \infty} f_C(t) h_0^\prime (t)= 0  .$$

The left boundary conditions (\ref{bc 1}) for $f_l$ and $h_l$ replacing $f$ and $h$, respectively, is satisfied if $a\leq 1$, $\alpha =a\sqrt { -2 \ln (\lambda )}$, $b=\alpha r$ and $\epsilon (b_2)=\arcsin \left (\frac {f(a_3)} {\beta N}\right)$.

Reiser shows that for $a,b,C$ sufficiently small the Ricci curvature of $\left ((\partial E)_{\vert Z_3}=[a_3,b_3]\times S^{q-1}\times S^{p-1},g_{f,h}\right )$ is positive for $t\leq t_b$ where $\lim _{b\to 0} t_b = \infty$ and $\lim _{b\to 0} f_l^\prime (t_b)=1$ (see \cite{R23}).

In the following, we will choose $b$ and $t_1<t_b$ such that $f_l^\prime (t_1)> \cos (R/N)$. Moreover, for $\tilde \epsilon  < \sin (R/N)$ we will choose $\rho >0$ sufficiently small such that $\frac {f_l(t_1)}{h_l(t_1)} \cdot \frac \rho N <\tilde \epsilon $ and choose $\beta >\frac {h_l(t_1)}\rho$ such that $\frac {f_l(t_1)}{\beta N}<\tilde \epsilon $. Note that  $f_l(t_1) < \beta N \sin (R/N)$ for these choices. This completes the description of the pieces $f_l,h_l:[a_3,t_1]\to \R _{>0}$.

Next, we describe how to extend these functions to $[a_3,b_3]$ for some $b_3>t_1$ using  pieces $f_r,h_r$. Let $h_r:[t_1,b_3]\to \R _{>0}$ be a smooth function, such that $h_r(t)$ extends $h_l(t)$ smoothly in $t=t_1$, such that $h_r^\prime (t)\geq 0$, $h_r^{\prime \prime}  (t)\leq 0$ for $t\in [t_1,b_3]$, and such that the metric  $dt^2+h_r^2(t)ds^2_{q-1}$ is in a neighborhood of $\{b_3\}\times S^{q-1}\subset [a_3,b_3]\times S^{q-1}$ isometric to a neighborhood of $\partial B^q_{\pi /2}(\beta \rho)\subset B^q_{\pi /2}(\beta \rho)$ (since $h_l(t_1)<\beta \rho$ the function $h_r$ exists). In particular, $ h_r (b_3)=\beta \rho,  h_r^\prime (b_3)=0, h_r^{(k)}(b_3)=(\beta \rho)^{1-k} \sin ^{(k)} (x)(\pi /2), k\geq 2$.

Let $f_r:[t_1,b_3]\to \R _{>0}$ be a smooth function, such that
$$f_r(t_1)=f_l(t_1), f_r^\prime (t_1)= f_l^\prime (t_1), $$
$$f_r^\prime (t)\in (\cos (R/N), 1), f_r^{\prime \prime}(t)\leq 0, \quad \forall \, t\in [t_1,b_3],$$
$$ \tilde f(b_3)=\beta N \sin ( R/N), \tilde f^\prime (b_3)=\cos (R/N), \text{ and } \tilde f^{(k)}(b_3)=(\beta /N)^{1-k} \cos ^{(k-1)}(R/N), k\geq 2.$$ 

It follows from a direct computation (using the formulas \cite[(3.9)-(3.11)]{R23}) that $f_r,h_r$ can be chosen such that $([t_1,b_3]\times S^{q-1}\times S^{p-1},dt^2+h_r^2(t)ds^2_{q-1} + f_r^2(t)ds^2_{p-1})$ has positive Ricci curvature.

The functions $h_l$ and $h_r$ define a smooth function $h:[a_3,b_3]\to \R _{>0}$. The functions $f_l$ and $f_r$ define a $C^1$-function on $[a_3,b_3]$ which is smooth away from $t_1$.

Applying Theorem \ref{Perelman gluing} one sees that the functions $f_l$ and $f_r$ can be smoothed at $t=t_1$ 
to obtain a smooth function $f$ on $[a_3, b_3]$, such that $(\partial E)_{\vert Z_3}=[a_3,b_3]\times S^{q-1}\times S^{p-1}$ equipped with the metric $g_{f,h}:=dt^2+h^2(t)ds^2_{q-1} + f^2(t)ds^2_{p-1}$ has positive Ricci curvature.

\subsubsection{Metric gluing of boundaries}\label{subsubsection smoothing} Note that the functions $f,h$ satisfy the right boundary condition (\ref{bc 2}). In the following we will identify  ($[a_3,b_3]\times D^{p}, dt^2 + f^2(t)ds^2_{p-1})$ via $\Phi$ isometrically with $X\subset [a_3,\tilde b_3]\times S^p(\beta N)$, will use $\Phi$ to identify $E_{\vert Z_3}=[a_3,b_3]\times S^{q-1}\times D^p$ with $X\times S^{q-1}\subset [a_3,\tilde b_3] \times S^p \times S^{q-1}$, and will denote the metric on $X\times S^{q-1}$ induced by $\bar g_{f,h}$ by $\hat g_{f,h}$.

Note that $F(\tilde b_3)=\beta R$ (since $f(b_3)=\beta N \sin ( R/N)$), that the fiber of $X\to [a_3 , \tilde b_3]$ over $\tilde b_3$ is a geodesic ball $D^p_{\beta R}(\beta N)$, and that $(E_{\vert \{b_3\}\times S^{q-1}},\bar g_{f,h})$ is isometric to $\partial B^q_{\pi /2}(\beta \rho) \times D^p_{\beta R}(\beta N)=S^{q-1}(\beta \rho) \times D^p_{\beta R}(\beta N)$.
Recall that $(E_{\vert B^q_+},\beta ^2 g_+)$ is equal to $B_{\pi /2} ^q(\beta \rho) \times D^p_{\beta R}(\beta N)$ and isometric to the image of the nice metric coordinate function $\psi: D^p_{\beta R}(\beta N)\times B_{\pi /2} ^q(\beta \rho)\hookrightarrow (V,\beta ^2g_V)$ via cross identification. We may assume that the embedding extends to an isometric embedding of $D^p_{\beta R+\delta}(\beta N)\times B_{\pi /2} ^q(\beta \rho)$ into $(V,\beta ^2 g_V)$ for some $\delta >0$.

It follows from the properties of $f$ and $h$ that
\begin{itemize}
\item $(X\times S^{q-1},\hat g_{f,h}) $ and $(V,\beta ^2 g_V)$ can be glued isometrically along $S^{q-1}(\beta \rho) \times D^p_{\beta R}(\beta N)$,
\item $(\partial X\times S^{q-1}) \cup \partial V$ is smooth, the curve $\tilde t \mapsto (\tilde t, F(\tilde t))$ provides a straightening of the angles in the plumbing construction for $E$ and $V$, and
\item the resulting metric on $(X\times S^{q-1})\cup V$ is smooth, has positive scalar curvature and restricts to a $Ric >0$ metric on $(\partial X\times S^{q-1})\cup \partial V$.
\end{itemize}
Since $(E_{\vert Z_3}, \bar g_{f,h})$ is isometric to $(X\times S^{q-1},\hat g_{f,h})$, the corresponding properties hold for $E\square V$, i.e., the function $f$ provides a smoothing of the corners via the curve above, the metric on $E\square V$ is smooth except maybe for $t=a_3$, has $scal>0$, and restricts to a $Ric>0$ metric on $\partial (E\square V)$ (away from $t=a_3$).

Next, we consider the metric $\bar g_{f,h}$ near $E_{\vert \{a_3\}\times S^{q-1}}$. Since the functions $f,h$ satisfy the left boundary conditions (\ref{bc 1}), $(E_{\vert \{a_3\}\times S^{q-1}}, \bar g_{f,h})$ and $(E_{\vert \{b_2\}\times S^{q-1}}, \alpha ^2g_2)$ are isometric (and isometric to $S^{q-1}(\alpha )\times B^p_{\epsilon (b_2)}(\alpha r)$). Hence, $E_{\vert Z_3}$ and $E_{\vert Z_2}$ can be glued isometrically along $S^{q-1}(\alpha )\times B^p_{\epsilon (b_2)}(\alpha r)$. The resulting metric on $E_{\vert Z_2\cup Z_3}$ is continuous and smooth away from $t=b_2=a_3$.

Note, that if one identifies $(\tau _2, b_2]\times B^p_{\epsilon (b_2)}(r)$ with $$Y:=\{ (\tilde t, (s, x))\in \R \times S^p \, \mid \, \tilde t\in [\tau _2 ,a_3],\, s\leq F(a_3), \, x\in S^{p-1}\}\subset \R \times S^p $$ via $\Phi$ then $Y\cup X\subset \R \times S^p$ is smooth except for boundary points with $\tilde t=a_3$.

Applying Theorem \ref{Perelman gluing}, one sees that the functions $f$ and $h$ can be smoothed at the gluing area $t=a_3$ 
to obtain smooth functions which will be also denoted by $f,h$ such that the new metric on $E_{\vert Z_2\cup Z_3}$ is smooth has positive scalar curvature and restricts to a $Ric >0$ metric on $(\partial E)_{\vert Z_2\cup Z_3}$.

This completes the construction of a metric on  $E\square V$ which will be later modified to yield a metric with all the properties stated in Theorem \ref{plumbing theorem}. To do so, we will first show that the metric above can be chosen such that the boundary of $E\square V$ has nonnegative mean curvature.

\subsection{Nonnegative mean curvature}
We claim that the construction above can be carried out such that the boundary of $E\square V$ has nonnegative mean curvature. For the boundary part belonging to $V$ this is the case since $V$ admits a nice metric coordinate function (see Definition \ref{nice coordinate function}). It remains to consider the mean curvature of $\partial E\subset E$ at point in $\partial E$ over $S^q \setminus B^q_+$. Note that the bundle $E_{\vert B^q_-\cup Z_1}\to B^q_-\cup Z_1$ of hemispheres embeds isometrically into a bundle $F\to B^q_-\cup Z_1$ of round $p$-dimensional spheres such that $\partial E_{\vert B^q_-\cup Z_1}$ is a totally geodesic submanifold of $F$ (the construction uses the connection on $P$ and the associated $S^p$-bundle). In particular, $\partial E$ has vanishing mean curvature over $B^q_-\cup Z_1$. This also follows by direct computation from the fact that the fibers of $E_{\vert B^q_-\cup Z_1}\to B^q_-\cup Z_1$ are totally geodesic hemispheres.

Next, consider the boundary $\partial E$ over $Z_2$. Recall that the metric $g_2$ on $E_{\vert Z_2}=Z_2\times D^p=[a_2,b_2]\times S^{q-1}\times D^p$ is of the form $dt^2 +k^2(t)ds^2_{q-1}+g(t)$, $t\in [a_2,b_2]$, with $(D^p, g(t))=B^p_{\epsilon (t)}(r)$, where $\epsilon (t):[a_2,b_2]\to \R _{>0}$ is a smooth monotonously decreasing function, constant near the endpoints, with $\epsilon (t)=\pi /2$ for $t\leq \tau _1$, $\epsilon (t)=\epsilon (b_2)$ for $t\geq \tau _2$, and $\epsilon ^\prime (t)<0$ for $\tau _1 <t <\tau _2$.  Arguing as before, it follows directly that the mean curvature vanishes for points with  $t\leq \tau _1$. For points with $t\geq \tau _2$ the unit normal vector is equal to $-\partial s$ and a direct computation using formula (\ref{koszul}) shows that the mean curvature is positive.CHECK!

The metric on $E_{\vert Z_2}$ has the form $dt^2 + k^2(t) ds^2_{q-1} +f^2(t,s) ds^2_{p-1}$, $t\in [a_2,b_2]$, $s\in [0, s(t)]$, where $f(t,s):=\frac r { \sin (\epsilon (t))} \sin \left( \frac {s \sin (\epsilon (t))} r\right )$ and $s(t):=\frac {\epsilon (t)r}  {\sin (\epsilon (t))}$. The normal vector field of the boundary $(\partial E)_{\vert Z_2}$ at $(t,s(t))$ is given by $\nu =(s^\prime (t)\partial _t - \partial _s)/\sqrt {1 + s^\prime (t)^2}$.

The second fundamental form may be computed using formula (\ref{koszul}). The computation shows that if $r$ is sufficiently small, then for all $t\in (\tau _1, \tau _2)$ the second fundamental form restricted to $TS^{p-1}$ is positive and its principal curvatures dominate the other principal curvatures. Hence, for $r$ sufficiently small, the mean curvature for points in $(\partial E)_{\vert (\tau _1,\tau _2)\times S^{q-1}}$ is positive.

\medskip
We are left to show that the mean curvature of $\partial E$ at points over $Z_3$ is nonnegative. Recall that $[a_3,b_3]\times D^p$ has been identified with $$X=\{ (\tilde t, (s, x))\in \R \times S^p \, \mid \, \tilde t\in [a_3,\tilde b_3],\, s\leq F(\tilde t), \, x\in S^{p-1}\}\subset \R \times S^p ,$$
where $\varphi (t):[a_3, b_3]\to [a_3, \tilde b_3 ]$, $t\mapsto \tilde t$, is a diffeomorphism
with $\varphi (a_3)=a_3$, $F(\tilde t)=\beta N\arcsin \left (\frac {f(t)}{\beta N}\right)$, and  the curve $t \to (\tilde t, F(\tilde t))$ is parametrized by arc length.

Recall also that $g_X$ denotes the metric on $X$ induced by the metric on $\R \times S^p (\beta N)$, that the metric $\bar g_{f,h}$ on $E_{\vert Z_3}=[a_3,b_3]\times S^{q-1}\times D^p$ restricted to $[a_3,b_3]\times \{pt\} \times D^p$ is isometric to $g_X$ under this identification, and $\bar g_{f,h}$ restricted to $\{t\}\times S^{q-1}\times \{pt\}$ is equal to $h^2(t)ds^2_{q-1}$. The functions $f,h$ are defined in \ref{subsubsection functions}.

In the following we will work with $X\times S^{q-1}$. As before, let $\hat g_{f,h}$ denote the metric induced by $\bar g_{f,h}$ under the identification $E_{\vert Z_3}\cong X\times S^{q-1}$. Recall that with respect to $\hat g_{f,h}$, the tangent space of $\partial X\times S^{q-1}$ at $(\tilde t,(F(\tilde t),x),y)$, $\tilde t \in [a_3,\tilde b_3]$, $x\in S^{p-1}$, $y\in S^{q-1}$, is the orthogonal direct sum
$$\R \langle \partial _{\tilde t} + F^\prime (\tilde t) \partial _s \rangle \oplus T_xS^{p-1}\oplus T_yS^{q-1}.$$

To compute the second fundamental form, principal curvatures, and mean curvature, we use formula (\ref{koszul}) (alternatively, one can view the boundary $\partial X$ as the graph of the inverse of $F$ and apply the formulas given in \cite[Lemma 2.5]{R24I}). We will express the mean curvature in terms of $f,h$ and will use the following formulas which can be obtained by a direct computation:
\begin{equation}\label{F f formulas 1}F^\prime (\tilde t)=\frac {f^{\prime }(t)}{\sqrt { 1- \frac {f^2(t)}{(\beta N)^2}-f^{\prime 2}(t)}},\quad \sqrt {1+F^{\prime 2}(\tilde t)}=\frac 1 {\varphi ^\prime(t)}=\sqrt {\frac { 1- \frac {f^2(t)}{(\beta N)^2}} { 1- \frac {f^2(t)}{(\beta N)^2}-f^{\prime 2}(t)}} \quad \text{and}\end{equation}
\begin{equation}\label{F f formulas 2}F^{\prime \prime }(\tilde t)=\frac {\sqrt {1- \frac {f^2(t)}{(\beta N)^2}}}{(1- \frac {f^2(t)}{(\beta N)^2}-f^{\prime 2}(t))^3}\left ( f^{\prime \prime }(t)\left (1- \frac {f^2(t)}{(\beta N)^2} \right ) + \frac { f^{\prime 2} (t)f(t)}{(\beta N)^2} \right),\end{equation}
where $\tilde t =\varphi (t)$.

Note that by (\ref{koszul}) the second fundamental form $\mathrm{I\!I}(A,B)$ at a boundary point vanishes if $A$ is tangent to $\partial X$ and $B$ is tangent to $S^{q-1}$ and that $\mathrm{I\!I}$ restricted to $T\partial X$ is equal to the second fundamental form $\mathrm{I\!I}^{\partial X}$ of $\partial X\subset X$. Moreover, its restriction  to $TS^{q-1}$, which will be denoted by $\mathrm{I\!I}^{\star}$, is given by
\begin{equation}\mathrm{I\!I}^{\star}=-\frac {F^\prime (\tilde t)}{\sqrt {1 + F^\prime (\tilde t)^2}}\frac {h^\prime (t)h(t)}{\varphi ^\prime(t)} ds^2_{q-1} = -F^\prime (\tilde t)h^\prime (t)h(t)ds^2_{q-1}\tag{$\star$} \end{equation}
Hence,
$$\mathrm{I\!I}=\left ( \begin{smallmatrix} \mathrm{I\!I}^{\partial X} & 0 \\ 0 & \mathrm{I\!I}^{\star} \end{smallmatrix}\right).$$

The principal curvature in direction $\partial _{\tilde t} + F^\prime (\tilde t) \partial _s$ is given by the curvature of the curve $(\tilde t, F(\tilde t))$,
$$\mathrm{I\!I}^{\partial X}\left (\frac {\partial _{\tilde t} + F^\prime (\tilde t) \partial _s}{\sqrt {1+ F ^{\prime 2}(\tilde t)}},\frac {\partial _{\tilde t} + F^\prime (\tilde t) \partial _s}{\sqrt {1+ F ^{\prime 2}(\tilde t)}}\right )=-\frac {F^{\prime \prime}(\tilde t)}{(1+F^\prime (\tilde t)^2)^{3/2}}.$$

For $A,A_1,A_2$ tangent to $S^{p-1}$ at  $(\tilde t,(F(\tilde t),x),y)$, one obtains:
$$\mathrm{I\!I}^{\partial X}\left (\partial _{\tilde t} + F^\prime (\tilde t) \partial _s,A\right )=0\quad \text{and}$$
$$\mathrm{I\!I}^{\partial X}\left (A_1,A_2\right )=\frac{1}{\sqrt {1+F^{\prime 2}(\tilde t)}}\frac {\cot (F(\tilde t)/(\beta N))}{\beta N} ds^2_{p-1}(A_1,A_2).$$

Hence, the mean curvature is equal to
$$-\frac {F^{\prime \prime}(\tilde t)}{(1+F^{\prime 2} (\tilde t))^{3/2}}+(p-1)\frac{1}{\sqrt {1+F^{\prime 2}(\tilde t)}}\frac {\cot (F(\tilde t)/(\beta N))}{\beta N}-(q-1)F^\prime (\tilde t)h^\prime (t)h(t),\quad \tilde t =\varphi (t),$$
which is nonnegative if and only if

\begin{equation}\label{nonneg MC}-F^{\prime \prime}(\tilde t)+(p-1)(1+F^{\prime 2}(\tilde t))\frac {\cot (F(\tilde t)/(\beta N))}{\beta N}-(q-1) (1+F^{\prime 2}(\tilde t))^{3/2}F^\prime (\tilde t)h^\prime (t)h(t)\geq 0 .\end{equation}

Using formulas (\ref{F f formulas 1}) and (\ref{F f formulas 2}) one computes that (\ref{nonneg MC}) is satisfied if and only if
\begin{equation}\label{nonneg MC1}
{\cal A}_\beta (t)-{\cal B}_\beta (t)\geq 0,
\end{equation}
where
$${\cal A}_\beta (t):=
(p-1)\left(1- \frac {f^2(t)}{(\beta N)^2}-f^{\prime 2}(t)\right )^2\sqrt{1- \frac {f^2(t)}{(\beta N)^2}}\frac {\cot (F(\tilde t)/(\beta N))}{\beta N}$$
and
$${\cal B}_\beta (t):=\left ( f^{\prime \prime }(t)\left (1- \frac {f^2(t)}{(\beta N)^2} \right ) + \frac { f^{\prime 2} (t)f(t)}{(\beta N)^2} \right)+(q-1)\left (1- \frac {f^2(t)}{(\beta N)^2}-f^{\prime 2}(t) \right)\left (1- \frac {f^2(t)}{(\beta N)^2}\right )f^\prime (t)h^\prime (t)h(t).$$

It follows from the construction in \ref{subsubsection functions} that for $0<\delta < \sin ^2(R/N)$ we may choose $f$ such that $\cos (R/N)<f^\prime (t_1)< \cos (R/N) + \delta $.  Hence, $0< \sin ^2(R/N)- \delta <1-f^{\prime 2} (t)$ for all $t\in [a_3,t_1]$.

Recall that $\frac {f(t_1)}{h(t_1)} \cdot \frac \rho N <\tilde \epsilon  < \sin (R/N)$ and that $\beta >\frac {h(t_1)}\rho$ satisfies $\frac {f(t_1)}{\beta N}<\tilde \epsilon $. Note that for
$\beta  \to \infty$, the function $f(t)/(\beta N)$ on $[a_3,t_1]$ goes uniformly to $0$.

We will see that for $\beta \gg 0$ and $b $ small the mean curvature is nonnegative for $t\leq t_1$. 
Note that for any $\epsilon ^\prime > 0$
$$\vert {\cal A}_\beta (t)-(p-1)(1-f^{\prime 2}(t) )^2 F(\tilde t)\vert <\epsilon ^\prime$$
for all $t \in [a_3,t_1]$ provided $\beta $ is sufficiently large.
Using $f(t)=bf_C(t)$, $h(t)=ah_0(t)$, $h_0^\prime (t)=e^{-\frac 1 2 h_0^2(t)}$ and $f^{\prime \prime}_C(t)=Ce^{-h_0^2(t)}f_C(t)$, one computes that
{\small{
$${\cal B}_\beta (t)=\left (bCf_C(t)h_0^{\prime 2}(t)\left (1- \frac {f^2(t)}{(\beta N)^2} \right ) + \frac {f^{\prime 2} (t)f(t)}{(\beta N)^2}\right )
+(q-1)a^2\left (1- \frac {f^2(t)}{(\beta N)^2}-f^{\prime 2}(t) \right)\left (1- \frac {f^2(t)}{(\beta N)^2}\right )f^\prime (t)h_0^\prime (t)h_0(t).$$}}Note that $(f_Ch_0^\prime):\R _{\geq 0}\to \R _{>0}$ is uniformly bounded from above since $\lim _{t\to \infty} (f_C h_0^\prime )(t)= 0$. Also $h_0(t)\leq 1$ for all $t\in [a_3,t_1]$. Hence, ${\cal B}_\beta (t)>C_1b +C_2a^2$ for all $t\in [a_3,t_1]$ where $C_1,C_2$ are positive constants. Since $\lim _{b\to 0}a=0$, it follows that for any $\epsilon ^{\prime \prime} > 0$, one has ${\cal B}_\beta (t)<\epsilon ^{\prime \prime }$ for all $t\in [a_3,t_1]$ provided $b$ is sufficiently small.

Since $F(\tilde t)\geq F(a_3)>0$, we see that for $\beta \gg 0$ and $b$ sufficiently small, ${\cal A}_\beta (t)-{\cal B}_\beta (t)\geq 0$  for all $t\in [a_3,t_1]$. Hence, the mean curvature for 
boundary points with $t\in [a_3,t_1]$ is nonnegative.

Let us now consider the mean curvature for boundary points with $t_1 < t\leq b_3$. Recall from \ref{subsubsection functions} that $f^{\prime \prime }\leq 0$  for $t\in [t_1,b_3]$. It follows that for boundary points with $t_1 < t\leq b_3$ the principal curvature 
in direction $\partial _{\tilde t} + F^\prime (\tilde t) \partial _s$ is nonnegative.  Arguing as above one finds that the mean curvature for these boundary points is also nonnegative.

\medskip

The metric on $E\square V$ is obtained from pieces, each of nonnegative mean curvature. These pieces glue together to a continuous metric which is smooth away from $t=a_3$ and $t=t_1$. Smoothing at $t=a_3$ (resp. $t=t_1$), using Perelman's Theorem \ref{Perelman gluing}, yields the metric on $E\square V$.

We discuss the effect on the mean curvature of the smoothing at $t=t_1$ (the case $t=a_3$ can be treated in a similar way). As before, let us denote the functions  $f$ and $h$ of the pieces near $t=t_1$ by $f_{l }, f_r$ and $h_l,h_r$, respectively. Recall that $h_l,h_r$ define $h$ whereas the function defined by $f_l,f_r$ is only $C^1$ at $t=t_1$.

The smoothing at $t=t_1$ is obtained by an interpolation of $f_l,f_r$ near $t=t_1$. The interpolation can be chosen such that the first and second derivatives of the interpolation are bounded by the maximum and minimum of the corresponding derivatives of $f_l$ and $f_r$ on the pieces near $t=t_1$, see \cite[\S 2.2]{BWW19} for details.

Arguing as before, it follows that for $\beta \gg 0$ and $b$ sufficiently small the mean curvature of $(\partial E)_{\vert Z_3}$ with respect to the interpolations of $f_l,f_r$ near $t=t_1$ remains nonnegative.

\medskip

Altogether, we see that the metric $g_{E \square V}$ on $E \square V$ can be chosen such that the boundary of $E\square V$ has also nonnegative mean curvature.

Note that at $p_i\in Z_1$ there is, up to a scaling factor $\alpha $, an isometric embedding $\psi_i$ of $D^q_{\epsilon_i} (1) \times B^p_{\pi /2}(r)$ into $(E\square V, g_{E \square V})$, where $r>0$ can be chosen to be arbitrarily small. It follows from the construction above, that  $\psi _i$ is a nice metric coordinate function of $E\square V$ which can be used for a subsequent plumbing.

By iterating this process, one obtains a $scal>0$ metric $g_{W\square V}$ on $W\square V$ which restricts to a $Ric>0$ metric $g_{\partial (W\square V)}$ on $\partial (W\square V)$ such that the boundary has nonnegative mean curvature.

\medskip
Finally, we add a metric color $(\partial (W\square V)\times [0,\delta], g_{\partial (W\square V)} +ds^2)$  to $(W\square V,g_{W\square V})$ along $\partial (W\square V)=\partial (W\square V)\times \{0\}$ to obtain a $C^0$-metric on the union. By Remark \ref{GL mean curvature remark}, this metric can be deformed near the gluing area to yield a smooth metric of positive scalar curvature on $(\partial (W\square V)\times [0,\delta])\cup W\square V\cong W\square V$. The resulting metric has positive Ricci curvature at the boundary and is of product type near the boundary.

This finishes the proof of Theorem \ref{plumbing theorem}.\qed

\bigskip
\noindent
\textsc{Department of Mathematics, University of Fribourg, Switzerland}\\
{\em E-mail address:} \textrm{anand.dessai@unifr.ch}

\end{document}